\documentclass[11pt]{article}

\usepackage{latexsym}
\usepackage{amssymb}
\usepackage{amsmath}
\usepackage{amsfonts}
\usepackage{amsthm}
\usepackage{stmaryrd}
\usepackage[all]{xy}
\usepackage{mathrsfs}
\usepackage{hyperref}

\numberwithin{equation}{subsection}
\newtheorem{theorem}[equation]{Theorem}

\newtheorem*{thm}{Theorem}

\newtheorem{conjecture}[equation]{Conjecture}
\newtheorem{corollary}[equation]{Corollary}
\newtheorem{lemma}[equation]{Lemma}
\newtheorem{proposition}[equation]{Proposition}

\theoremstyle{definition}
\newtheorem{caution}[equation]{Caution}

\newtheorem{construction}[equation]{Construction}
\newtheorem{convention}[equation]{Convention}

\newtheorem{definition}[equation]{Definition}

\newtheorem{example}[equation]{Example}

\newtheorem{fake-assumption}[equation]{Fake-assumption}
\newtheorem{notation}[equation]{Notation}

\newtheorem{remark}[equation]{Remark}
\newtheorem{situation}[equation]{Situation}
\newtheorem{hypo}[equation]{Hypothesis}

\def\bbk{\mathbf{k}}

\def\bbv{\mathbf{v}}

\def\bbx{\mathbf{x}}
\def\bby{\mathbf{y}}

\def\bbF{\mathbf{F}}

\def\bbN{\mathbf{N}}

\def\FF{\mathbb{F}}

\def\NN{\mathbb{N}}

\def\QQ{\mathbb{Q}}
\def\RR{\mathbb{R}}

\def\ZZ{\mathbb{Z}}

\def\calA{\mathcal{A}}

\def\calE{\mathcal{E}}
\def\calF{\mathcal{F}}

\def\calI{\mathcal{I}}
\def\calJ{\mathcal{J}}

\def\calO{\mathcal{O}}

\def\calR{\mathcal{R}}
\def\calS{\mathcal{S}}

\def\gothb{\mathfrak{b}}
\def\gothc{\mathfrak{c}}
\def\gothd{\mathfrak{d}}
\def\gothe{\mathfrak{e}}

\def\gothh{\mathfrak{h}}

\def\gothm{\mathfrak{m}}

\def\gothp{\mathfrak{p}}
\def\gothq{\mathfrak{q}}
\def\gothr{\mathfrak{r}}

\def\gothu{\mathfrak{u}}
\def\gothv{\mathfrak{v}}

\def\gothE{\mathfrak{E}}

\def\gothI{\mathfrak{I}}

\def\gothN{\mathfrak{N}}

\def\gothR{\mathfrak{R}}
\def\gothS{\mathfrak{S}}

\newcommand{\bs}{\backslash}

\newcommand{\Def}{\stackrel{\mathrm{def}}=}

\newcommand{\serie}[2]{{#11},\dots,{#1{#2}}}
\newcommand{\seriezero}[2]{{#10}, \dots, {#1{#2}}}
\newcommand{\vectzero}[2]{\begin{pmatrix}
{#10} \\ \vdots \\ {#1{#2}}
\end{pmatrix}}

\newcommand{\rar}{\rightarrow}

\newcommand{\lrar}{\longrightarrow}

\newcommand{\Rar}{\Rightarrow}

\newcommand{\LRar}{\Leftrightarrow}
\newcommand{\inj}{\hookrightarrow}
\newcommand{\surj}{\twoheadrightarrow}
\newcommand{\isom}{\stackrel \sim \rar}
\newcommand{\map}[2]{\stackrel{#2}{#1}}

\newcommand{\Fil}{\mathrm{Fil}}

\newcommand{\Ker}{\mathrm{Ker}\,}

\newcommand{\Mat}{\mathrm{Mat}}
\renewcommand{\mod}{\mathrm{\;mod\;}}

\newcommand{\rank}{\mathrm{rank}\,}

\renewcommand{\log}{\mathrm{log}}

\newcommand{\alg}{\mathrm{alg}}
\newcommand{\Art}{\mathrm{Art}}

\newcommand{\Gal}{\mathrm{Gal}}

\newcommand{\sep}{\mathrm{sep}}
\newcommand{\Swan}{\mathrm{Swan}}

\newcommand{\geom}{\mathrm{geom}}

\newcommand{\Spec}{\mathrm{Spec}\,}

\renewcommand{\sp}{\mathrm{sp}}
\newcommand{\Spm}{\mathrm{Spm}}

\newcommand{\Fp}{\FF_p}

\newcommand{\OK}{{\calO_K}}

\newcommand{\Zp}{\ZZ_p}

\newcommand{\gen}{\mathrm{gen}}

\begin{document}

\title{On Ramification Filtrations and $p$-adic Differential Equations, II: mixed characteristic case}

\author{Liang Xiao \\ Department of Mathematics, Room 2-090 \\ Massachusetts Institute of Technology \\ 77 Massachusetts Avenue \\ Cambridge, MA 02139 \\
\texttt{lxiao@mit.edu}}

\date{Version of July 24, 2011}

\maketitle

\begin{abstract}
Let $K$ be a complete discrete valuation field of mixed characteristic (0, $p$) with possibly imperfect residue field. We prove a Hasse-Arf theorem for the arithmetic ramification filtrations \cite{AS-cond1} on $G_K$, except possibly in the absolutely unramified and non-logarithmic case, or $p=2$ and logarithmic case.  As an application, we obtain a Hasse-Arf theorem for filtrations on finite flat group schemes over $\calO_K$ \cite{AM-sous-groupes, Hattori-tame-char}.
\end{abstract}

\tableofcontents

\setcounter{section}{-1}
\section{Introduction}

\subsection{Main results}
\label{S:intro}

This paper is a sequel to \cite{Me-condI}, in which we proved a comparison theorem between the arithmetic ramification conductors defined by Abbes and Saito \cite{AS-cond1} and the differential ramification conductors defined by Kedlaya \cite{KSK-Swan1}.  In that paper, a key consequence is that one can use the Hasse-Arf theorem for the differential conductors to obtain a Hasse-Arf theorem for the arithmetic conductors in the equal characteristic $p>0$ case.

In this paper, we  combine the ideas from \cite{KSK-Swan1, Me-condI} with the techniques of nonarchimedean differential modules in \cite{KSK-me-pDE-TRP}, to give a proof of the following Hasse-Arf theorem for the arithmetic ramification conductors in the mixed characteristic case.

\begin{thm}
Let $K$ be a complete discrete valuation field of mixed characteristic $(0, p)$ and let $G_K$ be its absolute Galois group.  Let $\Fil^\bullet G_K$ and $\Fil^\bullet_\log G_K$ denote the ramification filtrations defined by Abbes and Saito \cite{AS-cond1}.
\begin{enumerate}
\item[1] (Hasse-Arf Theorem) Let $\rho: G_K \rar GL(V_\rho)$ be a continuous representation of finite monodromy, where $V_\rho$ is a finite dimensional vector space over a field of characteristic zero.  Then the Artin conductor $\Art(\rho)$ (defined using $\Fil^\bullet G_K$) is a nonnegative integer if $K$ is not absolutely unramified; the Swan conductor $\Swan(\rho)$ (defined using $\Fil^\bullet_\log G_K$) is a nonnegative integer if $p>2$, and $\Swan(\rho) \in \frac 12\ZZ_{\geq 0}$ if $p=2$.
\item[2] The subquotients $\Fil^a G_K / \Fil^{a+} G_K$ for $a > 1$ and $\Fil_\log^a G_K / \Fil_\log^{a+} G_K$ for $a>0$ of the ramification filtrations are trivial if $a \notin \QQ$ and are abelian groups killed by $p$ if $a \in \QQ$, except in the absolutely unramified and non-logarithmic case.
\end{enumerate}
\end{thm}

This theorem summarizes the results from Theorems~\ref{T:main-thm-nonlog}, \ref{T:log-HA-thm}, and \ref{T:log-subquot-p-ele}.

We do not know if $\Swan(\rho)$ may fail to be an integer when $p=2$ in general; the exclusion of absolute unramified and non-logarithmic case seems to be essential.

The theorem is first asked in \cite{AS-cond2}, in which Abbes and Saito proved that the subquotients of the filtrations are abelian groups, except in the absolutely unramified and non-logarithmic case.  After that, Hattori \cite{Hattori-ram-ffgs, Hattori-tame-char} gave some partial results on the first part of the theorem when the corresponding field extension can be realized by a commutative finite flat group scheme.  After the first draft of this paper was written, Saito \cite{saito3} proved the second part of the theorem in the logarithmic case independently; it follows that $\Swan(\rho) \in \ZZ[\frac 1p]$.

The technique used in this paper is very different from the approaches above except that we need a small technical lemma (see Subsection~\ref{S:etale}) which is borrowed from \cite{AS-cond2}.  This paper shares some core ideas with the first paper \cite{Me-condI} in the series, but it is logically independent of that paper.

\subsection{Idea of the proof}
\label{S:outline}

To best convey the idea, we assume that we are not in the unfortunately excluded cases listed in the main theorem.  We will come back to the reasons to exclude these cases later.  We start with a na\"ive approach to the above theorem in the non-logarithmic case.  One easily reduces to the following case.

Let $L/K$ be a finite totally ramified and wildly ramified Galois extension of complete discrete valuation fields of mixed characteristic $(0, p)$.  Let $\calO_K$, $\pi_K$, and $k$ denote the ring of integers, a uniformizer, and the residue field, respectively.  Assume that $\dim_{k^p}k < +\infty$.  There are elements $\serie{\bar b_}m \in k$ such that $\bar b_1^{i_1}\cdots \bar b_m^{i_m}$ for $i_1, \dots, i_m \in \{0, \dots, p-1\}$, form a basis of $k$ as a $k^p$-vector space;  let $\serie{b_}m$ be lifts of $\serie{\bar b_}m$ in $\calO_K$.
Our representation $\rho$ is assumed to be absolutely irreducible and it factors exactly through the Galois group $G_{L/K}$.  We need to prove that $b(L/K)\cdot \dim \rho \in \ZZ$, where $b(L/K)$ is the ramification break, i.e. the maximal number $b$ such that $\Fil^bG_{L/K} = G_L\Fil^bG_K / G_L \neq \{1\}$.

\vspace{5pt}
\textbf{\underline{Step I:} $AS=TS$ Theorem}.  (Make the Abbes-Saito space more functorial.)

Roughly speaking, the ramification break $b(L/K)$ is defined as follows.  For the extension $L/K$ and any rational number $a \in \QQ_{>0}$, Abbes and Saito \cite{AS-cond1} defined a rigid analytic space $AS^a$ together with a finite morphism $\Pi': AS^a \to A_K^{m+1}[0, |\pi_K|^a]$ (of degree $[L:K]$), where $A_K^{m+1}[0, |\pi_K|^a]$ denotes a (closed) polydisc over $K$ of radius $|\pi_K|^a$. The ramification break $b(L/K)$ is the infimum among all $a \in \QQ_{>0}$ such that the number of geometric connected components $\# \pi_0^\geom(AS^a) = [L:K]$.
A problem of this rigid analytic space is that it is not functorial under the operation of replacing $K$ by a (not necessarily finite) complete extension $K'$, which we refer to as base change later on.

Pretend for a moment that we have a continuous homomorphism $\psi: \calO_K \rar \calO_K \llbracket \delta_0, \dots, \delta_m \rrbracket$ such that $\psi(\pi_K) = \pi_K + \delta_0$, and $\psi(b_i) = b_i + \delta_i$ for $i = \serie{}m$.  We define a new rigid analytic space, called the thickening space, to be 
$$
TS_{L/K}^a = \Spm\big(L \otimes_{K, \psi} K \langle \pi_K^{-a}\delta_0, \dots, \pi_K^{-a}\delta_m \rangle \big) \map\rar\Pi A_K^{m+1}[0, |\pi_K|^a],
$$
where $\Pi$ is the projection to the second factor.

We can prove that $AS^a \simeq TS_{L/K}^a$ as rigid analytic $K$-spaces (see Theorem~\ref{T:ts=as}); this isomorphism does \emph{not} respect the morphisms $\Pi$ and $\Pi'$ to the polydisc.  The rigid analytic space $TS_{L/K}^a$ also carries the information of ramification break $b(L/K)$; together with $\Pi$, it is functorial under base change.

\vspace{5pt}
\textbf{\underline{Step II:} generic $p^\infty$-th roots.} (A procedure to reduce to the perfect residue case.)

It is natural to make the following observation. Let $a$ be a rational number slightly bigger than $b(L/K)$, then $TS_{L/K}^a$($=AS^a$) is geometrically the disjoint union of $[L:K]$ (poly)discs.  What often happens is that if you only increase the radius on certain $\delta_i$, $\pi_0^\geom(TS_{L/K})$ stays the same even when the radius goes beyond the cut-off point $|\pi_K|^{b(L/K)}$.  In contrast, if one increases the radius along some other $\delta_i$, $\pi_0^\geom(TS_{L/K})$ will change as soon as the radius reaches $|\pi_K|^{b(L/K)}$.  In the latter case, we say that the corresponding $\delta_i$ dominates.  We remark that if we change the lift of $\bar b_j$ from $b_j$ to $b_j + \pi_K$, then whether the ``uniformizer direction" $\delta_0$ is dominant may be changed as well.

The ideal situation is when  $\delta_0$ is dominant.  In this case, we can ``forget" about other directions, or more concretely, we can make the residue field perfect by simply adding in all $p$-power roots of $b_j$ for all $j$ (and then complete.)  We will talk about this procedure in more detail in the next step.  As remarked above, for this to happen, we need to find the ``correct lift" of each $b_j$.
Following the idea of Borger \cite{Borger-conductor}, we consider the notion of \emph{generic rotation}.  Let $\serie{x_}m$ be transcendental over $K$, let $K'$ be the completion of $K(\serie{x_}m)$ with respect to the $(1, \dots, 1)$-Gauss norm and let $L' = K'L$.  It easy to see that $b(L'/K') = b(L/K)$.  The upshot is that if we set the $p$-basis of $K'$ to be $\{b_1 + x_1\pi_K, \dots, b_m+ x_m\pi_K, \serie{x_}m\}$, then the uniformizer direction is going to be dominant.  So, if we set $\widetilde K$ to be the completion of the field obtained by adjoining to $K'$ all $p$-power roots of $b_i + x_i\pi_K$ and $x_i$, then we should have $b(\widetilde KL/\widetilde K) = b(L/K)$ and we are reduced to the classical situation because $\widetilde K$ has a perfect residue field.

\vspace{5pt}
\textbf{\underline{Step III:} ramification break v.s. radii of convergence for differential modules} (Where differential modules come into the picture)

Since we ``pretended" earlier that we have a homomorphism $\psi$, the morphism $\Pi: TS_{L/K}^a \to A_K^{m+1}[0, |\pi_K|^a]$ is \'etale; we can then pushforward the ring of functions on $TS_{L/K}^a$ to get a differential module $\calE$ on the polydisc (compatible as $a$ varies).  Consider the na\"ive extension of scalar to $A_L^{m+1}[0, |\pi_K|^a]$.  It is not hard to show that $\pi_0^\geom(TS_{L/K}^a) = [L:K]$ is almost equivalent to the differential module $\calE$ being trivial over $A_L^{m+1}[0, |\pi_K|^a]$ (see Proposition~\ref{P:AS-break=spec-norms}).

A good thing about radii of convergence is that it is quite computable under base change.  When replacing $K$ by $\widetilde K$, we should have a Cartesian diagram
\begin{equation}
\label{E:introduction}
\xymatrix{
TS_{L/K}^a \ar[d]^\Pi & TS_{L\widetilde K/\widetilde K}^a \ar[l] \ar[d]^{\Pi_{\widetilde K}} \\
A_K^{m+1}[0, |\pi_K|^a] & \ar[l]_f A_{\widetilde K}^{2m+1}[0, |\pi_K|^a]
}
\end{equation}
where $f$ is induced by some map $f^*: \calO_K\llbracket \delta_0, \dots, \delta_m\rrbracket \to \calO_{\widetilde K}\llbracket \eta_0, \dots, \eta_{2m}\rrbracket$ characterized by $f^* \circ \psi = \psi_{\widetilde K}|_K: \calO_K \to \calO_{\widetilde K}\llbracket \eta_0, \dots, \eta_{2m}\rrbracket$.
It is very easy to compare the radii of convergence of $\calE$ with the radii of convergence of $f^*\calE$ and the comparison of $b(L/K)$ and $b(L\widetilde K/\widetilde K)$ follows.

\vspace{5pt}
\textbf{\underline{Step IV:} Logarithmic filtration.}  (A trick to deal with logarithmic filtration.)

We briefly discuss the idea behind the proof in the logarithmic case. We do not expect that we can always make the uniformizer direction ``log-dominant".  Instead, we expect a dichotomy:
\begin{itemize}
\item if the uniformizer direction is log-dominant, we are good anyway;
\item if the uniformizer direction is not log-dominant, we expect that, after a large tame base change to $K_n = K(\pi_K^{1/n})$ and then a generic rotation for $K_n$ as in Step II, $b(L'_n / K'_n) = n b_\log(L/K)$ and the uniformizer direction is \emph{non-log}-dominant.  Here the multiple $n$ comes from the normalization; the key is that after the followed-up generic rotation, the nonlog ramification break is one less than the log ramification break.
\end{itemize}
Thus, we can always deduce that $n \cdot \Swan(\rho) \in \ZZ$ for $n \gg 0$ and $p \nmid n$.  Taking two coprime numbers $n_1$ and $n_2$ will imply that $\Swan(\rho)$ itself is an integer.

\vspace{10pt}
We now come back to real life and discuss where the na\"ive approach fails and how we fix it.

(1)  The first thing to notice is that the desired homomorphism $\psi$ \emph{never} exists, as we cannot make $\psi(p) = p$ and $\psi(\pi_K) = \pi_K + \delta_0$ happen at the same time.  As a salvage, we take $\psi$ to be a function, which becomes a homomorphism if we modulo the ideal $I_K = p (\delta_0 /\pi_K, \serie{\delta_}m)$ (Proposition~\ref{P:psi-almost-hom}).  When $K$ is absolutely unramified or, in other words, $v_K(p) = 1$, this condition is significantly weakened.  This is the only hindrance to extend our main result to the absolutely unramified and non-logarithmic case (see also Remark~\ref{R:beta_K=1-evil}).

We define the space $TS_{L/K, \psi}^a$ by writing down the equations generating the extension $\calO_L/ \calO_K$ and applying $\psi$ termwise.  When considering the effect of adding a generic $p$-th root (instead of $p^\infty$-th root, see Remark~\ref{R:gen-p-infty-not-work}), we similarly require that $f \circ \psi$ and $\psi_{\widetilde K}$ only agree modulo $I_{\widetilde K} = p(\eta_0/\pi_{\widetilde K}, \eta_1, \dots, \eta_{2m})$.  We have to carefully keep track of the error terms due to the non-homomorphism $\psi$ and non-commutativity of $f \circ \psi$ and $\psi_{\widetilde K}$.  In particular, if we still want \eqref{E:introduction} to be a Cartesian diagram, we need to make modification on $TS_{L\widetilde K / \widetilde K }^a$ (see Theorem~\ref{T:AS-inv-gen-rot}); this is the most difficult theorem of the paper.  Luckily, the modification made here is not too serious so that we still have AS=TS Theorem \ref{T:ts=as} for this modified thickening space.

(2) Since we have the problem with defining $\psi$, the morphism $\Pi:TS_{L/K, \psi}^a \to A_K^{m+1}[0, |\pi_K|^a]$ is only finite and \'etale if $a \geq b(L/K) -\epsilon$ for some $\epsilon > 0$.  This is the only technical place we need to refer back to Abbes and Saito's approach: \cite[Theorem~7.2]{AS-cond1} (and \cite[Corollary~4.12]{AS-cond2} in the logarithmic case).  This \'etaleness statement validates the construction of differential modules. The auxiliary \'etale locus given by $\epsilon$ enables us to find the exact loci where the intrinsic radii are maximal (or equivalently the loci where the differential module is trivial), and hence to identify the ramification break.

(3) Since $\psi$ fails to be a homomorphism, we have a minor technical issue when using differential modules.  We have to study the generic radii of convergence over polydiscs instead of one dimensional discs (as did in \cite{Me-condI}); this makes essential use of the recent results on $p$-adic differential modules from \cite{KSK-me-pDE-TRP}.  As a result, the proof of the logarithmic case is slightly more complicated and for $p=2$, we can only prove that Swan conductors lie in $\frac 12 \ZZ$ instead of in $\ZZ$.

\subsection{Who cares about the imperfect residue field case, anyway?}
\label{S:who-cares-nonperfect}

In algebraic geometry, if one wants to measure the ramification of an $l$-adic sheaf along a divisor, it is natural to pass to the completion at the generic point of the divisor; this would naturally give rise to a complete discrete valuation field with imperfect residue field, if the dimension of the divisor is not zero.

It is natural to ask how the ramification information varies from one divisor to another.
Kedlaya \cite{KSK-Swan2} started an interesting study along this line, inspired by the semicontinuity results of Andr\'e \cite{Andre-semiconti-irr} in complex algebraic geometry.  In \cite{KSK-Swan2}, Kedlaya took an $F$-isocrystal on a smooth surface $X$ overconvergent along the complement divisor $D$ of simple normal crossings, in a compactification of $X$.  If we blow up the intersection of two irreducible components of $D$, we may realize $\calF$ over this new space and  measure the Swan conductor along the exceptional divisor.  This process can be iterated.  Kedlaya proved in \cite{KSK-Swan2} that, after suitable normalization, the Swan conductors along these exceptional divisors are interpolated by a continuous piecewise linear convex function.  This result also holds for general smooth varieties of arbitrary dimension (see \cite{KSK-Swan2}), and also for lisse $l$-adic sheaves.

An interesting question is: does the same phenomenon happen for a noetherian complete regular local ring $\calO_K \llbracket \serie{t_}n \rrbracket$, where $\OK$ is a complete discrete valuation ring of mixed characteristic?

\vspace{5pt}
Another application is to the study of finite flat group schemes via ramification filtration initiated by Abbes and Mokrane in \cite{AM-sous-groupes}.  Hattori conjectured that one can give a bound on the denominators of ramification breaks.  This can be proved by an analogous Hasse-Arf theorem for finite flat group schemes.  Thus, as a consequence of the main theorem of this paper, we obtain a Hasse-Arf theorem for finite flat group schemes in the mixed characteristic case by an argument originally due to Hattori.

\subsection{Structure of the paper}
\label{S:structure}
In Section 1, we first recall some results of $p$-adic differential modules from \cite{KSK-me-pDE-TRP}.  Then we review the definition of ramification filtrations in Subsection~\ref{S:review-AS-fil}.  

In Section 2, we set up the framework for the proof of the main result.  In Subsection~\ref{S:standard-AS}, we introduce the standard Abbes-Saito spaces.  In Subsections~\ref{S:psi-map}-\ref{S:diff-eqn}, we define the function $\psi$ we mentioned earlier and construct the thickening spaces and the associated differential modules; the aim is to translate the question about the ramification breaks into a question about the intrinsic radii of convergence.  In Subsection~\ref{S:recursive-TS}, we discuss a variant of thickening spaces.

The proofs of the main Theorems~\ref{T:main-thm-nonlog}, \ref{T:log-HA-thm}, and \ref{T:log-subquot-p-ele} occupy the whole Section 3.  In the first three subsections, we deduce the Hasse-Arf theorem for non-logarithmic ramification filtration.  In Subsection~\ref{S:applications}, we apply the Hasse-Arf theorem for Artin conductors to obtain a Hasse-Arf theorem for finite flat group schemes.  In Subsection~\ref{S:tame}, we deduce the integrality of Swan conductors from that of Artin conductors by tame base change.  In the last two subsections, we use a trick of Kedlaya to prove that the subquotients of the logarithmic filtration (on the wild ramification group) are abelian groups killed by $p$.

\subsection{Acknowledgments}
\label{S:acknowledgments}
Many thanks are due to my advisor, Kiran Kedlaya, for introducing me to the problem, for generating some crucial ideas, for answering my stupid questions, and for spending many hours reviewing early drafts.

Thanks to Jennifer Balakrishnan for helping me review early drafts, correct the grammar, and smooth the argument.
Thanks also to Ahmed Abbes, Shin Hattori, Ruochuan Liu, Shun Ohkubo, Andrea Pulita, Takeshi Saito, Yichao Tian, and Xin Zhou for helpful discussions.  The author would also like to express his special thank to Ahmed Abbes and Takeshi Saito for organizing the wonderful conference \textit{On Vanishing Cycles} in Japan in 2007.  It provided the author a great opportunity to learn about this field.

Thanks to referees for great help on improving the presentation of this paper.

Financial support was provided by MIT Department of Mathematics.  Also, when working on this paper, the author had Research Assistantship funded by Kedlaya's NSF CAREER grant DMS-0545904.

\subsection{Notation}
Due to the technical details involved, the notation in this paper is in particularly complicated.  We list a few important ones together with short explanations and first appearance.  We hope that this could make the paper more accessible.

$K$ CDVF of mixed char of absolute ramification degree $\beta_K$; $L$ finite extension; $\theta = |\pi_K|$.

$\widetilde K$ (\ref{N:base-change}) adding a generic $p$-th root to $K$. 

$\widetilde K_n$ (\ref{N:tilde-K_n}) adding $\pi_K^{1/n}$ and generic $p$-th roots to $K$.

$K_*$ (\ref{L:K*-over-K}) an ``Artin-Scheier" extension of $K$. 

$\widetilde K_\gamma$ (\ref{N:tilde-K-gamma}) adding generic $p$-th roots to $K_*$.

$J = \{1, \dots, m\}$ and $J^+ =J \cup \{0\}$; they are used to index $p$-basis.

$b_1, \dots, b_m$ or $b_J$ (\ref{C:generators-of-calI}) lifts of a $p$-basis of $k$.

$c_1, \dots, c_m$ or $c_J$
(\ref{C:generators-of-calI}) lifts of a $p$-basis of $l$.

$u_0, \dots, u_m$ (\ref{C:generators-of-calI}) proxies for $c_{J^+}$. 

$p_0, \dots, p_m$ or $p_{J^+}$ (\ref{C:generators-of-calI}) relations of the extension $\calO_L$ over $\calO_K$ with generators $c_J$ and $\pi_L$. 

$N^a$ (\ref{N:norm-OK-uJ+}) set of elements of $\calO_K[u_{J^+}]$ with norm $\leq \theta^a$.

$AS_{L/K(,\log)}^a$ and $\calO_{AS, L/K(,\log)}^a$ (\ref{D:standard-AS}) (standard) Abbes-Saito spaces and their rings of functions.

$\calR_K = \calO_K \llbracket \delta_0 / \pi_K, \delta_{J} \rrbracket$ (\ref{N:R_K}), similar for $\calR_{\widetilde K}$ (\ref{N:R-tilde-K}).

$\psi_K: \calO_K \to \calO_K \llbracket \delta_{J^+}\rrbracket \subseteq \calR_K$ (\ref{C:psi-map}), similar for $\psi_{\widetilde K}$ (\ref{N:R-tilde-K}) and other fields.

$\calS_K = \calR_K \langle u_{J^+} \rangle$ (\ref{D:admissible}).

$R_{J^+}$ (\ref{D:admissible}) elements of $(\delta_{J^+})\calS_K$ representing the error terms with error gauge $\leq \omega$.

$TS_{L/K(,\log), R_{J^+}}^a$ and $\calO_{TS, L/K(,\log), R_{J^+}}^a$ (\ref{D:th-space}) thickening spaces and their rings of functions, similar for the standard ones $TS_{L/K(,\log), \psi}^a$ and $\calO_{TS, L/K(,\log),\psi}^a$ (\ref{D:th-space}).

$\Delta: \calS_K / (\psi(p_{J^+}) + R_{J^+}) \to \calO_K\langle u_{J^+}\rangle / (p_{J^+}) \stackrel \sim\to \calO_L$ (\ref{C:generators-of-calI} and \ref{N:Delta-for-ts}); $\overline \Delta$ its reduction.

$ET_{L/K, R_{J^+}}$ or $ET_{L/K}$ (\ref{D:etale-locus}) \'etale locus over which the thickening space is \'etale.

$\gothc_{0,I},\gothc_\Lambda, \gothu_{0,I}, \gothu_\Lambda, \gothp_{0,I}, \gothp_\Lambda,  \gothS_K, \gothR_{0, I}, \gothR_\Lambda, \gothN^a, \dots$ (Subsection~\ref{S:recursive-TS}) recursive version of all above.

$\Delta: \gothS_K / (\psi(\gothp_{0,I}) + \gothR_{0,I}, \psi(\gothp_\Lambda) + \gothR_\Lambda) \to \calO_K\langle \gothu_{0,I}, \gothu_\Lambda \rangle / (\gothp_{0,I}, \gothp_\Lambda) \stackrel \sim\to \calO_L$ (\ref{C:small-D-sp} and \ref{D:error-gauge-rec}).

$ \tilde \gothc_{0,I}, \tilde \gothc_\Lambda, \tilde \gothu_{0,I},\tilde\gothu_\Lambda,\tilde \gothv,\tilde \gothp_{0,I}, \tilde\gothp_\Lambda,\tilde \gothq, \gothS_{\widetilde K}, \widetilde \gothR_{0,I},\widetilde \gothR_\Lambda,  \widetilde \gothR_{\tilde \gothq}, \dots$ (proof of \eqref{T:base-change}) recursive version for $\widetilde K$.

\section{Background Reviews}
\setcounter{equation}{0}

\subsection{Differential modules}\label{S:pDE}

We first recall some recent results in the theory of $p$-adic differential modules.  This subject was first studied by Christol, Dwork, Mebkhout, and Robba \cite{ChrDwork-dif-mod-on-annulus, CM-index-thm-III, ChrRobba-pDE}.  Recently, Kedlaya and the author improved some of the techniques in \cite{KSK-notes, KSK-me-pDE-TRP}.  We record some useful results from these sources.

\begin{convention}
Throughout this paper, $p>0$ will be a prime number.  By a \emph{$p$-adic field}, we mean a field $K$ of characteristic zero, complete with respect to a nonarchimedean norm for which $|p| = 1/p$.  In particular, the residue field of $K$ has characteristic $p$.
\end{convention}

\begin{convention}
For an index set $J$, we write $e_J$ or $(e_J)$ for a tuple $(e_j)_{j \in J}$.  For another tuple $b_J$, denote $b_J^{e_J} = \prod_{j \in J} b_j^{e_j}$ if only finitely many $e_j \neq 0$. We also use $\sum_{e_J = 0}^{n}$ to mean the sum over $e_j \in \{0, 1, \dots, n\}$ for each $j \in J$, only allowing finitely many of them to be nonzero.  For simple notation, we may suppress the range of the summation when it is clear.  For a set $A$, we write $e_J \subset A$ or $(e_J) \subset A$ to mean that 
$e_j \in A$ for any $j \in J$.
\end{convention}

\begin{notation}
From now on, let $K$ be a $p$-adic field and fix an element $\pi_K \in K^\times$ of norm $\theta<1$.  When $K$ has  discrete valuation, we take $\pi_K$ to be a uniformizer.
\end{notation}

\begin{notation}\label{N:affinoids}
For an interval $I \subset [0, +\infty]$, we denote the $n$-dimensional polyannulus with radii in $I$ by $A_K^n(I)$.  
(We do not impose any rationality condition on the endpoints of $I$, so this
space should be viewed as an analytic space in the sense of 
Berkovich \cite{Berkovich-book}.)
If $I$ is written explicitly in terms of its
endpoints (e.g., $[\alpha, \beta]$),
we suppress the parentheses around $I$ (e.g., 
$A_K^n[\alpha, \beta]$).
\end{notation}

\begin{notation}
For $R$ a complete topological ring, we use $R \langle \serie{u_}m \rangle$ to denote the completion of the polynomial ring $R[\serie{u_}m]$ with respect to the topology induced from $R$.  When $R$ is a complete $\calO_K$-algebra, we write $R \langle \pi_K^{-a_1}\delta_1, \dots, \pi_K^{-a_m}\delta_m \rangle$ to denote the formal substitution of $R \langle \serie{u_}m \rangle$ via $u_j = \pi_K^{-a_j}\delta_j$ for $j = \serie{}m$, where $a_1, \dots, a_m \in \RR$.  In particular, $K \langle\pi_K^{-a_1}\delta_1, \dots, \pi_K^{-a_m}\delta_m \rangle$ is the ring of analytic functions on $A_K^1[0, \theta^{a_1}] \times \cdots \times A_K^1[0, \theta^{a_m}]$.

We use $K \llbracket T \rrbracket_0$ to denote the bounded power series ring consisting of formal power series $\sum_{i \in \ZZ_{\geq 0}}a_iT^i$ for which $a_i \in K$ and $|a_i|$ are bounded.
\end{notation}

\begin{notation} \label{N:J+1}
In this subsection, let $J = \{1, \dots, m\}$ and $J^+ = J \cup \{0\}$.
\end{notation}

\begin{definition}
For $s_{J^+} \subset \RR$, the \emph{$\theta^{s_{J^+}}$-Gauss norm} on $K[\delta_{J^+}]$ is the norm given by
$$
\Big| \sum_{e_{J^+}} a_{e_{J^+}} \delta_{J^+}^{e_{J^+}} \Big|_{s_{J^+}} = \max \big\{ |a_{e_{J^+}}| \cdot \theta^{e_0s_0 + \cdots + e_ms_m}\big\}.
$$
It extends uniquely to $K(\delta_{J^+})$; we denote the completion by $F_{s_{J^+}}$.  This Gauss norm also extends continuously to $K \langle\pi_K^{-a_0}\delta_0, \dots, \pi_K^{-a_m}\delta_m \rangle$ if $s_j \in [a_j, + \infty)$ for all $j \in J^+$.  Hence,  $K \langle \pi_K^{-a_0}\delta_0, \dots, \pi_K^{-a_m}\delta_m \rangle$ embeds into $F_{s_{J^+}}$.
\end{definition}

\begin{convention}
Throughout this paper, all (relative) differentials and derivations are continuous and all connections are integrable.  For simple notation, we may suppress the continuity and integrability.
\end{convention}

\begin{definition}\label{D:spectral-norms}
Let $F$ be a differential field of order 1 and characteristic zero, i.e., a field of characteristic zero equipped with a derivation $\partial$.  Assume that $F$ is complete for a nonarchimedean norm $|\cdot|$.   Let $V$ be a differential module with the differential operator $\partial$. The \textit{spectral norm of $\partial$ on $V$} is defined to be
$$
|\partial|_{\sp, V} = \lim_{n \rar +\infty} |\partial^n|_V^{1/n}.
$$
One can show that $|\partial|_{\sp, V} \geq |\partial|_{\sp, F}$ \cite[Lemma~6.2.4]{KSK-notes}.

Define the \emph{intrinsic $\partial$-radius} of $V$ to be
$$
IR_\partial(V) = |\partial|_{\sp, F} / |\partial|_{\sp, V} \in (0,1].
$$
\end{definition}

\begin{example}
For $s_{J^+} \subset \RR$, the spectral norms of $\partial_{J^+}$ on $F_{s_{J^+}}$ are as follows.
\[
|\partial_j|_{F_{s_{J^+}}, \sp} = 
p^{-1/(p-1)}\theta^{-s_j}, \quad j \in J^+.
\]
\end{example}

\begin{remark}
If $F'/F$ is a complete extension and $\partial$ extends to $F'$, then for any differential module $V$ on $F$, $V \otimes F'$ is a differential module on $F'$.  Moreover, if $|\partial|_{\sp, F} = |\partial|_{\sp, F'}$, we have $IR_\partial(V) = IR_\partial(V \otimes F')$.
\end{remark}

\begin{notation}
Let $a_{J^+} \subset \RR$ be a tuple and let $X = A_K^1[0, \theta^{a_0}] \times \cdots \times A_K^1[0, \theta^{a_m}]$ be the closed polydisc with radii $\theta^{a_{J^+}}$ and with $\delta_{J^+}$ as coordinates.
\end{notation}

\begin{notation}\label{N:diff-mod}
A \emph{differential module} over $X$ (relative to $K$) is a finite locally free coherent sheaf $\calE$ on $X$ together with an integrable connection
$$
\nabla: \calE \rar \calE \otimes_{\calO_X} \Big( \bigoplus_{j \in J^+} \calO_X \cdot d\delta_j \Big).
$$
Let $\partial_{J^+} = \partial / \partial \delta_{J^+}$ be the dual basis of $d\delta_{J^+}$.  They act commutatively on $\calE$.  A section $\bbv$ of $\calE$ over $X$ is called \emph{horizontal} if $\partial_j(\bbv) = 0$ for $\forall j \in J^+$.  Let $H^0_\nabla(X, \calE)$ denote the set of horizontal sections on $\calE$ over $X$.  A differential module is called \emph{trivial} if there exists a set of horizontal sections which forms a basis of $\calE$ as a free coherent sheaf.

Let $s_j \in [a_j, +\infty)$ for $j \in J^+$.  For $j \in J^+$, let $IR_j(\calE; s_{J^+})$ denote the intrinsic $\partial_j$-radius $IR_{\partial_j}(\calE \otimes_{\calO_X} F_{s_{J^+}})$.  Let $IR(\calE; s_{J^+}) = \min_{j \in J^+} \big\{ IR_j(\calE; s_{J^+}) \big\}$ be the \emph{intrinsic radius} of $\calE$.  If $s_{j'} = s$ for all $j' \in J$, we simply write $IR_j(\calE; s_0, \underline s)$ and $IR(\calE; s_0, \underline s)$ for intrinsic $\partial_j$-radius and intrinsic radius, respectively.  Moreover, if $s_0 = s$, we may further simplify the notation as $IR_j(\calE; \underline s)$ and $IR(\calE; \underline s)$.
\end{notation}

\begin{lemma} \label{L:generic-point}
Fix $j \in J^+$.  There exists a unique continuous $K$-homomorphism $f_{\gen, j}^*: F_{a_{J^+}} \rar F_{a_{J^+}} \llbracket \pi_K^{-a_j} T_j \rrbracket_0$, such that $f^*_{\gen,j}(\delta_{J^+ \bs\{j\}}) = \delta_{J^+ \bs\{j\}}$ and $f^*_{\gen,j}(\delta_j) = \delta_j + T_j$.
\end{lemma}
\begin{proof}
See \cite[Lemma~1.2.12]{KSK-me-pDE-TRP}.
\end{proof}

\begin{lemma}\label{L:sp-norms-vs-gen-rad}
Set $F = F_{a_{J^+}}$ for short.  The pullback $f_{\gen,j}^* (\calE \otimes_{\calO_X} F)$ becomes a differential module over $A_F^1[0,  \theta^{a_j})$ relative to $F$.  Then for any $r \in [0, 1]$, $IR_j(\calE; a_{J^+}) \geq r$ if and only if $f_{\gen,j}^*(\calE \otimes_{\calO_X} F)$ is trivial over $A_F^1[0, r\theta^{a_j})$.
\end{lemma}
\begin{proof}
This is essentially because the Taylor series $\sum_{n=0}^\infty \partial^n_{T_j}(\bbv)\cdot T_j^n/(n!) = \sum_{n=0}^\infty \partial^n_j(\bbv)\cdot T_j^n/(n!)$ converges when $|T_j| < r\theta^{a_j}$ for any section $\bbv$ if and only if $IR_j(\calE; a_{J^+}) \geq r$.  For more details, see \cite[Proposition~1.2.14]{KSK-me-pDE-TRP}.
\end{proof}

We reproduce some basic properties of intrinsic radii, starting with the following off-centered tame base change, which is a fun exercise in \cite[Chap.~9,~Exercise~8]{KSK-notes}.  To ease the readers who are not familiar with differential modules, we give a complete proof.
\begin{construction}\label{C:off-center-tame}
Fix $n \in \NN$ prime to $p$.  Assume for a moment that $m=0$ (and $a = a_0$), i.e., we consider the one dimensional case $X = A_K^1[0, \theta^a]$.  Fix $x_0 \in K$ such that $|x_0| = \theta^b > \theta^a$ ($b<a$). In particular, the point $\delta_0 = -x_0$ is not in the disc $X$.  Denote $K_n = K(x_0^{1/n})$, where we fix an $n$-th root $x_0^{1/n}$ of $x_0$.

Consider the $K$-homomorphism $f_n^*: K \langle \pi_K^{-a}\delta_0 \rangle \rar K_n \langle \pi_K^{-a + b(n-1)/n } \eta_0 \rangle$, sending $\delta_0$ to \[
(x_0^{1/n} + \eta_0)^n - x_0 = x_0^{(n-1)/n} \eta_0 \Big( \sum_{i = 0}^{n-1} \binom n{i+1} \big(\frac{\eta_0}{x_0^{1/n}} \big)^i \Big),
\]
where the term in the parentheses on the right has norm 1 and is invertible because $|x_0^{1/n}| > |\eta_0|$. Hence $f_n^*$ extends continuously to a homomorphism $F_a \rar F'_{a-b(n-1)/n}$, where $F'_{a-b(n-1)/n}$ is the completion of $K_n(\eta_0)$ with respect to the $\theta^{a-b(n-1)/n}$-Gauss norm.

Also, $f_n^*$ gives a morphism of rigid $K$-spaces $f_n: Z = A_{K_n}^1[0, \theta^{a-b(n-1)/n}] \rar X = A_K^1[0, \theta^a]$.  It is finite and \'etale because the branching locus is at $\delta_0 = -x_0$, outside the disc $X$.  Thus, for a differential module $\calE$ on $X$, its pullback $f_n^* \calE$ is a differential module over $Z$ via
$$
f_n^*\calE \stackrel {f_n^*\nabla} \longrightarrow  f_n^*\Big(\calE \otimes_{\calO_X} \calO_X d\delta_0 \Big) \longrightarrow f_n^* \calE \otimes_{\calO_Z} \calO_{Z} d\eta_0,
$$
where the last homomorphism is given by $d\delta_0 \mapsto n(x_0^{1/n} + \eta_0)^{n-1} d\eta_0$.
\end{construction}

\begin{proposition}\label{P:off-center-tame}
Keep the notation as above.  We have
$$
IR_{\partial_{\eta_0}}(f_n^*\calE; a-b(n-1)/n) = IR_{\partial_0}(\calE; a).
$$
\end{proposition}
\begin{proof}
The proof is essentially the same as \cite[Lemma~5.11]{KSK-overview} or \cite[Proposition~9.7.6]{KSK-notes}.  Lemma~\ref{L:generic-point} gives the following commutative diagram
$$
\xymatrix{
F_a \ar[d]^{f_n^*} \ar[r]^-{f_{\gen,0}^*} & F_a \llbracket \pi_K^{-a} T_0 \rrbracket_0 \ar[d]^{\tilde f_n^*} \\
F'_{a-b(n-1)/n} \ar[r]^-{f'^*_{\gen,0}} & F'_{a-b(n-1)/n} \llbracket \pi_K^{-a+b(n-1)/n} T'_0 \rrbracket_0 
}
$$
where $\tilde f_n^*$ extends $f^*_n$ by sending $T_0$ to $(x_0^{1/n} + \eta_0 + T'_0)^n - (x_0^{1/n} + \eta_0)^n$.

We claim that for $r \in [0, 1]$, $\tilde f_n$ induces an isomorphism between
$$
F'_{a-b(n-1)/n} \times_{f_n^*, F_a} \big( A_{F_a}^1[0, r \theta^a) \big) \cong  A_{F'_{a-b(n-1)/n}}^1[0, r \theta^{a-b(n-1)/n}).
$$
Indeed, if $|T'_0| < r \theta^{a-b(n-1)/n} <\theta^{b/n}$, then
$$
|T_0| = |(x_0^{1/n} + \eta_0 + T'_0)^n - (x_0^{1/n} + \eta_0)^n| = |nT'_0 (x_0^{1/n} + \eta_0)^{n-1}| < r \theta^{a-b(n-1)/n} \cdot  (\theta^{b/n})^{n-1} = r \theta^a.
$$
Conversely, if $|T_0| < r\theta^a$, we define the inverse map by the binomial series
$$
T'_0 = (x_0^{1/n} + \eta_0) \cdot \Big[ -1 + \Big(1+ \frac {T_0}{(x_0^{1/n} + \eta_0)^n} \Big)^{1/n}\Big] = \sum_{i = 1}^\infty \binom{1/n}i \frac{T_0^i}{(x_0^{1/n} + \eta_0)^{ni-1}}.
$$
The series converges to an element with norm $< r \theta^{a-b(n-1)/n}$.

Therefore, Lemma~\ref{L:sp-norms-vs-gen-rad} implies that for $r \in [0, 1]$, 
\begin{align*}
&\ IR_{\partial_0}(\calE; a) \geq r \\
\LRar&\ f_{\gen,0}^*(\calE \otimes_{\calO_X} F_a) \textrm{ is trivial over } A_{F_a}^1[0, r \theta^a)\\
\LRar&\ \tilde f_n^* f_{\gen,0}^*(\calE \otimes_{\calO_X} F_a) = f'^*_{\gen,0} \big(f_n^*\calE \otimes_{\calO_Z} F'_{a-b(n-1)/n} \big) \textrm{ is trivial over } A_{F'_{a-b(n-1)/n}}^1[0, r \theta^{a-b(n-1)/n})\\
\LRar&\ IR_{\partial_{\eta_0}}(f_n^*\calE; a-b(n-1)/n) \geq r.
\end{align*}
The proposition follows.
\end{proof}

Similarly, we can study a type of off-centered Frobenius.

\begin{construction}
Let $b>0$ and $0< a < \min\{-\log_\theta p + b, pb\}$ and let $\beta \in K$ be an element of norm $1$.  Let $L$ be the completion of $K(x)$ with respect to the $\theta^a$-Gauss norm.

Let $f: Z = A_L^1[0, \theta^b] \rar A_K^1[0, \theta^a]$ be the morphism given by $f^*: \delta_0 \mapsto (\beta + \eta_0)^p - \beta^p +x$.  By our choices of $a$ and $b$, the leading term of $f^*(\delta_0)$ is $x$, which is transcendental over $K$.  Hence $f^*$ extends continuously to a homomorphism $F_a \rar F'_b$, where $F'_b$ is the completion of $L(\eta_0)$ with respect to the $\theta^b$-Gauss norm.  Moreover, $f^*\Omega^1_X \cong \Omega^1_Z$ because the branching locus is at $\eta_0 = -\beta$, outside the disc.  Thus $f^*\calE$ becomes a differential module over $Z = A^1_L[0, \theta^b]$ via
$$
f^*\calE \stackrel {f^*\nabla} \longrightarrow  f^*\Big(\calE \otimes_{\calO_X} \calO_X d\delta_0 \Big) \longrightarrow f^* \calE \otimes_{\calO_Z} \calO_{Z} d\eta_0,
$$
where the second homomorphism is given by $d\delta_0 \mapsto p(\beta+\eta_0)^{p-1} d \eta_0$.
\end{construction}

\begin{proposition}\label{P:sp-norm-under-Frob}
Keep the notation as above.  We have
$$
IR_{\partial_0}(f^*\calE; b) \geq IR_{\partial_{\eta_0}}(\calE; a).
$$
\end{proposition}
\begin{proof}
As in Proposition~\ref{P:off-center-tame}, we start with the following commutative diagram from Lemma~\ref{L:generic-point}.
$$
\xymatrix{
F_a \ar[d]^{f^*} \ar[r]^-{f_{\gen,0}^*} & F_a \llbracket \pi_K^{-a} T_0 \rrbracket_0 \ar[d]^{\tilde f^*} \\
F'_b \ar[r]^-{f'^*_{\gen,0}} & F'_b \llbracket \pi_K^{-b} T'_0 \rrbracket_0 
}
$$
where $\tilde f^*$ extends $f^*$ by sending $T_0$ to $(\beta + \eta_0 + T'_0)^p - (\beta +\eta_0)^p$.

For $r \in [0, 1]$, by Lemma~\ref{L:radius-p-th-power} below, $|T'_0| < r \theta^a$ implies $|T_0| < \max\{r^p \theta^{pa}, p^{-1}r \theta^a\} < r\theta^b$.

Therefore, Lemma~\ref{L:sp-norms-vs-gen-rad} implies that
\begin{eqnarray*}
&& IR_{\partial_0}(\calE; a) \geq r \\
&\LRar& f_{\gen,0}^*(\calE \otimes_{\calO_X} F_a) \textrm{ is trivial over } A_{F_a}^1[0, r \theta^a)\\
&\Rar& \tilde f^* f_{\gen,0}^*(\calE \otimes_{\calO_X} F_a) = f'^*_{\gen,0}(f^*\calE \otimes_{\calO_Z} F'_b) \textrm{ is trivial over } A_{F'_b}^1[0, r \theta^b)\\
&\LRar& IR_{\partial_{\eta_0}}(f^*\calE; b) \geq r.
\end{eqnarray*}
The proposition follows.
\end{proof}

\begin{lemma}\label{L:radius-p-th-power}
\emph{\cite[Lemma~10.2.2(a)]{KSK-notes}}
Let $K$ be a non-archimedean field and let $b, T \in K$.  For $r \in (0, 1)$, if $|b - T| < r |b|$, then
$$
|b^p - T^p| \leq \max \{ r^p|b|^p, p^{-1}r|b|^p \}.
$$
\end{lemma}

\begin{remark}
A stronger form of Proposition~\ref{P:sp-norm-under-Frob} above for (straight) Frobenius can be found in \cite[Lemma~10.3.2]{KSK-notes} or \cite[Lemma~1.4.11]{KSK-me-pDE-TRP}.
\end{remark}

Now, we study the variation of intrinsic radii on polydiscs.

\begin{definition}
An \emph{affine functional} on $\RR^{m+1}$ is a function $\lambda: \RR^{m+1} \to \RR$ of the
form $\lambda(x_0,\dots,x_m) = a_0 x_0 + \cdots + a_m x_m + b$ for some
$a_0,\dots, a_m, b \in \RR$. If $a_0,\dots,a_m \in \ZZ$, we say $\lambda$
is \emph{transintegral} (short for ``integral after translation'').

A subset $C \subseteq \RR^{m+1}$ is \emph{polyhedral} if there exist finitely many
affine functionals $\lambda_1, \dots, \lambda_r$ such that
\[
C = \{x \in \RR^{m+1}: \lambda_i(x) \geq 0 \qquad (i=1,\dots,r)\}.
\]
If the $\lambda_i$ can be all taken to be transintegral, we say that $C$ is 
\emph{transrational polyhedral}.
\end{definition}

\begin{proposition}\label{P:prop-diff-eqns}
Let $a_{J^+} \subset \RR$ be a tuple and let $X = A_K^1[0, \theta^{a_0}] \times \cdots \times A_K^1[0, \theta^{a_m}]$ be the polydisc with radii $\theta^{a_{J^+}}$ and coordinates $\delta_{J^+}$.  Let $\calE$ be a differential module over $X$.  Then

(a) (Continuity) The function $\log_\theta IR(\calE; s_{J^+})$ is continuous for $s_j \in [a_j, +\infty)$ and $j \in J^+$.

(b) (Monotonicity) Let $s_j \geq s'_j \geq a_j$ for all $j \in J^+$.  Then $IR(\calE; s_{J^+}) \geq IR(\calE; s'_{J^+})$.

(c) (Zero Loci) The subset $Z(\calE) = \{s_{J^+} \in [a_0, +\infty) \times \cdots \times [a_m, +\infty)| IR(\calE; s_{J^+}) = 1\}$ is transrational polyhedral.  Moreover, it contains $[a'_0, +\infty) \times \cdots \times [a'_m, +\infty)$ for $a'_0, \dots, a'_m$ sufficiently large.
\end{proposition}
\begin{proof}
Statements (a) and (c) follow from \cite[Theorem~3.3.9]{KSK-me-pDE-TRP}; $Z(\calE)$ contains $[a'_0, +\infty) \times \cdots \times [a'_m, +\infty)$ for $a'_0, \dots, a'_m$ sufficiently large because the intrinsic radii are always nonzero. For (b), by drawing zig-zag lines parallel to axes linking the two points $s_{J^+}$ and $s'_{J^+}$, it suffices to consider the case when $s_j = s'_j$ for $j \in J^+ \bs \{j_0\}$ and $s_{j_0} \geq s'_{j_0}$.  In this case, we may base change to the completion of $K(\delta_{J^+ \bs \{j_0\}})$ with respect to the $s_{J^+ \bs \{j_0\}}$-Gauss norm.  The result follows from \cite[Theorem~2.4.4(c)]{KSK-me-pDE-TRP}.
\end{proof}

\subsection{Ramification filtrations}
\label{S:review-AS-fil}

In this subsection, we sketch Abbes and Saito's definition of ramification filtrations on the Galois group $G_K$ of a complete discrete valuation field $K$ of mixed characteristic $(0,p)$.  For more details, one can consult \cite{AS-cond1} and \cite{AS-cond2}.

In this subsection, we temporarily drop Notation~\ref{N:J+1}.

\begin{notation}
  For any complete discrete valuation field $K$ of mixed characteristic $(0, p)$, we denote its ring of integers and residue field by $\calO_K$ and $k$, respectively.  Let $\pi_K$ denote a uniformizer and $\gothm_K$ denote the maximal ideal of $\calO_K$ (generated by $\pi_K)$.  We normalize the valuation $v_{K}(\cdot)$ on $K$ so that $v_K(\pi_K) = 1$; the \emph{absolute ramification degree} is defined to be $\beta_K = v_K(p)$.  We say that $K$ is \emph{absolutely unramified} if $\beta_K = 1$.  For an element $a \in \calO_K$, we write its reduction in $k$ as $\bar a$; $a$ is called a \emph{lift} of $\bar a$.

We choose and fix an algebraic closure $K^\alg$ of $K$; all finite extensions of $K$ are taken inside $K^\alg$.  Let $G_K$  denote the absolute Galois group $\Gal(K^\alg / K)$.  If $L$ is a finite Galois extension of $K$, we denote the Galois group by $G_{L/K}$.  We use $\bbN_{L/K}(x)$ to denote the norm of an element $x \in L$.  If $L$ is a (not necessarily algebraic) complete extension of $K$ and is itself a discrete valuation field, we use $e_{L/K}$ to denote its \emph{na\"ive ramification degree}, i.e., the index of the value group of $K$ in that of $L$.  We say that $L/K$ is tamely ramified if $p \nmid e_{L/K}$ and the residue field extension $l/k$ is \emph{algebraic} and separable.  If moreover $e_{L/K} = 1$, we say that $L/ K$ is unramified.
\end{notation}

\begin{notation}\label{N:K-and-L}
From now on, $K$ will denote a complete discrete valuation field of mixed characteristic $(0, p)$, and $L$ will be a finite Galois extension of $K$ of na\"ive ramification degree $e = e_{L/K}$.  Set $\theta = |\pi_K|$; this agrees the convention in the previous subsection.
\end{notation}

\begin{definition}\label{D:AS-space}
Take $Z = (z_j)_{j \in J} \subset \calO_L$ to be a finite set of elements generating $\calO_L$ over $\calO_K$, i.e., $\calO_K [u_J] / \calI \isom \calO_L$ mapping $u_j$ to $z_j$ for all $j \in J = \{\serie{}m\}$. Let $(f_i) _{i=\serie{}n}$ be a finite set of generators of $\calI$.  For $a\in\QQ_{>0}$, define the \emph{Abbes-Saito space} to be
\[
AS_{L/K, Z}^a = \big\{ (u_1, \dots, u_m) \in A_K^m[0, 1] \; \big| \;
|f_i(u_J)| \leq \theta^a,   1 \leq i \leq n \big\}.
\]

We denote the set of \emph{geometric} connected components of $AS_{L/K,Z}^a$ by $\pi_0^\geom (AS_{L/K, Z}^a)$. The \emph{highest ramification break} $b(L/K)$ of the extension $L/K$ is defined to be the minimal $b \in \RR_{\geq 0}$ such that for any rational number $a>b,$ $\# \pi_0^\geom (AS_{L/K,Z}^a) = [L:K]$.
\end{definition}

\begin{definition}\label{D:AS-space-log}
Keep the notation as above.  Take a subset $P \subset Z$ and assume that $P$ and hence $Z$ contain $\pi_L$. Let $e_j = v_L(z_j)$, $z_j \in P$. Take a lift $g_j \in \calO_K [u_J]$ of $z_j^e / \pi_K^{e_j}$ for each $z_j \in P$; take a lift $h_{i,j} \in \calO_K [u_J]$ of $z_j^{e_i} / z_i^{e_j}$ for each pair $(z_i, z_j) \in P \times P$.  For $a\in \QQ_{>0}$, define the \emph{logarithmic Abbes-Saito space} to be
\[
AS_{L/K, \log,Z,P}^a = \left\{ (u_J) \in A_K^m[0,1] \Bigg| 
\begin{array}{cc}
|f_i(u_J)| \leq \theta^a,  & 1 \leq i \leq n \\
|u_j^e - \pi_K^{e_j} g_j| \leq \theta^{a+e_j} & \textrm{for all } z_j \in P \\
|u_j^{e_i} - u_i^{e_j} h_{i,j}| \leq \theta^{a+e_i e_j/e}& \textrm{for all } (z_i, z_j) \in P \times P
\end{array}
\right\}.
\]

Similarly, the \emph{highest logarithmic ramification break} $b_\log(L/K)$ of the extension $L/K$ is defined to be the minimal $b \in \RR_{\geq 0}$ such that for any rational number $a>b$, $\# \pi_0^\geom (AS_{L/K,\log, Z,P}^a) = [L:K]$.
\end{definition}

We reproduce several statements from \cite{AS-cond1} and \cite{AS-cond2}.

\begin{proposition} \label{P:AS-space-properties}
The Abbes-Saito spaces have the following properties.

\emph{(1)} For $a\in \QQ_{>0}$, the Abbes-Saito spaces $AS_{L/K,Z}^a$ and $AS_{L/K, \log, Z, P}^a$ do not depend on the choices of the generators $(f_i)_{i = \serie{}n}$ of $\calI$ and the lifts $g_j$ and $h_{i,j}$ for $i, j \in P$ \emph{\cite[Section~3]{AS-cond1}}.

\emph{(1')} If in the definition of both Abbes-Saito spaces, we choose polynomials $(f_i)_{i = \serie{}n}$ as generators of $\Ker(\calO_K \langle u_J \rangle \rar \calO_L)$ instead of $\Ker(\calO_K[u_J] \rar \calO_L)$, the spaces do not change.

\emph{(2)} If we use another pair of generating sets $Z$ and $P$ satisfying the same properties, then we have a canonical bijection on the sets of the geometric connected components $\pi_0^\geom (AS_{L/K, Z}^a)$ and $\pi_0^\geom (AS_{L/K,\log,Z,P}^a)$ for different generating sets, where $a\in \QQ_{>0}$. In particular, both highest ramification breaks are well-defined \emph{\cite[Section~3]{AS-cond1}}.

\emph{(3)} The highest ramification break (resp. highest logarithmic ramification break) gives rise to a filtration on the Galois group $G_K$ consisting of normal subgroups $\Fil^a G_K$ for (resp., $\Fil_{\log}^a G_K$) for $a\geq 0$ such that $b(L/K) = \inf\{a| \Fil^aG_K \subseteq G_L\}$ (resp. $b_\log(L/K) = \inf\{a| \Fil^a_\log G_K \subseteq G_L\}$) \emph{\cite[Theorems~3.3, 3.11]{AS-cond1}}. Moreover, for $L/K$ a finite Galois extension, both highest ramification breaks are rational numbers \emph{\cite[Theorem~3.8, 3.16]{AS-cond1}}.

\emph{(4)} Let $K'/K$ be a (not necessarily finite) extension of complete discrete valuation fields.  If $K'/K$ is unramified, then $\Fil^a G_{K'} = \Fil^a G_K$ for $a>0$ \emph{\cite[Proposition~3.7]{AS-cond1}}. If $K'/K$ is tamely ramified with ramification index $e < \infty$, then $\Fil_\log^{ea}G_{K'} = \Fil_\log^a G_K$ for $a>0$ \emph{\cite[Proposition~3.15]{AS-cond1}}.

\emph{(4')} More generally, let $L/K$ be a finite algebraic extension and let $K'/K$ be a complete extension of discrete valuation fields with the \emph{same} valued group and linearly independent of $L$.  Denote $L' = K'K$.  If $\calO_{L'} = \calO_L \otimes_{\OK} \calO_{K'}$, then $b(L/K) = b(L'/K')$ \emph{\cite[Lemme~2.1.5]{AM-sous-groupes}}.

\emph{(5)} For $a\geq 0$, define $\Fil^{a+} G_K = \overline{\cup_{b>a}\Fil^b G_K}$ and $\Fil_\log^{a+} G_K = \overline{\cup_{b>a}\Fil_\log^b G_K}$. Then, the subquotients $\Fil^a G_K / \Fil^{a+} G_K$ are abelian $p$-groups if $a \in \QQ_{> 1}$ and are $0$ if $a \notin \QQ$, except when $K$ is absolutely unramified \emph{(\cite[Theorem~3.8]{AS-cond1} and \cite[Theorem~1]{AS-cond2})}.  The subquotients $\Fil_\log^a G_K / \Fil_\log^{a+} G_K$ are abelian $p$-groups if $a \in \QQ_{> 0}$ and are $0$ if $a \notin \QQ$ \emph{(\cite[Theorem~3.16]{AS-cond1}, \cite[Theorem~1]{AS-cond2})}.

\emph{(6)} For $a >0$, $\Fil^{a+1} G_K \subseteq \Fil_\log^a G_K \subseteq \Fil^a G_K$ \emph{\cite[Theorem~3.15(1)]{AS-cond1}}.

\emph{(7)} The inertia subgroup is $\Fil^{a} G_K$ for $a \in (0, 1]$ and the wild inertia subgroup is $\Fil^{1+} G_K = \Fil^{0+}_\log G_K$ \emph{\cite[Theorems~3.7 and 3.15]{AS-cond1}}.

\emph{(8)} When the residue field $k$ is perfect, the arithmetic ramification filtrations agree with the classical upper numbered filtration \emph{\cite{BOOK-local-fields}} in the following way: $\Fil^a G_K = \Fil_\log^{a-1} G_K = G_K^{a-1}$ for $a \geq 1$, where $G_K^a$ is the classical upper numbered filtration on $G_K$ \emph{\cite[Section~6.1]{AS-cond1}}.
\end{proposition}
\begin{proof}
Only (1') is not proved in any literature.  But one can prove it verbatim as (1).  For a brief summary of the proofs for other statements, one may consult  \cite[Proposition~4.1.6]{Me-condI}.  (Although the statements there are stated for equal characteristic case, the proofs work just fine.)
\end{proof}

\begin{remark}
To avoid confusion, we point out that in the proof of our main theorem, we do not need (5) and the second statement of (3) on the rationality of the breaks in the proposition above.  Therefore, we will prove these properties along the way of proving the main theorem.
\end{remark}

\begin{remark}
Recently, T. Saito \cite{saito3} gave a proof of the fact that $\Fil_\log^a G_K / \Fil_\log^{a+} G_K$ are  abelian groups killed by $p$ for $a \in \QQ_{> 0}$.  This will be proved independently in our main theorem (which in fact appeared before his preprint).
\end{remark}

\begin{definition}
For $b \geq 0$, we write $\Fil^b G_{L/K} = (G_L \Fil^b G_K) /G_L$ and $\Fil^b_\log G_{L/K} = (G_L \Fil^b_\log G_K) /G_L$.  We call $b$ a \emph{non-logarithmic (resp. logarithmic) ramification break} of $L/K$ if $\Fil^b G_{L/K} / \Fil^{b+} G_{L/K}$ (resp. $\Fil^b_\log G_{L/K} / \Fil^{b+}_\log G_{L/K}$) is non-trivial.
\end{definition}

\begin{definition}
By a representation of $G_K$, we mean a continuous homomorphism $\rho: G_K \rar GL(V_\rho)$, where $V_\rho$ is a finite dimensional vector space over  a  field $F$ of characteristic zero.  We allow $F$ to have a non-archimedean topology; hence the image of $G_K$ may not be finite.  We say that $\rho$ has \emph{finite monodromy} if the image of the inertia subgroup of $G_K$ is finite.
\end{definition}

\begin{definition}\label{D:conductors}
For a representation $\rho: G_K \rar GL(V_\rho)$ of $G_K$ with finite monodromy, define the \emph{Artin and Swan conductors} of $\rho$ as
\begin{eqnarray}
\Art(\rho) & \Def & \sum_{a \in \QQ_{\geq 0}} a \cdot \dim \big( V_\rho^{\Fil^{a+} G_K} \big/ V_\rho^{\Fil^a G_K} \big), \\
\Swan(\rho) &\Def& \sum_{a \in \QQ_{\geq 0}} a \cdot \dim \big( V_\rho^{\Fil_\log^{a+} G_K} \big/ V_\rho^{\Fil_\log^a G_K} \big).
\end{eqnarray}
In fact, they are finite sums.
\end{definition}

\begin{conjecture}[Hasse-Arf Theorem]
\label{C:HA-conj}
Let $K$ be a complete discrete valuation field of mixed characteristic $(0, p)$ and let $\rho: G_K \rar GL(V_\rho)$ be a representation with finite monodromy.  Then we have

(1) $\Art(\rho)$ and $\Swan(\rho)$ are non-negative integers, and

(2) the subquotients $\Fil^a G_K / \Fil^{a+}G_K$ and $\Fil^a_\log G_K / \Fil^{a+}_\log G_K$ are abelian groups killed by $p$.
\end{conjecture}

In Theorems~\ref{T:main-thm-nonlog}, \ref{T:log-HA-thm}, and \ref{T:log-subquot-p-ele}, we will prove this conjecture except in the absolutely unramified and non-logarithmic case, or the $p=2$ and logarithmic case.

\begin{proposition}
\label{P:classical-HA-thm}
When the residue field $k$ is perfect, Conjecture~\ref{C:HA-conj} is true.
\end{proposition}
\begin{proof}
By Proposition~\ref{P:AS-space-properties}(8), it follows from the classical Hasse-Arf theorem \cite[\S~VI.2~Theorem 1]{BOOK-local-fields}.
\end{proof}

\section{Construction of Spaces}
In this section, we construct a series of rigid analytic spaces and study their relations; in particular, we prove that the Abbes-Saito spaces are the same as thickening spaces, and hence translate the question on ramification breaks to a question on generic radii of differential modules.

\subsection{Standard Abbes-Saito spaces}
\label{S:standard-AS}

In this subsection, we introduce the standard Abbes-Saito spaces by choosing a distinguished set of generators of $\calO_L / \calO_K$.

\begin{definition}
For a field $k$ of characteristic $p$, a \emph{$p$-basis} of $k$ is a set $\bar b_J \subset k$ such that $\bar b_J^{e_J}$, where $e_j \in \{0, 1, \dots, p-1\}$ for all $j \in J$ and $e_j = 0$ for all but finitely many $j$, form a basis of $k$ as a $k^p$-vector space.  For a complete discrete valuation field $K$ of mixed characteristic $(0, p)$, a \emph{$p$-basis} is a set of lifts $b_J \subset \OK$ of a $p$-basis of the residue field $k$.
\end{definition}

\begin{hypo}\label{H:J-finite-set}
Throughout this section, let $K$ be a discrete valuation field of mixed characteristic $(0, p)$ with \emph{separably closed} and \emph{imperfect} residue field.  Assume that $K$ admits a \emph{finite} $p$-basis.  Also, let $L/K$ be a wildly ramified Galois extension of na\"ive ramification degree $e = e_{L/K}$.  In particular, $L/K$ is totally ramified and $b(L/K) > 1$, $b_\log(L/K) >0$.
\end{hypo}

\begin{remark}
In case there is a confusion of the terminology here, by wildly ramified extension, we mean a finite extension which is not tamely ramified, namely, it \emph{can} have tamely ramified part.

This is a mild hypothesis because the conductors behave well under unramified base changes, and the tamely ramified case is well-studied.
\end{remark}

\begin{notation}\label{N:J+}
For the rest of the paper, we retrieve Notation~\ref{N:J+1}, namely, let $J = \{\serie{}m\}$ and $J^+ = J \cup \{0\}$.  We will save the notations $j$ and $m$ only for indexing $p$-bases and related variables, and $j=0$ refers to the uniformizer.
\end{notation}

\begin{notation}\label{N:norm-OK-uJ+}
We define a norm on $\OK [u_{J^+}]$: for $h = \sum_{e_{J^+}} \alpha_{e_{J^+}} u_{J^+}^{e_{J^+}}$, where $\alpha_{e_{J^+}} \in \OK$, we set $|h| = \max_{e_{J^+}} \{ |\alpha_{e_{J^+}}|\cdot \theta^{e_0 / e}\}$.  For $a \in \frac 1e\ZZ_{\geq 0}$, denote $N^a$ to be the set of elements with norm $\leq \theta^a$; it is in fact an ideal.
\end{notation}

The following construction provides a good set of generators for the extension $\calO_L / \calO_K$.  Essentially, we just need some generators and relations with no redundancy which we can write down and work with.

\begin{construction}
\label{C:generators-of-calI}
Choose $p$-bases $b_J \subset \OK$ and $c_J \subset \calO_L$ of $K$ and $L$, respectively.  Let $\bbk_0 = k$ with $p$-basis $(\bar b_j)_{j \in J}$.  By possibly rearranging the indexing in $b_J$, we can filter the extension $l/k$ by subextensions $\bbk_j = k(\bar c_1, \dots, \bar c_j)$ with $p$-bases $\big\{ \bar c_1, \dots, \bar c_j, \bar b_{j+1}, \dots, \bar b_m \big\}$ for $j \in J$.  Moreover, if $[\bbk_j: \bbk_{j-1}] = p^{r_j}$, then $\bar c_j^{p^{r_j}} \in \bbk_{j-1}$.

Write $\Delta: \calO_K \langle u_{J^+} \rangle / \calI_{L/K} \isom \calO_L$ mapping $u_j$ to $c_j$ for $j \in J$ and $u_0$ to $\pi_L$, where $\calI_{L/K}$ is some proper ideal.  Let $\overline\Delta$ be the composite of $\Delta$ with the reduction $\calO_L \surj l$.  Hence,
\begin{equation} \label{E:basis-standard}
\big\{u_{J^+}^{e_{J^+}} |e_j \in \{0, \dots, p^{r_j}-1\} \textrm{ for all } j\in J \textrm{, and } e_0 \in \{0, \dots, e-1\}\big\}
\end{equation}
form a basis of $\calO_K \langle u_{J^+} \rangle / \calI_{L/K}$ as a free $\calO_K$-module.  We choose a set of generators $p_{J^+}$ of $\calI_{L/K}$ by writing each $u_j^{p^{r_j}}$ (for $j \in J$) or $u_0^e$ (for $j =0$) in terms of the basis \eqref{E:basis-standard}.  We say that $p_j$ \emph{corresponds} to $c_j$.  Obviously, $p_{J^+}$ generates $\calI_{L/K}$.  Moreover,
\begin{eqnarray*}
p_j & \in & u_j^{p^{r_j}} - \tilde b_j(u_1, \dots, u_{j-1}) + N^{1/e} \cdot \calO_K [u_{J^+}],\quad j \in J, \\
p_0 & \in & u_0^e - d(u_1, \dots, u_m)\pi_K + \pi_K N^{1/e} \cdot \calO_K [u_{J^+}],
\end{eqnarray*}
where $\tilde b_j(u_1, \dots, u_{j-1}) \in \calO_K [\serie{u_}{j-1}]$ with powers on $u_i$ smaller than $p^{r_i}$ for all $i = 1, \dots, j-1$, where $d(u_1, \dots, u_m) \in \calO_K[u_1, \dots, u_m]$ is a polynomial such that $d(c_1,\dots, c_m) \in \calO_L^\times$.
\end{construction}

\begin{remark}
\label{R:Ohkubo}
One may not avoid introducing $\tilde b_j(u_1, \dots, u_{j-1})$ and $d(u_1, \dots, u_m)$.  Counterexamples are provided and communicated to the author by Shun Ohkubo, see \cite[Remark 3.3.6 and Example 3.3.10]{Me-condI}.  However, to best convey the idea of the proof, the readers are invited to pretend that these two elements are trivial, which is already quite general.
\end{remark}

\begin{definition}
\label{D:standard-AS}
The \emph{(standard) Abbes-Saito spaces} $AS_{L/K}^a$ for $a\in \QQ_{>1}$ and $AS_{L/K, \log}^a$ for $a\in \QQ_{>0}$ are defined by taking  generators to be $\{c_J, \pi_L\}$ and  relations to be $p_{J^+}$ (see Proposition~\ref{P:AS-space-properties}(1')).  In particular, their rings of functions are
\begin{eqnarray*}
& \calO_{AS, L/K}^a = K \langle u_{J^+}, \pi_K^{-a}V_{J^+} \rangle \big/ \big(p_0 (u_{J^+}) - V_0, \dots, p_m (u_{J^+}) - V_m \big), \textrm{ and} \\
&\calO_{AS, L/K, \log}^a = K \langle u_{J^+}, \pi_K^{-a-1} V_0, \pi_K^{-a}V_J \rangle \big/ \big(p_0 (u_{J^+}) - V_0, \dots, p_m (u_{J^+}) - V_m \big).
\end{eqnarray*}
\end{definition}

\subsection{The $\psi$-function and thickening spaces}\label{S:psi-map}

In this subsection, we first define a function (\emph{not} a homomorphism) $\psi: \calO_K \rar \calO_K \llbracket \delta_0/ \pi_K, \delta_J \rrbracket$, which is an approximation to the deformation of the uniformizer $\pi_K$ and $p$-basis as in \cite[Theorem~3.2.7]{Me-condI}.  Then, we introduce the thickening spaces for the extension $L/K$ (See \cite[Section~3.1]{Me-condI} for motivations).

As a reminder, we assume Hypothesis~\ref{H:J-finite-set} for this section; we fix a finite $p$-basis $(b_J)$ and a uniformizer $\pi_K$ of $K$.

\begin{construction}\label{C:psi-map}
Let $r \in \NN$ and $h \in \calO_K^\times$.  An \emph{$r$-th $p$-basis decomposition} of $h$ is to write $h$ as
\begin{equation}\label{E:f-element-in-OK}
h = \sum_{e_J = 0}^{p^r-1} b_J^{e_J} \Big( \sum_{n=0}^\infty \big( \sum_{n'=0}^{\lambda_{r, e_J, n}} \alpha_{r, e_J, n, n'}^{p^r}\big) \pi_K^n \Big)
\end{equation}
for some $\alpha_{r, e_J, n, n'} \in \calO_K^\times \cup \{0\}$ and some $\lambda_{r, e_J, n} \in \ZZ_{\geq 0}$.  Such expressions always exist but are not unique.  For $r'>r$, we can express each of $\alpha_{r, e_J, n, n'}$ in \eqref{E:f-element-in-OK} using an $(r'-r)$-th $p$-basis decomposition and then rearrange the formal sum to obtain an $r'$-th $p$-basis decomposition.  For $h \in \calO_K^\times$, we say that an $r'$-th $p$-basis decomposition is \emph{compatible} with the $r$-th $p$-basis decomposition in \eqref{E:f-element-in-OK} if it can be obtained in the above sense.

We define the function $\psi: \calO_K \rar \calO_K \llbracket \delta_{J^+} \rrbracket$ as follows: for each $h \in \calO_K^\times \bs \{1\}$, we fix a compatible system of $r$-th $p$-basis decomposition for all $r \in \NN$, and define
\begin{equation}\label{E:psi-map}
\psi(h) = \lim_{r \rar +\infty} \sum_{e_J = 0}^{p^r-1} (b_J+\delta_J)^{e_J} \Big( \sum_{n=0}^\infty \big( \sum_{n'=0}^{\lambda_{r, e_J, n}} \alpha_{r, e_J, n, n'}^{p^r} \big) (\pi_K+\delta_0)^n \Big);
\end{equation}
this expression converges by the compatibility of the $p$-basis decompositions.  Define $\psi(1) = 1$, which corresponds to the na\"ive compatible system of $p$-basis decomposition of the element $1$.  For $h \in \calO_K \backslash \{0\}$, write $h = \pi_K^s h_0$ for $s \in \NN$ and $h_0 \in \calO_K^\times$.  Define $\psi(h) = (\pi_K + \delta_0)^s \psi'(h_0)$, where $\psi'(h_0)$ is the limit as in \eqref{E:psi-map} with respect to a compatible system of $p$-basis decompositions of $h_0$ (which does not have to be the same as the one that defines $\psi(h_0)$).  Finally, we define $\psi(0) = 0$.

Most of the time, it is more convenient to view $\psi$ as a function on $\calO_K$ which takes value in the larger ring $\OK \llbracket \delta_0/\pi_K, \delta_J \rrbracket$.

We naturally extend $\psi$ to polynomial rings or formal power series rings with coefficients in $\calO_K$ by applying $\psi$ termwise.
\end{construction}

\begin{notation}
\label{N:R_K}
For the rest of the paper, let $\calR_K = \OK \llbracket \delta_0 / \pi_K, \delta_J \rrbracket$.
\end{notation}

\begin{caution}
The map $\psi$ is \emph{not} a homomorphism, nor is it canonically defined. This is because one cannot ``deform" the uniformizer in the mixed characteristic case.  Moreover, since $K$ will not be absolutely unramified in applications, $p$-basis may not deform  freely either.  However, Proposition~\ref{P:psi-almost-hom} below says that $\psi$ is approximately a homomorphism.
\end{caution}

\begin{remark}
In the $p$-basis decomposition \eqref{E:f-element-in-OK}, we allow extra freedom given by $n'$.  So, we have the freedom of writing $1+p$ as itself or as $1+1+\cdots+1$.  This is one of the place where the above ambiguity  comes from. Allowing this extra freedom in $n'$ is in fact not necessary except at Construction~\ref{C:psi-K*}, where we need the diagram~\eqref{E:KK_*-base-change} to \emph{commute}.
\end{remark}

\begin{definition} \label{D:approx-homo}
For two $\OK$-algebras $R_1$ and $R_2$ and an ideal $I$ of $R_2$, an \emph{approximate homomorphism modulo $I$} is a function $f: R_1 \rar R_2$ such that for $h_1 \in \pi_K^{a_1}R_1$ and $h_2 \in \pi_K^{a_2}R_2$ with $a_1, a_2 \in \ZZ_{\geq 0}$, $\psi(h_1h_2) - \psi(h_1)\psi(h_2) \in \pi_K^{a_1 + a_2}I$ and $\psi(h_1+h_2) - \psi(h_1) - \psi(h_2) \in \pi_K^{\min\{a_1, a_2\}}I$.

Moreover, if $R'_1$ and $R'_2$ are two $\calO_K$-algebras, a diagram of functions
\[
\xymatrix{
R'_1 \ar[r]^{f'} \ar[d]^g & R'_2 \ar[d]^{g'} \\
R_1 \ar[r]^{f} & R_2
}
\]
is called \emph{approximately commutative modulo $I$} if for $h \in \pi_K^aR'_1$, $g'(f'(h)) - f(g(h)) \in \pi_K^a I$.
\end{definition}

\begin{proposition}\label{P:psi-almost-hom}
For $h \in \OK$, we have $\psi(h) - h \in (\delta_{J^+})\cdot \calO_K \llbracket \delta_{J^+} \rrbracket$.  Modulo $I_K = p(\delta_0/\pi_K, \delta_J) \calR_K$, $\psi(h)$ does not depend on the choice of the compatible system of $p$-basis decompositions.  Moreover, $\psi$ is an approximate homomorphism modulo $I_K$.
\end{proposition}
\begin{proof}
First, $\psi(h) - h \in (\delta_{J^+}) \cdot \calO_K \llbracket \delta_{J^+} \rrbracket$ is obvious from the construction.  Next, we observe that when $p^r >\beta_K$, in any $r$-th $p$-basis decomposition for $h \in \calO_K^\times$, the sum $\sum_{n'=0}^{\lambda_{(r), e_J, n}} \alpha_{(r), e_J, n, n'}^{p^r} \pi_K^n$ for any $e_J$ and $n$ in \eqref{E:f-element-in-OK} is well-defined modulo $p$.  So, the ambiguity of defining $\psi$ lies in $I_K$.

For $h_1, h_2 \in \calO_K^\times$, the formal sum or product of compatible systems of $p$-basis decompositions of $h_1$ and $h_2$ are just some compatible systems of $p$-basis decompositions of $h_1 + h_2$ or $h_1 h_2$.  Thus, $\psi(h_1) + \psi(h_2)$ and $\psi(h_1) \psi(h_2)$ are the same as $\psi(h_1+h_2)$ and $\psi(h_1h_2)$ modulo $I_K$.  The statement for general elements in $\OK$ follows from this.
\end{proof}

\begin{remark}  \label{R:beta_K=1-evil}
From Proposition~\ref{P:psi-almost-hom}, we see that the ideal case is when $\beta_K \gg 1$.  In contrast, when $\beta_K = 1$, $I_K = (\delta_0, p\delta_J)$.  The above proposition does not give us much information about $\psi$.  This is why we are not able to prove Conjecture~\ref{C:HA-conj} in the absolutely unramified and non-logarithmic case.  This reflects the restraints in \cite{AS-cond2} from a different point of view, where Abbes and Saito formulated the dichotomy as follows.
\[
\Omega^1_{\OK / \Zp} \otimes_{\OK} k = \left\{ \begin{array}{ll}
\bigoplus_{j \in J} k\cdot db_j & \textrm{if }\beta_K = 1, \\
\bigoplus_{j \in J} k\cdot db_j \oplus k \cdot d\pi_K & \textrm{if }\beta_K > 1.
\end{array}
\right.
\]
\end{remark}

\begin{hypo} \label{H:beta_K>1}
For the rest of the section, assume that $K$ is not absolutely unramified, i.e., $\beta_K \geq 2$.
\end{hypo}

\begin{lemma} \label{L:psi=diff}
Let $h \in \OK$.  Denote $dh = \bar h_0 d \pi_K + \bar h_1d b_1 + \cdots + \bar h_m d b_m$ when viewed as a differential in $\Omega^1_{\OK / \Zp} \otimes_{\OK} k$.  Then $\psi(h) - h \equiv \bar h_0 \delta_0 + \cdots + \bar h_m \delta_m$ modulo $(\pi_K) + (\delta_0 / \pi_K, \delta_J)^2$ in $\calR_K$.
\end{lemma}
\begin{proof}
For an $r$-th $p$-basis decomposition ($r \geq 1$) as in \eqref{E:f-element-in-OK}, we have, modulo the ideal $(\pi_K) + (\delta_{J^+})(\delta_0 / \pi_K, \delta_J)$, 
\begin{align*}
\psi(h) - h & \equiv \sum_{e_J = 0}^{p^r-1} \sum_{n=0}^\infty \sum_{n'=0}^{\lambda_{(r), e_J, n}} \Big(  (b_J+\delta_J)^{e_J} \alpha_{(r), e_J, n, n'}^{p^r} (\pi_K+\delta_0)^n - 
b_J^{e_J} \alpha_{(r), e_J, n, n'}^{p^r} \pi_K^n \Big) \\
& \equiv \sum_{e_J = 0}^{p^r-1} \sum_{n=0}^\infty \sum_{n'=0}^{\lambda_{(r), e_J, n}} \alpha_{(r), e_J, n, n'}^{p^r} b_J^{e_J}\pi_K^n 
\big( \frac {n\delta_0}{\pi_K} + \frac{e_1\delta_1}{b_1} + \cdots + \frac{e_m\delta_m}{b_m} \big)
\equiv \bar h_0 \delta_0 + \cdots + \bar h_m \delta_m.
\end{align*}
Taking limit does not break the congruence relation.
\end{proof}

\begin{definition}
\label{D:admissible}
Denote $\calS_K = \calR_K \langle u_{J^+} \rangle$.  For $\omega \in \frac 1e \NN \cap [1, \beta_K]$, we say a set of elements $(R_{J^+}) \subset (\delta_{J^+})\cdot \calS_K$ \emph{has error gauge} $\geq \omega$  if $R_0 \in (N^\omega\delta_0, N^{\omega +1} \delta_J)\cdot \calS_K$ and $R_j \in (N^{\omega-1}\delta_0, N^\omega \delta_J) \cdot \calS_K$ for all $j \in J$.  We say that $(R_{J^+})$ is \emph{admissible} if it has error gauge $\geq 1$.
\end{definition}

\begin{definition}\label{D:th-space}
Let $a\in \QQ_{>1}$.  We define the \emph{standard (non-logarithmic) thickening space (of level $a$)} $TS_{L/K, \psi}^a$ of $L/K$ to be the rigid space associated to
$$
\calO_{TS, L/K, \psi}^a = K \langle \pi_K^{-a}\delta_{J^+} \rangle \langle u_{J^+} \rangle \big/ \big(\psi(p_{J^+})\big).
$$
For $(R_{J^+}) \subset (\delta_{J^+}) \cdot \calS_K$ admissible, we define the \emph{(non-logarithmic) thickening space (of level $a$)} $TS^a_{L/K, R_{J^+}}$ to be the rigid space associated to
$$
\calO_{TS, L/K, R_{J^+}}^a = K \langle \pi_K^{-a}\delta_{J^+} \rangle \langle u_{J^+} \rangle \big/ \big(\psi(p_{J^+}) + R_{J^+} \big).
$$

Similarly, for $a\in \QQ_{>0}$, we define the \emph{standard logarithmic thickening space (of level $a$)} $TS_{L/K, \log, \psi}^a$ of $L/K$ to be the rigid space associated to
$$
\calO_{TS, L/K, \log, \psi}^a = K \langle \pi_K^{-a-1} \delta_0, \pi_K^{-a}\delta_J \rangle \langle u_{J^+} \rangle \big/ \big(\psi(p_{J^+})\big).
$$
For $(R_{J^+}) \subset (\delta_{J^+}) \cdot \calS_K$ admissible, we define the \emph{logarithmic thickening space (of level $a$)} $TS^a_{L/K, \log, R_{J^+}}$ to be the rigid space associated to
$$
\calO_{TS, L/K, \log, R_{J^+}}^a = K \langle \pi_K^{-a-1} \delta_0, \pi_K^{-a}\delta_J \rangle \langle u_{J^+} \rangle \big/ \big(\psi(p_{J^+}) + R_{J^+} \big).
$$

Denote $TS_{L/K, R_{J^+}} = \cup_{a\in \QQ_{>0}} TS_{L/K, \log, R_{J^+}}^a$.  Then we have the following natural Cartesian diagram for $a\in \QQ_{>0}$.
\[
\xymatrix{
TS_{L/K, R_{J^+}}^{a+1} \ar[d]^\Pi \ar@{^{(}->}[r] & 
TS_{L/K, \log, R_{J^+}}^a \ar[d]^\Pi \ar@{^{(}->}[r] &
TS_{L/K, R_{J^+}} \ar[d]^\Pi\\
A_K^{m+1}[0, \theta^{a+1}] \ar@{^{(}->}[r] &
 A_K^1[0, \theta^{a+1}] \times A_K^m[0, \theta^a] \ar@{^{(}->}[r] & A_K^1[0, \theta) \times A_K^m[0, 1)
}
\]
Here $\Pi$ denotes the natural projection to the polydiscs with coordinates $\delta_{J^+}$.
\end{definition}

\begin{remark}\label{R:gauge<beta_K}
Error gauge is supposed to measure how ``standard" a thickening space is.  Unfortunately, a standard thickening space itself depends on a very non-canonical function $\psi$.  The upshot is that, by Proposition~\ref{P:psi-almost-hom}, the notion of having error gauge $\geq \omega$ does not depend on the choice of $\psi$ if $\omega \in [1, \beta_K]$; note that the terms in $p_0$ are all divisible by $\pi_K$, except $u_0^e$.
\end{remark}

\begin{remark}
The reason of introducing non-standard thickening spaces (or rather thickening spaces which do not have error gauge $\geq \beta_K$) is, as we will show later, that adding a generic $p$-th root results in the error gauge of $(R_{J^+})$ dropping by one; the comparison Theorem~\ref{T:ts=as} guarantees that as long as $(R_{J^+})$'s are admissible (i.e., $\beta_K \geq 1$), the thickening spaces still compute the same ramification break.  On the same issue, if $\beta_K = 1$, we can not afford to drop the error gauge; this is why we are not able to prove Conjecture~\ref{C:HA-conj} in the absolutely unramified and non-logarithmic case (see also Remark~\ref{R:beta_K=1-evil}).
\end{remark}

\begin{notation}
\label{N:Delta-for-ts}
Let $(R_{J^+}) \subset (\delta_{J^+}) \cdot \calS_K$ be admissible.  By abuse of notation, we still use $\Delta$ to denote the composite
$$
\xymatrix{
\calS_K \big/ \big(\psi(p_{J^+}) + R_{J^+}\big) \ar[rr]^-{\mod (\delta_0/\pi_K, \delta_J)} & & \calO_K \langle u_{J^+} \rangle / (p_{J^+}) \ar[r]^-\Delta_-\simeq & \calO_L
}
$$
We remark that $\psi(p_{J^+}) - p_{J^+} + R_{J^+}$ are in fact contained in the ideal of $\calS_K$ generated by $\delta_{J^+}$.  We denote the composition of $\Delta$ and the reduction $\calO_L \surj l$ by $\overline\Delta$.
\end{notation}

\begin{lemma}\label{L:basis}
Let $(R_{J^+}) \subset (\delta_{J^+}) \cdot \calS_K$ be admissible.  Then
\begin{equation} \label{E:basis-simplification}
\big\{ u_{J^+}^{e_{J^+}} |e_j \in \{0, \dots, p^{r_j}-1\} \textrm{ for all } j\in J \textrm{, and } e_0 \in \{0, \dots, e-1\}\big\}
\end{equation}
is a basis of $\calS_K \big/ (\psi(p_{J^+}) + R_{J^+})$ over $\calR_K$.  As a consequence, it also gives a basis of $\calO_{TS, L/K, R_{J^+}}^a$ over $K \langle \pi_K^{-a}\delta_{J^+} \rangle$ for $a\in \QQ_{>1}$ and a basis of $\calO_{TS, L/K, \log, R_{J^+}}^a$ over $K \langle \pi_K^{-a-1} \delta_0, \pi_K^{-a}\delta_J \rangle$ for $a\in \QQ_{>0}$.  In particular, the morphism $\Pi: TS_{L/K, R_{J^+}} \rar A_K^1[0, \theta) \times A_K^m[0, 1)$ is finite and flat.
\end{lemma}
\begin{proof}
Given an element $h \in \calS_K \big/ (\psi(p_{J^+}) + R_{J^+})$, we first take a representative $\tilde h \in \calS_K$. Then we can simplify it by iteratively replacing $u_0^e$ and $u_j^{p^{r_j}}$ by $u_0^e - \psi(p_0) - R_0$ and $u_j^{p^{r_j}} - \psi(p_j) - R_j$ for $j \in J$, respectively.  This procedure converges and gives an element with the power of $u_0$ smaller than $e$ and power of $u_j$ smaller than $p^{r_j}$ for $j \in J$.
\end{proof}

\subsection{$AS=TS$ theorem}
\label{S:AS=TS}

In \cite{Me-condI}, the essential step which links the arithmetic conductors and the differential conductors is the comparison theorem (\cite[Theorem~4.3.6]{Me-condI}), which asserts that the lifted Abbes-Saito spaces are isomorphic to the thickening spaces.  In the mixed characteristic case, we do not have to lift the Abbes-Saito spaces.  Instead, in this subsection, we prove a (slightly general) comparison theorem over the base field $K$.

Remember that we continue to assume Hypotheses~\ref{H:J-finite-set} and \ref{H:beta_K>1}.  We start with a lemma.

\begin{lemma}\label{L:jacob-of-R_J+}
Let $(R_{J^+}) \subset (\delta_{J^+}) \cdot \calS_K$ be admissible.  We have
$$
\det\Big(\frac{\partial (\psi(p_i) - p_i + R_i)} {\partial \delta_j} \Big)_{i,j \in J^+} \Big| _{\delta_{J^+} = 0} \in  \big( \calO_K \langle u_{J^+}\rangle / (p_{J^+}) \big)^\times = \calO_L^\times.
$$
\end{lemma}
\begin{proof}
The proof is quite similar to \cite[Lemmas~4.3.1~and~4.3.3]{Me-condI}.  We also remark that the proof is made very technical to salvage the appearance of $\tilde b_j(u_1, \dots,u_{j-1})$ and $d(u_1,\dots, u_m)$ and partially $R_{J^+}$ (see Remark~\ref{R:Ohkubo}).  If we could haven taken $\tilde b_j(u_1, \dots,u_{j-1})$ and $d(u_1,\dots, u_m)$ to be 1 and $R_{J^+}=0$, the lemma is almost immediate because the leading term in each $\psi(p_i)-p_i$ is just $\delta_i$, and the matrix becomes the identity matrix modulo $\pi_L$.

It is enough to prove that the matrix is of full rank modulo $\pi_L$.  By Lemma~\ref{L:psi=diff} and the admissibility of $R_{J^+}$, modulo $\pi_L$, the first row will be all zero except the first element which is $d(\bar c_1, \dots, \bar c_m) \in \kappa_L^\times$.  Hence, we need only to look at 
\begin{equation} \label{E:matrix}
\Big(\frac{\partial (\psi(p_i) -p_i)} {\partial \delta_j} \Big)_{i, j \in J} \mod (\pi_L, \delta_0/\pi_K, \delta_J)
 = \Big(\frac{\partial (\psi(\tilde b_i) - \tilde b_i)} {\partial \delta_j} \Big)_{i,j \in J} \mod (\pi_L, \delta_0/\pi_K, \delta_J),
\end{equation}
where $\tilde b_i = \tilde b_i(u_1, \dots, u_{i-1})$ in Construction~\ref{C:generators-of-calI}.
Let $\bar \alpha_{ij} \in l$ denote the entries in the matrix on the right hand side of \eqref{E:matrix}, where we identify $\calO_K \langle u_{J^+} \rangle / (p_{J^+}, u_0) \isom l$.  Under this identification,  $\tilde b_i$ will become $\bar c_i^{p^{r_i}}$ for all $i \in J$.  It suffices to show that the $i$-th row is $l$-linearly independent from the first $i-1$ rows for all $i$.  If we set 
\[
\bar c_i^{p^{r_i}} = \sum_{e_1=0}^{p^{r_0}-1} \cdots \sum_{e_{i-1} = 0}^{p^{r_{i-1}}-1} \bar
\lambda_{e_1, \dots, e_{i-1}} \bar c_1^{e_1} \cdots \bar c_{i-1}^{e_{i-1}},
\]
where $\bar \lambda_{e_1, \dots, e_{i-1}} \in k$, then we would have, modulo $\pi_K$,
\[
\tilde b_i(u_1, \dots, u_{j-1}) \equiv  \sum_{e_1=0}^{p^{r_0}-1} \cdots \sum_{e_{i-1} = 0}^{p^{r_{i-1}}-1} \bar
\lambda_{e_1, \dots, e_{i-1}}  u_1^{e_1} \cdots  u_{i-1}^{e_{i-1}}.
\]
Hence, if we set $d \bar \lambda_{e_1, \dots, e_{i-1}} = \bar \mu_{e_1, \dots, e_{i-1}, 1}d \bar b_1 + \cdots + \bar \mu_{e_1, \dots, e_{i-1}, m}d \bar b_m$, then by Lemma~\ref{L:psi=diff},
\begin{align*}
\bar \alpha_{i1} d\bar b_1 + \cdots + \bar \alpha_{im} d\bar b_m &= \sum_{e_1=0}^{p^{r_0}-1} \cdots \sum_{e_{i-1} = 0}^{p^{r_{i-1}}-1} 
u_1^{e_1} \cdots u_{i-1}^{e_{i-1}} \big( \bar \mu_{e_1, \dots, e_{i-1}, 1}d \bar b_1 + \cdots + \bar \mu_{e_1, \dots, e_{i-1}, m}d \bar b_m \big) \\
& \equiv d(\bar c_i^{p^{r_i}}) \textrm{ modulo } \big(d\bar c_1, \dots, d\bar c_{i-1} \big)
\end{align*}
in $\Omega^1_{\bbk_{i-1}/\Fp}$; it is in fact nontrivial because $d\bar c_1, \dots, d\bar c_m$ form a basis of $\Omega_{\kappa_L/\FF_p}^1$ and hence there should not be any auxiliary relation among $d\bar c_1, \dots, d\bar c_m$ in $\Omega_{\bbk_i/\FF_p}^1$.
But we know that the sums $\bar \alpha_{i'1} d\bar b_1 + \cdots + \bar \alpha_{i'm}d\bar b_m$ for $i'<i$ all lie in the subspace of $\Omega^1_{\bbk_{i-1}/\Fp}$ generated by $d\bar c_1, \dots, d\bar c_{i-1}$.  Hence the $i$-th row of the matrix in \eqref{E:matrix} is ($\bbk_{i-1}$-)linearly independent from the first $i-1$ rows.  The lemma follows.
\end{proof}

\begin{theorem}
\label{T:ts=as}
If $(R_{J^+}) \subset (\delta_{J^+})\cdot \calS_K$ is admissible, we have isomorphisms of $K$-algebras:
\begin{eqnarray*}
\calO^a_{AS, L/K} \simeq \calO^a_{TS, L/K, R_{J^+}} & \textrm{if} & a\in \QQ_{>1}, \\
\calO^a_{AS, L/K, \log} \simeq \calO^a_{TS, L/K, \log, R_{J^+}} & \textrm{if} & a\in \QQ_{>0}.
\end{eqnarray*}
\end{theorem}

\begin{example}
\label{E:AS-TS-example}
Before proving  the theorem, we try to illustrate the idea using an example.

Assume $p>2$.
Let $K$ be the completion of $\QQ_p(\zeta_p)(b)$ with respect to the $1$-Gauss norm on $b(=b_1)$; we take $\pi_K = \zeta_p-1$. (Rigorously speaking, Hypotheses~\ref{H:J-finite-set} requires $K$ to have separably closed residue field; but in fact Theorem~\ref{T:ts=as} holds without this assumption.)  Let $L = K ((b\pi_K)^{1/p})((b+\pi_K)^{1/p})$; it is a Galois extension with inseparable residue field extension and na\"ive ramification degree $p$.  We take the uniformizer of $L$ to be $\pi_L = (b\pi_K)^{1/p}$ and we take $c = (b+\pi_K)^{1/p}$; they generate the extension $\calO_L / \calO_K$ with relations $p_0(u_0, u_1) =p_0(u_0)= u_0^{p}-b\pi_K$ and $p_1(u_0, u_1) = p_1(u_1) = u_1^p - b-\pi_K$.  For $a>0$, the Abbes-Saito space is given by 
\[
\calO_{AS, L/K}^a = K \langle u_0, u_1, \pi_K^{-a}V_0, \pi_K^{-a}V_1 \rangle / (u_0^p-b\pi_K - V_0, u_1^p - b-\pi_K - V_1).
\]

We take the function $\psi: \calO_K \to \calO_K\llbracket\delta_0, \delta_1\rrbracket$ so that $\psi(b) = b+ \delta_1$ and $\psi(b\pi_K) = (b+\delta_1)(\pi_K + \delta_0)$.  Then the standard thickening space is given by
\[
\calO_{TS, L/K, \psi}^a = K \langle u_0, u_1, \pi_K^{-a}\delta_0, \pi_K^{-a}\delta_1 \rangle / (u_0^p - (b+\delta_1)(\pi_K + \delta_0), u_1^p - b-\delta_1-\pi_K - \delta_0).
\]

We will identify these two algebras by matching $u_0$ and $u_1$ from the two algebras.  For this, we first construct a (continuous) homomorphism $\chi_1:  \calO_{AS, L/K}^a \to \calO_{TS, L/K, \psi}^a$ such that $\chi_1(u_0) = u_0$ and $\chi_1(u_1) = u_1$; then we are forced to send $V_0$ to $\chi_1(u_0^p - b\pi_K) =\pi_K \delta_1 + b\delta_0 + \delta_0 \delta_1$, and send $V_1$ to $\chi_1(u_1^p - b-\pi_K) = \delta_0 + \delta_1$.  For $\chi_1$ to be well-defined, we need to check the convergence, which is quite obvious as the way it is written in this particular example.

Conversely, we want to construct the inverse (continuous) homomorphism $\chi_2: \calO_{TS, L/K, \psi}^a \to \calO_{AS, L/K}^a$.  Again, we need $\chi_2(u_1) =u_1$ and $\chi_2(u_2)=u_2$.  It is less obvious where we need to send $\delta_0$ and $\delta_1$.  But we know that the images  $\chi_2(\delta_0)$ and $\chi_2(\delta_1)$ must satisfy
\begin{align*}
b\chi_2(\delta_0)+ \pi_K \chi_2(\delta_1) &=\chi_2(u_0^p - b\pi_K -\delta_0\delta_1) =  V_0 - \chi_2(\delta_0) \chi_2(\delta_1), \textrm{ and }\\
\chi_2(\delta_0) + \chi_2(\delta_1) &= \chi_2(u_1^p-b-\pi_K) =  V_1.
\end{align*}
Thinking of two left hand sides as a system of linear equations, we have
\begin{equation}
\label{E:AS=TS-example}
\begin{pmatrix}
\chi_2(\delta_0)\\
\chi_2(\delta_1)
\end{pmatrix}
= \begin{pmatrix}
b&\pi_K\\1&1
\end{pmatrix}^{-1}
\begin{pmatrix}
V_0- \chi_2(\delta_0)\chi_2(\delta_1)\\
V_1
\end{pmatrix}.
\end{equation}
We can determine the value of $\chi_2(\delta_0)$ and $\chi_2(\delta_1)$ by iteratively plug in the left hand side of \eqref{E:AS=TS-example} to its right hand side.  In our special case, one can check by hand that this process will converge eventually to two elements of $\calO_{AS,L/K}^a$, which are the images of $\chi_2(\delta_0)$ and $\chi_2(\delta_1)$, respectively.  But for general case, it is better to employ a ``fixed-point theorem" argument.
\end{example}

We now prove Theorem~\ref{T:ts=as}.

\begin{proof}
The proof is similar to \cite[Theorem~4.3.6]{Me-condI}.  We will match up $u_{J^+}$ in both rings.

We first observe that
\begin{equation}
\label{E:basis-for-as-ts}
\big\{u_{J^+}^{e_{J^+}} |e_j \in \{0, \dots, p^{r_j}-1\} \textrm{ for all } j \in J \textrm{, and } e_0 \in \{0, \dots, e-1\}\big\}
\end{equation}
forms a basis of $\calO^a_{AS, L/K}$ (resp. $\calO^a_{AS, L/K, \log}$) over $K \langle \pi_K^{-a}V_{J^+} \rangle$ (resp. $K \langle \pi_K^{-a-1} V_0, \pi_K^{-a}V_J \rangle$) as a finite free module.  Given 
$$
h = \sum_{e_{J^+}, e'_{J^+}} \alpha_{e_{J^+}, e'_{J^+}}u_{J^+}^{e_{J^+}} V_{J^+}^{e'_{J^+}}  \in \calO^a_{AS, L/K} \textrm{ (resp. } \calO^a_{AS, L/K, \log})
$$
written in this basis, where $\alpha_{e_{J^+}, e'_{J^+}} \in K$, we define
\begin{eqnarray*}
&|h|_{AS, a} = \max_{e_{J^+}, e'_{J^+}} \big\{ |\alpha_{e_{J^+}, e'_{J^+}}| \cdot \theta^{ae'_0 + \cdots + ae'_m + e_0/e} \big\} \\
& \textrm{(resp. } |h|_{AS, \log, a} = \max_{e_{J^+}, e'_{J^+}} \big\{ |\alpha_{e_{J^+}, e'_{J^+}}| \cdot \theta^{(a+1)e'_0 + ae'_1 + \cdots + ae'_m + e_0/e} \big\} ).
\end{eqnarray*}
It is clear that $\calO_{AS, L/K}^a$ (resp. $\calO_{AS, L/K,\log}^a$) is complete and submultiplicative for this norm (i.e. $|h_1h_2|_{AS, a} \leq |h_1|_{AS,a}|h_2|_{AS,a}$ and $|h_1h_2|_{AS, \log, a} \leq |h_1|_{AS,\log,a}|h_2|_{AS,\log,a}$);  the requirement $a>1$ in the non-logarithmic case guarantees that when substituting $u_0^e$ by $u_0^e - p_0 - V_0$, the norm does not increase.

Similarly, by Lemma~\ref{L:basis}, \eqref{E:basis-for-as-ts} also forms a basis of $\calO^a_{TS, L/K, R_{J^+}}$ (resp. $\calO^a_{TS, L/K, \log, R_{J^+}}$) over $K \langle \pi_K^{-a}\delta_{J^+} \rangle$ (resp. $K \langle \pi_K^{-a-1}\delta_0, \pi_K^{-a}\delta_J \rangle$) as a finite free module.  Given 
$$
h = \sum_{e_{J^+}, e'_{J^+}} \alpha_{e_{J^+}, e'_{J^+}}u_{J^+}^{e_{J^+}} \delta_{J^+}^{e'_{J^+}}  \in \calO^a_{TS, L/K, R_{J^+}} \textrm{ (resp. } \calO^a_{TS, L/K,\log, R_{J^+}})
$$
written in this basis, where $ \alpha_{e_{J^+}, e'_{J^+}} \in K$, we define
\begin{eqnarray*}
& |h|_{TS, a} = \max_{e_{J^+}, e'_{J^+}} \big\{ |\alpha_{e_{J^+}, e'_{J^+}}| \cdot \theta^{ae'_0 + \cdots + ae'_m + e_0/e} \big\} \\
&\textrm{(resp. } |h|_{TS, \log, a} = \max_{e_{J^+}, e'_{J^+}} \big\{ |\alpha_{e_{J^+}, e'_{J^+}}| \cdot \theta^{(a+1)e'_0 + ae'_1 + \cdots + ae'_m + e_0/e} \big\} ).
\end{eqnarray*}
It is clear that $\calO^a_{TS, L/K, R_{J^+}}$ (resp. $\calO^a_{TS, L/K, \log, R_{J^+}}$) is complete and submultiplcative for this norm.  The requirement $a>1$ in the non-logarithmic case guarantees that when substituting $u_0^e$ by $u_0^e - \psi(p_0) - R_0$, the norm does not increase.

Define a continuous homomoprhism $\chi_1: \calO^a_{AS, L/K} \rar \calO^a_{TS, L/K, R_{J^+}}$ (resp. $\chi_1: \calO^a_{AS, L/K, \log} \rar \calO^a_{TS, L/K, \log, R_{J^+}}$) by sending $u_{J^+}$ to $u_{J^+}$ and hence $V_j$ to $p_j(u_{J^+}) = p_j(u_{J^+})- \psi(p_j(u_{J^+}))-R_j$ for all $j \in J^+$. We need to verify the convergence conditions for all $V_j$.  Indeed, Proposition~\ref{P:psi-almost-hom} and the admissibility of $R_{J^+}$ imply that
\begin{eqnarray*}
&|p_j - \psi(p_j)|_{TS, a} \leq \theta^a, \quad |R_j|_{TS, a} \leq \theta^a \textrm{ for all } j \in J^+\\
&\textrm{(resp. }|p_j - \psi(p_j)|_{TS, \log, a} \leq \left\{ \begin{array}{ll}
\theta^{a+1} & j =0 \\
\theta^a & j \in J
\end{array}\right., \  |R_j|_{TS, \log, a} \leq \left\{ \begin{array}{ll}
\theta^{a+1+1/e} & j =0 \\
\theta^{a+1/e} & j \in J
\end{array}\right. ).
\end{eqnarray*}

Now we define the inverse $\chi_2$ of $\chi_1$.  Obviously, one should send $u_{J^+}$ back to $u_{J^+}$.  We need to define $\chi_2(\delta_{J^+})$ properly.  Let $A=(A_{ij})_{i,j \in J^+}$ denote the unique matrix in $\calO_K\llbracket u_{J^+}\rrbracket$ such that
$$
A \equiv \Big(\frac{\partial (\psi(p_i) + R_i)}{\partial \delta_j}\Big)_{i,j \in J^+} \textrm{ mod } (\delta_{J^+}) \cdot \calS_K.
$$
By Lemma~\ref{L:jacob-of-R_J+}, the image of $A$, denoted by $\overline A$, in $\mathrm{Mat}_{m+1}\big( \calO_K \langle u_{J^+}\rangle / (p_{J^+})\big) = \mathrm{Mat}_{m+1}(\calO_L)$ is invertible.
Let $B$ denote the $(m+1) \times (m+1)$ matrix with coefficients in $\oplus_{e_0=0}^{e-1} \oplus_{e_1=0}^{p^{r_1}-1}\cdots \oplus_{e_m=0}^{p^{r_m}-1} \calO_K u_{J^+}^{e_{J^+}}$ whose image in $\mathrm{Mat}_{m+1}\big( \calO_K \langle u_{J^+}\rangle / (p_{J^+}) \big)$ is the \emph{inverse} of $\overline A$.  
Then, we have
\begin{equation}\label{E:A-times-A-1}
B A - I \in \mathrm{Mat}_{m+1} \big( (p_{J^+}) \cdot \calO_K\langle u_{J^+}\rangle \big),
\end{equation}
where $I$ is the $(m+1) \times (m+1)$ identity matrix. 

Define the
subset
\begin{align*}
\Lambda &= \{{}^t(x_0, \dots, x_m) \in (\calO_{AS, L/K}^{a})^{\oplus (m+1)}\,|\, |x_j|_{AS,a} \leq \theta^a, \forall j \in J^+\} \\
\textrm{(resp. }\Lambda &= \{{}^t(x_0, \dots, x_m) \in (\calO_{AS, L/K,\log}^{a})^{\oplus (m+1)}\,|\, |x_0|_{AS,\log,a} \leq \theta^{a+1}, |x_j|_{AS,\log,a} \leq \theta^a, \forall j \in J\} \ ).
\end{align*}
It carries a norm $|\cdot|_\Lambda$ by taking the maximum of $|\cdot|_{AS, a}$ (resp. $|\cdot|_{AS, \log, a}$) over its entries.
Consider the function $\bbF: \Lambda \to \Lambda$
given by
\begin{align}
\label{E:delta-in-terms-of-V1}
&\bbF\begin{pmatrix}
 x_0\\ \vdots\\ x_m
\end{pmatrix} = \begin{pmatrix}
x_0\\ \vdots\\ x_m
\end{pmatrix} -  B 
\begin{pmatrix}
(\psi(p_0)+R_0)(u_{J^+},x_{J^+})\\
\vdots \\
(\psi(p_m)+R_m)(u_{J^+},x_{J^+})
\end{pmatrix}
\\
\label{E:delta-in-terms-of-V}
=&\
(I-BA)\begin{pmatrix}
x_0\\ \vdots\\ x_m
\end{pmatrix} - B\left( \begin{pmatrix}
(\psi(p_0) + R_0)(u_{J^+},x_{J^+})-p_0 \\ \vdots \\ (\psi(p_m) + R_m)(u_{J^+},x_{J^+})-p_m \end{pmatrix} - A \vectzero{x_}m\right)  - B \vectzero{V_}m,
\end{align}
where $(\psi(p_j)+R_j)(u_{J^+},x_{J^+})$ is the formal substitution of $\delta_j$ by $x_j$ for any $j \in J^+$.

To see that $\bbF$ is well-defined, we need to bound the norms of each term in \eqref{E:delta-in-terms-of-V} when ${}^t(x_0, \dots,x_m) \in \Lambda$.
By \eqref{E:A-times-A-1}, $I - BA$ (viewed as an element in $\calO_{AS, L/K}^a$ (resp. $\calO_{AS, L/K, \log}^a$)) has norm $\leq \theta^a$.  Hence, in the non-logarithmic case, the first term of  \eqref{E:delta-in-terms-of-V} has norm $\leq \theta^{2a}$; in the logarithmic case the first term of  \eqref{E:delta-in-terms-of-V} has norm $\leq \theta^{2a}$, except for the first row, which has norm $\leq \theta^{2a+1}$.  By the definition of $A$, the second term  of \eqref{E:delta-in-terms-of-V} has entries in $(\delta_{J^+})^2\calS_K$, except for the first row, which is in $(\delta_{J^+})^2\calS_K \cap (x_0^2, \pi_K x_0)\calS_K$ (because of how $p_0$ is defined).  Hence, in the non-logarithmic case, this term has norm $\leq \theta^{2a-1}$; in the logarithmic case, this term  has norm $\leq \theta^{2a}$, except for the first row, which has norm $\leq \min\{\theta^{a+2}, \theta^{2a}\}\leq \theta^{a+2}$.

Hence, we clearly see that $\bbF$ does map $\Lambda$ into $\Lambda$.  Moreover, we observe that $\bbF$ is contractive, that is, there exists $\varepsilon \in (0,1)$ (in fact, $\varepsilon = \theta^{a-1}$ in the non-logarithmic case, and $\varepsilon = \theta^{\min\{a,1\}}$ in the logarithmic case), such that for $\bbx = {}^t(x_0, \dots,x_m), \bby ={}^t(y_0, \dots,y_m) \in \Lambda$, we have
\[
|\bbF(\bbx)- \bbF(\bby)|_\Lambda < \varepsilon|\bbx-\bby|_\Lambda \ \textrm{ (resp. }|\bbF(\bbx)- \bbF(\bby)| < \varepsilon|\bbx-\bby|_\Lambda).
\]

Therefore, $\bbF$ has a unique fixed-point in $\Lambda$, denoted by $\bbx = {}^t(x_0, \dots,x_m) \in \Lambda$.

Now, we define a continuous homomorphism $\tilde \chi_2: K\langle u_{J^+}, \pi_K^{-a}\delta_{J^+}\rangle \rar \calO_{AS, L/K}^a$ (resp. $\tilde\chi_2: K\langle u_{J^+}, \pi_K^{-a-1}\delta_0, \pi_K^{-a}\delta_{J}\rangle \rar \calO_{AS, L/K, \log}^a$) by $\tilde \chi_2(u_j) = u_j$ for $j \in J^+$ and $\tilde\chi_2 (\delta_j) = x_j$.

We now check that $\tilde \chi_2(\psi(p_j) +R_j) = 0$ for all $j \in J^+$.  Indeed, by \eqref{E:delta-in-terms-of-V1}, we have
\[
B\begin{pmatrix}
\tilde \chi_2(\psi(p_0)+R_0)\\
\vdots\\
\tilde \chi_2(\psi(p_m)+R_m)\\
\end{pmatrix}
=
B 
\begin{pmatrix}
(\psi(p_0)+R_0)(u_{J^+},x_{J^+})\\
\vdots \\
(\psi(p_m)+R_m)(u_{J^+},x_{J^+})
\end{pmatrix}
=
\begin{pmatrix}
 x_0\\ \vdots\\ x_m
\end{pmatrix} - \bbF\begin{pmatrix}
x_0\\ \vdots\\ x_m
\end{pmatrix} =\begin{pmatrix}0\\ \vdots\\ 0\end{pmatrix}.
\]
Hence, $\tilde \chi_2$ factors through a continuous homomorphism $\chi_2: \calO_{TS, L/K, R_{J^+}}^a \to \calO_{AS, L/K}^a$ (resp. $\chi_2: \calO_{TS, L/K,\log, R_{J^+}}^a\to \calO_{AS, L/K,\log}^a$). 

Finally, we claim that $\chi_2$ and $\chi_1$ are inverse to each other.  One may check this from the definition directly.  Alternatively, we observe that, by our definition, they are inverse to each other on a dense subset $K[u_{J^+}]$ (density proved in Lemma~\ref{L:density-u} below). Therefore, they have to be inverse to each other and give an isomorphism between the ring of functions on Abbes-Saito space and the ring of functions on thickening space.
\end{proof}

\begin{remark}
An alternative way to understand this theorem is to think of the thickening spaces as perturbations of the morphisms $AS_{L/K}^a \rar A_K^{m+1}[0, \theta^a]$ and $AS_{L/K, \log}^a \rar A_K^1[0, \theta^{a+1}] \times A_K^m[0, \theta^a]$.  Abbes-Saito spaces will behave better under base change using the new morphisms.
\end{remark}

\begin{lemma}
\label{L:density-u}
Let $(R_{J^+}) \subset (\delta_{J^+}) \cdot \calS_K$ be admissible.  Then $K[u_{J^+}]$ is dense in $\calO_{TS, L/K, R_{J^+}}^a$ and $\calO_{AS, L/K}^a$ for $a\in \QQ_{>1}$, and in  $\calO_{TS, L/K, \log, R_{J^+}}^a$  and $\calO_{AS, L/K, \log}^a$ for $a\in \QQ_{>0}$.
\end{lemma}
\begin{proof}
Since $V_j = p_j(u_{J^+}) \in K[u_{J^+}]$ for all $j \in J^+$, the density of $K[u_{J^+}]$ in $\calO_{AS, L/K}^a$ and $\calO_{AS, L/K, \log}^a$ is obvious from the definition.  We now prove the density for the thickening spaces.  It is enough to show that $\delta_{J^+}$ can be well-approximated by elements of $K[u_{J^+}]$.  We keep the notation as in the proof of Theorem~\ref{T:ts=as}.  Consider a variant of \eqref{E:delta-in-terms-of-V}:
\begin{equation}
\label{E:K[u]-approximation}
\begin{pmatrix}
\delta_0\\ \vdots\\ \delta_m
\end{pmatrix} = 
(I-BA)\begin{pmatrix}
\delta_0\\ \vdots\\ \delta_m
\end{pmatrix} - B\left( \begin{pmatrix}
(\psi(p_0) + R_0)-p_0 \\ \vdots \\ (\psi(p_m) + R_m)-p_m \end{pmatrix} - A \vectzero{\delta_}m\right)  - B \vectzero{p_}m.
\end{equation} 
Now that $I-BA \in \Mat_{m+1}\big((p_{J^+})\cdot \calO_K\langle u_{J^+}\rangle\big)$ implies that the first term of the RHS of \eqref{E:K[u]-approximation} has representatives in $(\delta_0/\pi_K, \delta_{J})^2 \calS_K$ under the quotient $\calS_K \to \calS_K/(\psi(p_{J^+}) + R_{J^+})$.  The second term of the RHS of \eqref{E:K[u]-approximation} is already written in terms of elements in $(\delta_0/\pi_K, \delta_{J})^2 \calS_K$.  The third term of the RHS of \eqref{E:K[u]-approximation} is a vector of elements in $K[u_{J^+}]$.

So, this means that we can approximate $\delta_{J^+}$ using $K[u_{J^+}]$ up to elements in $(\delta_0/\pi_K, \delta_{J})^2 \calS_K$.  We can use the same approximation to approximate $\delta_j\delta_{j'}$ for $j, j'\in J$ in the previous approximation and hence get an approximation of $\delta_{J^+}$ by elements in $K[u_{J^+}]$ up to $(\delta_0/\pi_K, \delta_{J})^3 \calS_K$.  Iterating this construction, we see that $K[u_{J^+}]$ is dense in $\calO_{TS, L/K, R_{J^+}}^a$  for $a\in \QQ_{>1}$ and in  $\calO_{TS, L/K, \log, R_{J^+}}^a$   for $a\in \QQ_{>0}$.
\end{proof}

\subsection{\'Etaleness of  thickening spaces}
\label{S:etale}

In this subsection, we will study a variant of \cite[Theorem~7.2]{AS-cond1} and \cite[Corollary~4.12]{AS-cond2}.

Remember that Hypotheses~\ref{H:J-finite-set} and \ref{H:beta_K>1} are still in force.

\begin{definition}
\label{D:etale-locus}
Let $(R_{J^+}) \subset (\delta_{J^+})\cdot \calS_K$ be an admissible subset.  Let $ET_{L/K, R_{J^+}}$ be the rigid analytic subspace of $A_K^1[0, \eta) \times A_K^m[0, 1)$ over which the morphism $\Pi$ defined in Definition~\ref{D:th-space} is \'etale.   When there is no confusion on the choice of $R_{J^+}$ or the choice is not important, we abbreviate $ET_{L/K,R_{J^+}}$ to $ET_{L/K}$.
\end{definition}

\begin{theorem}\label{T:etaleness-nonlog}
Let $b(L/K)$ be the highest non-logarithmic ramification break of $L/K$.  There exists $\epsilon \in (0, b(L/K) - 1)$ such that $b(L/K)-\epsilon \in \QQ$ and, for any $(R_{J^+}) \subset (\delta_{J^+})\cdot \calS_K$ admissible, $A_K^{m+1}[0, \theta^{b(L/K) - \epsilon}] \subseteq ET_{L/K, R_{J^+}}$.
\end{theorem}
\begin{proof}
This proof is essentially the same as \cite[Proposition~7.5]{AS-cond1}.  The essential point is the ``congruence" $\partial (\psi(p_i) + R_i) / \partial u_j \equiv \partial(p_i) / \partial u_j$ over the said locus.  For the convenience of readers, we include the proof.

Recall from \cite[Proposition~7.3]{AS-cond1} that
\begin{equation}\label{E:etale}
\Omega_{\calO_L/\calO_K}^1 = \oplus_{i = 1}^r \calO_L / \pi_L^{\alpha_i}\calO_L \textrm{ with }\alpha_i < e (b(L/K)-\epsilon)
\end{equation}
for some $\epsilon >0$ and $r \in \NN$.  It does not hurt to take $\epsilon < b(L/K) -1$ and $b(L/K)-\epsilon \in \QQ$.  Let $\calJ = \big(\partial (\psi(p_i) + R_i) / \partial u_j\big)_{i,j \in J^+}$ be the Jacobian matrix of $TS_{L/K, R_{J^+}}^a$ over $A_K^{m+1}[0, \theta^a]$, whose entries are elements in $\calO = \OK \langle u_{J^+}, \pi_K^{-a}\delta_{J^+}\rangle/ \big( \psi(p_i) + R_i \big)$.

Let $a = b(L/K) - \epsilon \in \QQ$.
Suppose that  $\bbx \in A_K^1[0, \theta^a]$ is a $K^\alg$-point at which $\mathrm{det} (\calJ)$ vanishes; it gives a homomorphism $\calO_{TS, L/K, R_{J^+}}^a \to K^\alg$.  We let $x_{J^+}$ and $\nu_{J^+}$ denote the images of $u_{J^+}$ and $\delta_{J^+}$, respectively; we have $x_j, \nu_j \in \calO_{K^\alg}$ and $|\nu_j|\leq \theta^a$, for all $j \in J^+$.  Hence, we have $|p_j(x_{J^+})|\leq \theta^a$ for all $j \in J^+$.

Now, we have two $\calO_K$-algebra \emph{homomorphisms}
\[
\xymatrix@=0pt@C=15pt{
\varphi: \calO_L = \calO_K[u_0, \dots, u_m]/(p_0, \dots, p_m) \ar[rr]&&\calO_{K^\alg} / \pi_K^a \calO_{K^\alg}\\
h(u_{J^+}) \ar@{|->}[rr]&& h(x_{J^+}).\\
\mathrm{ev}_\bbx: \calO = \OK\langle u_{J^+}, \pi_K^{-a}\delta_{J^+}\rangle  / \big( \psi(p_i) + R_i \big) \ar[rr]&&\calO_{K^\alg}\\
h(u_{J^+}, \delta_{J^+}) \ar@{|->}[rr]&& h(x_{J^+}, \nu_{J^+}).
}
\]
Here $\varphi$ is well-defined because $|p_j(x_{J^+})|\leq \theta^a$.

We consider the following commutative diagram of linear maps.
\begin{equation}
\label{E:etale-commutative-diagram}
\xymatrix{
\calO^{\oplus(m+1)} \ar[r]^{\mathrm{ev}_\bbx} \ar[d]^{\calJ} &
\calO_{K^\alg}^{\oplus(m+1)} \ar[d]^{\mathrm{ev}_\bbx(\calJ)}
\ar[rr]^-{\mod \pi_K^a} && \big(\calO_{K^\alg} / \pi_K^a \calO_{K^\alg}\big)^{\oplus(m+1)} \ar[d]^{(\partial p_i /\partial u_j)_{i,j\in J^+}\textrm{ mod }\pi_K^a} &&
\calO_L^{\oplus(m+1)}\ar[ll]_-\varphi \ar[d]^{(\partial p_i /\partial u_j)_{i,j\in J^+}}\\
\calO^{\oplus(m+1)} \ar[r]^{\mathrm{ev}_\bbx}  &
\calO_{K^\alg}^{\oplus(m+1)} \ar[rr]^-{\mod \pi_K^a} & & \big(\calO_{K^\alg} / \pi_K^a \calO_{K^\alg}\big)^{\oplus(m+1)} &&
\calO_L^{\oplus(m+1)}\ar[ll]_-\varphi 
}
\end{equation}
Here, the commutativity is clear except the middle one, which follows from the simple but key fact that $|\nu_{J^+}|\leq \pi_K^a\Rightarrow \mathrm{ev}_\bbx(\calJ) \equiv (\partial p_i /\partial u_j)_{i,j\in J^+}\textrm{ mod }\pi_K^a$.

Now, on one hand, \eqref{E:etale} implies that the cokernel of the right vertical arrow in \eqref{E:etale-commutative-diagram} is isomorphic to $\oplus_{i = 1}^r \calO_L/ \pi_L^{\alpha_i}\calO_L$.  Since $ea > \alpha_i$ for any $i$, the cokernel of the third vertical arrow in \eqref{E:etale-commutative-diagram} is isomorphic to $\oplus_{i = 1}^r \calO_{K^\alg}/ \pi_L^{\alpha_i}\calO_{K^\alg}$.

On the other hand, we have assumed that $\det(\mathrm{ev}_\bbx(\calJ)) = 0$; this implies that the cokernel of the second vertical arrow in \eqref{E:etale-commutative-diagram} has a torsion free constituent.  Therefore, we know that the the cokernel of the third arrow must have a direct summand isomorphic to $\calO_{K^\alg} / \pi_K^a \calO_{K^\alg}$; this contradicts the claim in previous paragraph.  We have the \'etaleness as stated.
\end{proof}

\begin{remark}
Theorem~\ref{T:etaleness-nonlog} (as well as Theorem~\ref{T:etaleness-log} later) states that the \'etale locus $ET_{L/K, R_{J^+}}$ is a bit larger than the locus where $TS_{L/K, R_{J^+}}^a$ (resp. $TS_{L/K, \log, R_{J^+}}^a$) becomes a geometrically disjoint union of $[L:K]$ discs.  This is crucial for the proof of Corollary~\ref{C:AS-break=spec-norms}.
\end{remark}

The following lemma is an easy fact about logarithmic relative differentials.  This is not a good place to introduce the whole theory of logarithmic structure.  For a systematic account of logarithmic structures and log-schemes, one may consult \cite[Section~4]{Kato-Saito-Bloch-cond} and \cite{Kato-log-schemes}.

\begin{lemma}\label{L:log-dif-OL-over-OK}
If we provide $\calO_L$ and $\calO_K$ with the canonical log-structures $\pi_L^\NN \inj \calO_L$ and $\pi_K^\NN \inj \calO_K$, respectively, then the logarithmic relative differentials
$$
\Omega^1_{\calO_L / \calO_K}(\log / \log) = \bigoplus_{j \in J} \calO_L du_j \oplus \calO_L \frac {du_0} {u_0} \Big/ \big(d(p_J), \frac {d(p_0)}{\pi_K}, \frac {d\pi_K}{\pi_K}, dx \textrm{ for } x \in \calO_K\big).
$$
\end{lemma}

\begin{theorem}\label{T:etaleness-log}
Let $b_\log(L/K)$ be the highest logarithmic ramification break of $L/K$.  Then there exists $\epsilon \in (0, b_\log(L/K))$ such that $b_\log(L/K) -\epsilon \in \QQ$, and for any $(R_{J^+}) \subset (\delta_{J^+}) \cdot \calS_K$ admissible, $A_K^1[0, \theta^{b_\log(L/K)+1 - \epsilon}] \times A_K^m[0, \theta^{b_\log(L/K) - \epsilon}] \subseteq ET_{L/K, R_{J^+}}$.
\end{theorem}
\begin{proof}
The proof is similar to Theorem~\ref{T:etaleness-nonlog} except that we need to invoke \cite[Proposition~4.11(2)]{AS-cond2} to give a bound on $\Omega^1_{\calO_L/ \calO_K}(\log / \log)$; the explicit description of $\Omega^1_{\calO_L/ \calO_K}(\log / \log)$ in Lemma~\ref{L:log-dif-OL-over-OK} singles out $\delta_0$ and gives rise to the smaller radius $\theta^{a+1}$.
\end{proof}

\subsection{Construction of differential modules}\label{S:diff-eqn}

In this subsection, we set up the framework of interpreting ramification filtrations by differential modules.

As a reminder, we keep Hypotheses~\ref{H:J-finite-set} and \ref{H:beta_K>1}.

\begin{construction}\label{C:diff-eqns}
Let $(R_{J^+}) \subset (\delta_{J^+}) \cdot \calS_K$ be admissible.   By Lemma~\ref{L:basis}, $\Pi: \Pi^{-1}(ET_{L/K}) \rar ET_{L/K}$ is finite and \'etale.  We call $\calE = \Pi_*(\calO_{\Pi^{-1}(ET_{L/K})})$ \emph{a differential module associated to $L/K$}; it is defined over $ET_{L/K}$ and the differential module structure is given by
$$
\nabla: \calE \rar \Pi_*\big( \Omega^1_{\Pi^{-1}(ET_{L/K})/ K}\big) \simeq \calE \otimes_{\calO_{ET_{L/K}}} \Omega^1_{ET_{L/K} / K} = \calE \otimes_{\calO_{ET_{L/K}}} \Big(\bigoplus_{j \in J^+} \calO_{ET_{L/K}} d\delta_j \Big).
$$
Thus, we can define the actions of differential operators $\partial_j = \partial / \partial \delta_j$ for $j \in J^+$ on $\calE$ and talk about intrinsic radii $IR(\calE; s_{J^+})$ as in Notation~\ref{N:diff-mod} if $A_K^1[0, \theta^{s_0}] \times \cdots \times A_K^1[0, \theta^{s_m}] \subseteq ET_{L/K}$.
\end{construction}

\begin{proposition}\label{P:AS-break=spec-norms}
The following statements are equivalent for $a\in \QQ_{>1}$ (resp. $a\in \QQ_{>0}$):

(1) The highest non-logarithmic (resp., logarithmic) ramification break satisfies $b(L/K) \leq a$ (resp. $b_\log(L/K) \leq a$);

(2) For any (some) admissible $(R_{J^+}) \subset \calS_K$ and any rational number $a' > a$, 
$$ 
\# \pi_0^\geom (TS_{L/K, R_{J^+}}^{a'}) = [L:K] \textrm{ (resp. } \# \pi_0^\geom (TS_{L/K, \log, R_{J^+}}^{a'}) = [L:K]\ ).
$$

(3) For any (some) admissible $(R_{J^+}) \subset \calS_K$, $A_K^{m+1}[0, \theta^a] \subseteq ET_{L/K, R_{J^+}}$ (resp. $A_K^1[0, \theta^{a+1}] \times A_K^m[0, \theta^a] \subseteq ET_{L/K, R_{J^+}}$) and the intrinsic radius of $\calE$ over $A_K^{m+1}[0, \theta^a]$ (resp. $A_K^1[0, \theta^{a+1}] \times A_K^m[0, \theta^a]$) is maximal:
$$
IR(\calE; \underline a) = 1 \textrm{ (resp. } IR(\calE; a+1, \underline a) = 1).
$$
\end{proposition}
\begin{proof}
The proof is similar to \cite[Theorem~3.4.5]{Me-condI}.

$(1) \LRar (2)$ is immediate from Theorem~\ref{T:ts=as}.

$(2) \Rar (3)$: For any rational number $a' > a$, (2) implies that for some finite extension $K'$ of $K$, $TS_{L/K, R_{J^+}}^{a'} \times_K K'$ (resp. $TS_{L/K, \log, R_{J^+}}^{a'} \times_K K'$) has $[L:K]$ connected components and is hence force to be $[L:K]$ copies of $A_{K'}^{m+1}[0, \theta^{a'}]$ (resp. $A_{K'}^1[0, \theta^{a'+1}] \times A_{K'}^m[0, \theta^{a'}]$) because $\Pi$ is finite and flat; in particular, $\Pi$ is \'etale there.  Therefore, $\calE \otimes_K K'$ is a trivial differential module over $A_{K'}^{m+1}[0, \theta^{a'}]$ (resp. $A_{K'}^1[0, \theta^{a'+1}] \times A_{K'}^m[0, \theta^{a'}]$).  As a consequence,
$$
IR(\calE; \underline {a'}) = IR(\calE \otimes K'; \underline{a'}) = 1 \textrm{ (resp. }
IR(\calE; a'+1, \underline{a'}) = IR(\calE \otimes_K K'; a'+1, \underline {a'}) = 1 \ ).
$$
Statement (3) follows from the continuity of intrinsic radii in Proposition~\ref{P:prop-diff-eqns}(a), by taking $a'$ sufficiently close to $a$.

$(3) \Rar (2)$: (3) implies that, for any  rational number $a' > a$, $\calE$ is a trivial differential module on $A_K^{m+1}[0, \theta^{a'}]$ (resp. $A_K^1[0, \theta^{a'+1}] \times A_K^m[0, \theta^{a'}]$).  Indeed, we have a bijection
\begin{equation}\label{E:break=IR}
H^0_{\nabla} (A_K^{m+1}[0, \theta^{a'}], \calE) \stackrel \cong \lrar \calE|_{\delta_{J^+} = 0} \textrm{ (resp. } H^0_{\nabla} (A_K^1[0, \theta^{a'+1}] \times A_K^m[0, \theta^{a'}], \calE) \stackrel \cong \lrar \calE|_{\delta_{J^+} = 0} \ ),
\end{equation}
whose inverse is given by Taylor series.  (The convergence of Taylor series is guaranteed by the condition on intrinsic radii.)  This is in fact a ring isomorphism by basic properties of Taylor series.  The left hand side of \eqref{E:break=IR} is a subring of $\calO_{TS, L/K, R_{J^+}}^{a'}$ (resp. $\calO_{TS, L/K, \log, R_{J^+}}^{a'}$); the right hand side is just $K\langle u_{J^+}\rangle /(p_{J^+}) \simeq L$.  Thus, after the extension of scalars from $K$ to $L$, we can lift the idempotent elements in $L \otimes_K L \simeq \prod_{g \in G_{L/K}} L_g$ to idempotent elements in $\calO_{TS, L/K, R_{J^+}}^{a'} \otimes_K L$ (resp. $\calO_{TS, L/K, \log, R_{J^+}}^{a'} \otimes_K L$).  This proves (2).
\end{proof}

\begin{corollary}\label{C:AS-break=spec-norms}
Given the differential module $\calE$ over $ET_{L/K, R_{J^+}}$ with respect to some admissible subset $(R_{J^+}) \subset (\delta_{J^+}) \cdot \calS_K$, we have
\begin{eqnarray*}
b(L/K) & = & \min \big\{ s \, \big|\; A_K^{m+1}[0, \theta^s] \subseteq ET_{L/K,R_{J^+}} \textrm{ and } IR(\calE; \underline s) = 1 \big\}, \textrm{ and}\\
b_\log(L/K) & = & \min \big\{ s\, \big|\; A_K^1[0, \theta^{s+1}] \times A_K^m[0, \theta^s] \subseteq ET_{L/K,R_{J^+}} \textrm{ and } IR(\calE; s+1, \underline s) = 1 \big\}.
\end{eqnarray*}

In other words, $b(L/K)$ (resp. $b_\log(L/K)$) corresponds to the intersection of the boundary of $Z(\calE)$ (cf. Proposition~\ref{P:prop-diff-eqns}(c))  with the line defined by $s_0 = \cdots = s_m$ (resp. $s_0 - 1 = s_1 = \cdots = s_m$).
\end{corollary}
\begin{proof}
By Theorem~\ref{T:etaleness-nonlog} and Theorem~\ref{T:etaleness-log}, $ET_{L/K,R_{J^+}}$ is large enough for use to pin down the exact boundary of $Z(\calE)$.
The corollary follows from  Propositions~\ref{P:AS-break=spec-norms} and \ref{P:prop-diff-eqns} immediately.
\end{proof}

\subsection{Recursive thickening spaces}
\label{S:recursive-TS}

In this subsection, we introduce a generalization of thickening spaces.  This will give us some freedom when changing the base field.

In this subsection, we continue to assume Hypotheses~\ref{H:J-finite-set} and \ref{H:beta_K>1}.

\begin{construction}\label{C:small-D-sp}
This is a variant of Construction~\ref{C:generators-of-calI}.  First, filter the (inseparable) extension $l/k$ by elementary $p$-extensions
$$
k = k_0 \subsetneq k_1 \subsetneq \cdots \subsetneq k_r = l,
$$
where for each $\lambda = 1, \dots, r$, $k_\lambda = k_{\lambda-1}(\bar \gothc_\lambda)$ with $\bar \gothc_\lambda^p = \bar \gothb_\lambda \in k_{\lambda-1}$.  Denote $\Lambda = \{ \serie{}r \}$.  Pick lifts $\gothc_\Lambda$ of $\bar \gothc_\Lambda$ in $\calO_L$.  Let $e = \seriezero{e_}{r_0} = 1$ be a strictly decreasing sequence of integers such that $e_i \mid e_{i-1}$ for $1 \leq i \leq r_0$.  Set $I = \{\serie{}r_0\}$.  For each $i \in I$, pick an element $\pi_{L,i}$ in $\calO_L$ with valuation $e_i$; in particular, we take $\pi_{L, r_0} = \pi_L$.  It is easy to see that $(\gothc_\Lambda, \pi_{L,I})$ generate $\calO_L$ over $\calO_K$.  So we have an isomorphism
$$
\Delta: \calO_K \langle \gothu_{0,I}, \gothu_\Lambda \rangle / \gothI \isom \calO_L,
$$
sending $\gothu_{0, i} \mapsto \pi_{L, i}$ for $i \in I$ and $\gothu_\lambda \mapsto \gothc_\lambda$ for $\lambda \in \Lambda$, where $\gothI$ is some proper ideal and we use the same $\Delta$ as in Construction~\ref{C:generators-of-calI}.  Moreover,
\begin{equation}  \label{E:basis-recursive}
\Big\{\gothu_{0, I}^{\gothe_{0, I}} \gothu_\Lambda^{\gothe_\Lambda} \Big| \gothe_{0, i} \in \{\seriezero{} {\frac {e_{i-1}}{e_i}} - 1\} \textrm{ for all } i \in I \textrm{ and } \gothe_\lambda \in \{\seriezero{}p-1\} \textrm{ for all } \lambda \in \Lambda \Big\}
\end{equation}
forms a basis of $\calO_K \langle \gothu_{0,I}, \gothu_\Lambda \rangle / \gothI$ as a free $\calO_K$-module, which we refer later as the \emph{standard basis}.

We provide $\calO_K [ \gothu_{0,I}, \gothu_\Lambda ]$ with the following norm: for $h = \sum_{\gothe_{0, I}, \gothe_\Lambda} \alpha_{\gothe_{0, I}, \gothe_\Lambda} \gothu_{0, I}^{\gothe_{0, I}} \gothu_\Lambda^{\gothe_\Lambda}$ with $\alpha_{\gothe_{0, I}, \gothe_\Lambda} \in \calO_K$, we set
$$
|h| = \max_{\gothe_{0, I}, \gothe_\Lambda} \{ |\alpha_{\gothe_{0, I}, \gothe_\Lambda}| \cdot \theta^{(\gothe_{0, 1} \cdot e_1 + \cdots + \gothe_{0, r_0} \cdot e_{r_0})/ e}\}.
$$
For $a \in \frac 1e\ZZ_{\geq 0}$, we use $\gothN^{a}$ to denote the set consisting of elements in $\calO_K[ \gothu_{0,I}, \gothu_\Lambda ]$ with norm $\leq \theta^a$; it is in fact an ideal.

In $\calO_K \langle \gothu_{0,I}, \gothu_\Lambda \rangle / \gothI$, we can write $\gothu_{0,i}^{e_{i-1} / e_i}$ for $i \in I$ and $\gothu_\Lambda^p$ in terms of the basis \eqref{E:basis-recursive}.  This gives a set of generators of $\gothI$:
\begin{eqnarray*}
\gothp_{0,1} & \in & \gothu_{0,1}^{e / e_1} - \gothd_1\pi_K + \gothN^{1+1/e} \cdot \calO_K [\gothu_{0,I}, \gothu_\Lambda],\\
\gothp_{0,i} & \in & \gothu_{0,i}^{e_{i-1} / e_i} - \gothd_i \gothu_{0,i-1} + \gothN^{(e_{i-1}+1)/e} \cdot \calO_K [\gothu_{0,I}, \gothu_\Lambda],\ i \in I \backslash \{1\}, \\
\gothp_\lambda & \in & \gothu_\lambda^p - \tilde \gothb_\lambda + \gothN^{1/e} \cdot \calO_K [\gothu_{0,I}, \gothu_\Lambda],
\end{eqnarray*}
where $\gothd_I$ are some elements in $\calO_K[\gothu_{0,I}, \gothu_\Lambda]$ whose images under $\Delta$ are invertible in $\calO_L$, and for each $\lambda$, $\tilde \gothb_\lambda$ is some element in $\calO_K [\gothu_1, \dots, \gothu_{\lambda -1}]$ whose image under $\Delta$ reduces to $\bar \gothb_\lambda \in k_{\lambda -1}$ modulo $\pi_L$.

We say that $\gothp_\lambda$ \emph{corresponds} to the extension $k_\lambda / k_{\lambda - 1}$.
\end{construction}

\begin{definition}\label{D:error-gauge-rec}
As in Definition~\ref{D:admissible}, we define $\gothS_K = \calR_K\langle \gothu_{0, I}, \gothu_\Lambda \rangle = \calO_K \llbracket \delta_0 / \pi_K, \delta_J \rrbracket \langle \gothu_{0, I}, \gothu_\Lambda\rangle$.  For $\omega \in \frac 1e \NN \cap [1, \beta_K]$, we say that a set of elements $(\gothR_{0,I}, \gothR_\Lambda) \subset (\delta_{J^+}) \cdot \gothS_K$ \emph{has error gauge} $\geq \omega$ if $\gothR_{0, i} \in (\gothN^{\omega -1+e_i/e}\delta_0, \gothN^{\omega + e_i/e}\delta_J) \cdot \gothS_K$ for $i \in I$ and $\gothR_\lambda \in (\gothN^{\omega -1}\delta_0, \gothN^{\omega}\delta_J) \cdot \gothS_K$ for $\lambda \in \Lambda$.  The subset $(\gothR_{0,I}, \gothR_\Lambda) \subset (\delta_{J^+}) \cdot \gothS_K$ is \emph{admissible} if it has error gauge $\geq 1$.

Let $(\gothR_{0,I}, \gothR_\Lambda) \subset (\delta_{J^+}) \cdot \gothS_K$ be admissible.  For $a\in \QQ_{>1}$, we define the \emph{(non-logarithmic) recursive thickening space (of level $a$)} $TS_{L/K, \gothR_{0, I}, \gothR_\Lambda}^a$ to be the rigid space associated to
$$
\calO_{TS, L/K, \gothR_{0,I}, \gothR_\Lambda}^a = K \langle \pi_K^{-a} \delta_{J^+} \rangle \langle \gothu_{0,I}, \gothu_\Lambda\rangle \big/ \big( \psi(\gothp_{0, I}) + \gothR_{0, I}, \psi(\gothp_\Lambda) + \gothR_\Lambda \big).
$$

For $a\in \QQ_{>0}$, we define the \emph{logarithmic recursive thickening space (of level $a$)} $TS_{L/K, \log, \gothR_{0, I}, \gothR_\Lambda}^a$ to be the rigid space associated to
$$
\calO_{TS, L/K, \log, \gothR_{0,I}, \gothR_\Lambda}^a = K \langle \pi_K^{-a-1} \delta_0, \pi_K^{-a} \delta_J \rangle \langle \gothu_{0,I}, \gothu_\Lambda\rangle \big/ \big( \psi(\gothp_{0, I}) + \gothR_{0, I}, \psi(\gothp_\Lambda) + \gothR_\Lambda\big).
$$

We still use $\Delta$ to denote the natural homomorphism
$$
\xymatrix{
\gothS_K \big/ \big( \psi(\gothp_{0, I}) + \gothR_{0, I}, \psi(\gothp_\Lambda) + \gothR_\Lambda \big) \ar[rr]^-{\mod (\delta_0/\pi_K, \delta_J)} & & \calO_K \langle \gothu_{0,I}, \gothu_\Lambda \rangle / (\gothp_{0, I}, \gothp_\Lambda) \ar[r]^-\Delta_-\simeq & \calO_L;
}
$$
we use $\overline \Delta$ to denote the composition with the reduction $\calO_L \rar l$.
\end{definition}

\begin{lemma} \label{L:basis-recursive}
Let $(\gothR_{0, I}, \gothR_\Lambda)\subset (\delta_{J^+}) \cdot \gothS_K$ be admissible.  Then \eqref{E:basis-recursive} forms a basis of $\gothS_K / \big(\psi(\gothp_{0, I}) + \gothR_{0, I}, \psi(\gothp_\Lambda) + \gothR_\Lambda \big)$ as a free $\calR_K$-module, which we refer later as the \emph{standard basis}.  As a consequence, they form a basis of $\calO_{TS, L/K, \gothR_{0,I}, \gothR_\Lambda}^a$ (resp. $\calO_{TS, L/K, \log, \gothR_{0,I}, \gothR_\Lambda}^a$) as a free module over $K \langle \pi_K^{-a} \delta_{J^+} \rangle$ (resp. $K \langle \pi_K^{-a-1} \delta_0, \pi_K^{-a} \delta_J \rangle$).
\end{lemma}
\begin{proof}
Same as Lemma~\ref{L:basis}.
\end{proof}

\begin{example}\label{Ex:D=>recursive-D}
The construction of the thickening spaces in Definition~\ref{D:th-space} is a special case of the above construction.  If we start with a uniformizer $\pi_L$, a $p$-basis $c_J$, and relations $p_{J^+}$ in Construction~\ref{C:generators-of-calI}, the following dictionary translates the information to fit in Construction~\ref{C:small-D-sp}.
\begin{eqnarray*}
\pi_{L, I} & \longleftrightarrow & \pi_L \quad (I = \{1\}), \\
\gothc_\Lambda & \longleftrightarrow & c_1, c_1^p, \dots, c_1^{p^{r_1-1}}, c_2, c_2^p, \dots, c_m^{p^{r_m-1}}, \\
\gothp_{0, I}, \gothp_\Lambda & \longleftrightarrow & \textrm{the ones determined by $\gothc_\Lambda$ and $\pi_{L, I}$}, \\
\gothR_{0, I} & \longleftrightarrow & R_0, \\
\gothR_\lambda & \longleftrightarrow & R_j \textrm{ when } \lambda \textrm{ corresponds to some } c_j^{p^{r_j-1}} \textrm{, and 0 otherwise}.
\end{eqnarray*}

Moreover, this construction preserves the error gauge.
\end{example}

Conversely, we have the following.

\begin{proposition}\label{P:recursive-ts=ts}
Let $(\gothR_{0,I}, \gothR_\Lambda) \subset (\delta_{J^+}) \cdot \gothS_K$ be admissible with error gauge $\geq \omega \in \frac 1e\NN \cap [1, \beta_K]$.  Then, for any choices of $c_J$ and $\pi_L$ as in Construction~\ref{C:generators-of-calI}, there exists an $\calR_K$-isomorphism
\begin{equation}\label{E:small-D=D}
\Theta: \calS_K \big/ \big(\psi(p_{J^+}) + R_{J^+}\big) \isom \gothS_K \big/ \big( \psi(\gothp_{0, I}) + \gothR_{0, I}, \psi(\gothp_\Lambda) + \gothR_\Lambda \big),
\end{equation}
for some admissible $R_{J^+}$ with error gauge $\geq \omega$, such that $\Theta \mod (\delta_0 / \pi_K, \delta_J)$ induces the identity map if we identify both sides (modulo $(\delta_0 / \pi_K, \delta_J)$) with $\calO_L$ via $\Delta$.  This gives rise to isomorphisms between the recursive thickening spaces and thickening spaces.
\[
TS_{L/K, \gothR_{0, I}, \gothR_\Lambda}^a  \simeq TS_{L/K, R_{J^+}}^a \ (a\in \QQ_{>1}) \quad \textrm{and} \quad TS_{L/K, \log, \gothR_{0, I}, \gothR_\Lambda}^a \simeq TS_{L/K, \log, R_{J^+}}^a \ (a\in \QQ_{>0}).
\]
\end{proposition}
\begin{proof}
For each $j \in J$, we express $c_j$ as a polynomial $\tilde \gothc_j$ in $\gothu_{0, I}$ and $\gothu_\Lambda$ with coefficients in $\OK$ via $\Delta^{-1}: \calO_L \isom \OK \langle \gothu_{0, I}, \gothu_\Lambda \rangle / (\gothp_{0, I}, \gothp_\Lambda)$.  We define a continuous homomorphism $\widetilde \Theta: \calS_K \to \gothS_K \big/ \big( \psi(\gothp_{0, I}) + \gothR_{0, I}, \psi(\gothp_\Lambda) + \gothR_\Lambda \big)$ by setting $\widetilde\Theta(u_j) = \psi(\tilde \gothc_j)$ for $j \in J$ and $\widetilde \Theta(u_0) = \gothu_{0, r_0}$.  It is then obvious that for $a \in \frac 1e \ZZ_{\geq 0}$, $\widetilde\Theta(N^a\cdot \calS_K) \subset \gothN^a \cdot \gothS_K$.

We need to determine $R_{J^+}$.  For each fixed $j_0 \in J^+$, since $\Delta(p_{j_0}(u_{J^+}))=0$, we can write
$$
p_{j_0}(\gothu_{0, r_0}, \tilde \gothc_J) = \sum_{i \in I}\gothh_{0, i}\gothp_{0, i} + \sum_{\lambda \in \Lambda} \gothh_\lambda\gothp_\lambda, \quad \textrm{ in } \calO_K \langle \gothu_{0, I}, \gothu_\Lambda \rangle
$$
for some $\gothh_{0, i}, \gothh_\lambda \in \calO_K \langle \gothu_{0, I}, \gothu_\Lambda \rangle$ for $i \in I$ and $\lambda \in \Lambda$.  Moreover, when $j_0 = 0$, we can require $\gothh_{0, i} \in \gothN^{1-e_{i-1}/e}\cdot \calO_K \langle \gothu_{0, I}, \gothu_\Lambda \rangle,$ and $\gothh_\lambda \in \gothN^1\cdot \calO_K \langle \gothu_{0, I}, \gothu_\Lambda \rangle$ for $i \in I$ and $\lambda \in \Lambda$.  Thus, we expect to define $R_{j_0}$ so that, under $\widetilde\Theta$, it is mapped to
\begin{eqnarray*}
-\psi(p_{j_0})(\widetilde\Theta(u_{J^+}))
&=& -\sum_{i \in I} \psi(\gothh_{0, i}) \psi(\gothp_{0, i}) - \sum_{\lambda \in \Lambda} \psi(\gothh_\lambda) \psi(\gothp_\lambda) + \gothE\\
&=& -\sum_{i \in I} \psi(\gothh_{0, i}) (-\gothR_{0, i}) - \sum_{\lambda \in \Lambda} \psi(\gothh_\lambda) (-\gothR_\lambda) + \gothE\\
&\in& \left\{ \begin{array}{ll}
(\gothN^{\omega }\delta_0, \gothN^{\omega +1}\delta_J) \cdot \gothS_K & j_0 =0\\
(\gothN^{\omega -1}\delta_0, \gothN^{\omega }\delta_J) \cdot \gothS_K & j_0 \in J
\end{array}\right. ,
\end{eqnarray*}
where $\gothE \in (\gothN^{\beta_K}\delta_0, \gothN^{(\beta_K+1)}\delta_J) \cdot \gothS_K$ if $j_0=0$ and $\gothE \in (\gothN^{(\beta_K-1)}\delta_0, \gothN^{\beta_K}\delta_J) \cdot \gothS_K$ if $j_0 \in J$; they correspond to the error terms coming from $\psi$ failing to be a homomorphism (See Proposition~\ref{P:psi-almost-hom}).

Thus, we can find polynomials $q_0, \dots, q_m \in \calO_K [u_{J^+}]$ such that
\begin{eqnarray*}
q_0 \in \left\{ \begin{array}{ll}
N^{\omega} \cdot \calS_K & j_0=0\\
N^{\omega -1} \cdot \calS_K & j_0 \in J
\end{array}\right. & &
q_1, \dots, q_m \in \left\{ \begin{array}{ll}
N^{\omega +1} \cdot \calS_K & j_0=0\\
N^{\omega } \cdot \calS_K & j_0 \in J
\end{array}\right., \textrm{ and}\\
-\psi(p_{j_0})(\widetilde \Theta(u_{J^+})) -\widetilde \Theta(q_0\delta_0 +\cdots + q_m\delta_m) &\in& \left\{ \begin{array}{ll}
(\delta_0/\pi_K, \delta_J) (\gothN^{\omega }\delta_0, \gothN^{\omega +1}\delta_J) \cdot \gothS_K & j_0=0\\
(\delta_0/\pi_K, \delta_J) (\gothN^{\omega-1}\delta_0, \gothN^{\omega}\delta_J) \cdot \gothS_K & j_0 \in J
\end{array}\right..
\end{eqnarray*}

Further, we can similarly find approximation of the coefficients for $\delta_j\delta_{j'}$ for $j , j' \in J^+$.  Iterating this approximation gives the expressions for $R_{J^+}$; they clearly have error gauge $\geq \omega$.

By the construction, $\widetilde \Theta$ factors through the quotient by $\psi(p_{J^+})+R_{J^+}$; we then obtain the homomorphism $\Theta$ as in \eqref{E:small-D=D}.
The surjectivity of $\Theta$ follows from the surjectivity modulo $(\delta_0 / \pi_K, \delta_J)$, which is the identity via $\Delta$.  Moreover, a surjective morphism between two finite free modules of the same rank over a noetherian base ring is automatically an isomorphism.  The theorem is proved.
\end{proof}

\begin{remark}
The isomorphism $\Theta$ is not unique.  Basically, $\Theta(u_0) \mod (\gothN^{\omega }\delta_0, \gothN^{\omega +1}\delta_J) \cdot \gothS_K$ and $\Theta(u_j) \mod (\gothN^{\omega -1}\delta_0, \gothN^{\omega }\delta_J) \cdot \gothS_K$ for $j \in J$ are fixed; any lifts of them will give a desired isomorphism (with different $(R_{J^+})$).
\end{remark}

\begin{lemma}\label{L:inv-elts-on-D}
Let $(\gothR_{0, I}, \gothR_\Lambda) \subset (\delta_{J^+}) \cdot \gothS_K$ be admissible.  Then an element
$$
h \in \gothS_K \big/ \big(\psi(\gothp_{0, I}) + \gothR_{0, I}, \psi(\gothp_\Lambda) + \gothR_\Lambda \big)
$$
is invertible if and only if $\Delta(h) \in \calO_L^\times$.  In particular, $\gothu_{0, r_0}^e / \pi_K$ is invertible.
\end{lemma}
\begin{proof}
The necessity is obvious.  To see the sufficiency, we construct the inverse of $h$ directly.  Let $h^{(-1)}$ be a lift of $\Delta(h^{-1}) \in \calO_L^\times$ in $\calO_K \langle \gothu_{0, I}, \gothu_\Lambda \rangle$.  We have $\Delta(1 - h^{(-1)} h) = 0$ and hence $1 - h^{(-1)} h = g \in (\delta_{J^+}) \cdot \gothS_K$.  Thus,
$$
\frac 1h = \frac {h^{(-1)}}{1 - g} = h^{(-1)} \cdot (1+ g + g^2 + \cdots).
$$
The series converges to the inverse of $h$.
\end{proof}

\section{Hasse-Arf Theorems}
\setcounter{equation}{0}

\subsection{Generic $p$-th roots}
\label{S:dummy-var}

The notion of generic $p$-th roots was first (implicitly) introduced by Borger in \cite{Borger-conductor}.  Kedlaya \cite{KSK-Swan1} realized that in the equal characteristic case, adding generic $p$-th roots into the field extension will not change the (differential) non-logarithmic ramification filtration; hence, one can prove the non-logarithmic Hasse-Arf theorem by reducing to the perfect residue field case.

In this subsection, we continue to assume Hypotheses \ref{H:J-finite-set} and \ref{H:beta_K>1}, except for Proposition~\ref{P:gen-rot=>HA-thm}.

\begin{notation}\label{N:K'}
Let $x$ be transcendental over $K$.  Define $K(x)^\wedge$ to be the completion of $K(x)$ with respect to the $1$-Gauss norm and define $K'$ to be the completion of the maximal unramified extension of $K(x)^\wedge$.  Set $L' = K' L$. 
\end{notation}

\begin{lemma}\label{L:L'}
Let $L(x)^\wedge$ be the completion with respect to the $1$-Gauss norm.  Then, $L'$ is the completion of the maximal unramified extension of $L(x)^\wedge$.  In particular, the residue field of $L'$ is $l' = k(x)^\sep \cdot l$, which is separably closed.
\end{lemma}
\begin{proof}
First, $L(x)^\wedge = LK(x)^\wedge$ because the latter is complete and is dense in the former.  So, it suffices to prove that $L'$ is complete and has separable residue field.  Since $L'/K'$ is finite, $L'$ is complete.  Moreover, the residue field $l'$ of $L'$ is separably closed because it is a finite extension of a separably closed field $k(x)^\sep$.
\end{proof}

\begin{proposition}
\label{P:dummy-var}
The highest ramification breaks do not change if we make a base change from $K$ to $K'$.  In other words, $b(L/K) = b(L' /K')$ and $b_\log(L/K) = b_\log(L'/K')$.
\end{proposition}
\begin{proof}
Since $\pi_L$ is a uniformizer of $L'$ and $\calO_L \otimes_{\calO_K} \calO_{K'}$ surjects onto $l'$ by previous lemma, we have $\calO_{L'} = \calO_L \otimes_{\calO_K} \calO_{K'}$.  The result follows from Proposition~\ref{P:AS-space-properties}(4').
\end{proof}

\begin{definition}\label{D:generic-pth-root}
Let $b_{j_0}$ be an element in a $p$-basis of $K$.  We will often need to make a base change $K \inj \widetilde K = K' ((b_{j_0} + x \pi_K)^{1/p})$, a process which we shall refer to as \textit{adding a generic $p$-th root (of $b_{j_0}$)}.  It is clear that the absolute ramification degree $\beta_{\widetilde K}$ equals $\beta_K$.  If we begin with a finite field extension $L/K$, adding a generic $p$-th root will mean considering the extension $\widetilde L =L\widetilde K / \widetilde K$.  We have $G_{\widetilde L / \widetilde K} = G_{L/K}$ as $\widetilde K$ is linearly independent from $L$ over $K$.  By convention, we take $\pi_{\widetilde K} = \pi_K$ as $\widetilde K / K$ is unramified.  We provide $\widetilde K$ with a $p$-basis $\{b_{J\backslash \{j_0\}}, (b_{j_0} + x \pi_K)^{1/p}, x\}$, which has one more element than the original $p$-basis.
\end{definition}

\begin{proposition}\label{P:finite-gen-pth-roots}
Let $L/K$ be as in Hypothesis~\ref{H:J-finite-set}.  Then after finitely many operations of adding generic $p$-th roots, the field extension we begin with becomes a non-fiercely ramified extension, namely, the residue field extension is trivial.
\end{proposition}
\begin{proof}
This proof is almost identical to \cite[Proposition~5.2.3]{Me-condI}, which is stated for equal characteristic complete discrete valuation field and for adding $p^{\infty}$-th roots (see \cite[Definition~5.2.2]{Me-condI}).

First, the tamely ramified part is always preserved under these operations.  So, we can assume that $L/K$ is totally wildly ramified and hence the Galois group $G_{L/K}$ is a $p$-group.  We can filter the extension $L/K$ as $K = K_0 \subset \cdots \subset K_n = L$, where $K_i / K_{i-1}$ is a (wildly ramified) $\ZZ / p\ZZ$-Galois extension and $K_i / K$ is Galois for each $i = \serie{}n$.  Each of these subextensions 

(a) either has inseparable residue field extension (and hence has na\"ive ramification degree 1),

(b) or has trivial residue field extension (and hence has na\"ive ramification degree $p$).

Let $i_0$ be the maximal number such that $K_i / K_{i-1}$ has trivial residual extension for $i = \serie{}i_0$.  Obviously adding a generic $p$-th root does not decrease $i_0$ because after adding a generic $p$-th root, the na\"ive ramification degree of $\widetilde K_{i_0} / \widetilde K$ still equals to the degree $p^{i_0}$.  Now, it suffices to show that after finitely many operations of adding generic $p$-th roots, $K_{i_0+1} / K_{i_0}$ has trivial residue field extension (if $i_0 < n$); this would suffice to imply the proposition.  Suppose the contrary.

Let $g \in G_{K_{i_0+1}/K_{i_0}} \simeq \ZZ / p\ZZ$ be a generator.  We claim that $\gamma = \min_{x \in \calO_{K_{i_0+1}}} \big(v_{K_{i_0+1}}(g(x) - x) \big)$ decreases by at least 1 after adding  generic $p$-th roots of each of the elements in the $p$-basis.  This would suffice to conclude, because $\gamma$ is always a nonnegative integer.

Let $z$ be a generator of $\calO_{K_{i_0+1}}$ as an $\calO_{K_{i_0}}$-algebra.  It satisfies an equation
\begin{equation}\label{E:k_i_0-over-k_i_0+1}
z^p + a_1 z^{p-1} + \cdots  + a_p = 0
\end{equation}
where $\serie{a_}{p-1} \in \gothm_{K_{i_0}}$ and $a_p \in \calO_{K_{i_0}}^\times$ with $\bar a_p \in k_{i_0}^\times \backslash (k_{i_0}^\times)^p = k^\times \backslash (k^\times)^p$.  It is easy to see that $\gamma = v_{K_{i_0}}(g(z) - z)$.

Adding generic $p$-th roots of each of the element in the $p$-basis gives us a field $\widehat K$.  Now, the field extension $\widehat KK_{i_0+1} / \widehat KK_{i_0}$ is also generated by $z$ as above.  But we can write $a_p = \alpha^p +\beta$ for $\alpha \in \calO_{\widehat KK_{i_0}}$ and $\beta \in \gothm_{\widehat KK_{i_0}}$.  Hence if we substitute $z' = z+\alpha$ into \eqref{E:k_i_0-over-k_i_0+1}, we get
$z'^p + a'_1 z'^{p-1} + \cdots + a'_p = 0$, with $\serie{a'_}p \in \gothm_{\widehat KK_{i_0}}$.  Hence, $v_{\widehat KK_{i_0+1}}(z') > 0$.  By assumption that the extension $\widehat KK_{i_0+1} / \widehat KK_{i_0}$ has na\"ive ramification degree $1$, $\pi_{K_{i_0}}$ is a uniformizer for $\widehat KK_{i_0+1}$ and hence $z' / \pi_{K_{i_0}}$ lies in $\calO_{\widehat KK_{i_0+1}}$.  Thus,
$$
\gamma' = \min_{x \in \calO_{\widehat KK_{i_0+1}}} \big(v_{\widehat KK_{i_0+1}}(g(x) - x) \big) \leq v_{\widehat KK_{i_0+1}} \big(g(z' / \pi_{K_{i_0}}) - z' / \pi_{K_{i_0}}\big) = v_{K_{i_0+1}}\big(g(z) - z\big) - 1 = \gamma - 1.
$$

This proves the claim and hence the proposition.
\end{proof}

\begin{remark}
It is worth to point out that, after these operations, the number of elements in the $p$-basis of the resulting field will be more than that of the original field.
\end{remark}

For the following theorem, we do not assume either of Hypotheses \ref{H:J-finite-set} and \ref{H:beta_K>1}.

\begin{proposition}\label{P:gen-rot=>HA-thm}
Fix $\beta_K \in \NN_{>1}$.  Assume that, for any complete discrete valuation field $K$ of mixed characteristic and with absolute ramification degree $\beta_K$, and any field extension $L/K$ satisfying Hypothesis~\ref{H:J-finite-set}, the highest non-logarithmic ramification break is invariant under the operation of adding a generic $p$-th root. Then, for all complete discrete valuation field $K$ of mixed characteristic and with absolute ramification degree $\beta_K$, we have:

(1) $\Art(\rho)$ is a non-negative integer for any representation $\rho: G_K \rar GL(V_\rho)$ with finite monodromy;

(2) the subquotients $\Fil^a G_K / \Fil^{a+}G_K$ are trivial if $a \notin \QQ$ and are abelian groups killed by $p$ if $a \in \QQ_{>1}$.
\end{proposition}
\begin{proof}
(1) Since the conductor is additive and is invariant when base change to the completion of the maximal unramified extension of $K$ (Proposition~\ref{P:AS-space-properties}(4)), we may assume that $\rho$ is irreducible and exactly factors through the Galois group of a totally ramified Galois extension $L/K$.  We may also assume that the residue field $k$ is imperfect and the extension is wildly ramified since the classical case is well-known (Propositions~\ref{P:AS-space-properties}(7) and \ref{P:classical-HA-thm}).  We need only to show that $\Art(\rho) = b(L/K) \cdot \dim \rho \in \ZZ$.

Now we reduce to the finite $p$-basis case.  Choose a finite subset $J_0 \subset J$ such that $k(\bar b_j^{1/p})$ is linearly independent from $l$ for any $j \in J\backslash J_0$.  Pick lifts $b_j \in \calO_K$ of $\bar b_j$ for each $j \in J \bs J_0$.  Define $K_1 = K\Big( b_j^{1/p^n}; j \in J \backslash J_0, n \in \NN \Big)^\wedge$ and $L_1 = K_1L$.  It is easy to see that $[L_1: K_1] = [L:K]$,  $e_{L_1 / K_1} \geq e_{L/K}$, and $[l_1: k_1] \geq [l:k]$, where $k_1$ and $l_1$ are the residue fields of $K_1$ and $L_1$, respectively.  Thus, all the inequalities are forced to be equalities.  This implies $G_{L_1 / K_1} = G_{L/K}$ and $\calO_{L_1} = \calO_L \otimes_{\calO_K} \calO_{K_1}$.  By Proposition~\ref{P:AS-space-properties}(4'), $b(L_1/ K_1) = b(L/K)$.  Therefore, we may reduce to the case when Hypothesis~\ref{H:J-finite-set} holds.

Since adding generic $p$-th roots does not change $\beta_K$, the condition of this proposition says that $b(L/K)$ is invariant under the operation of adding generic $p$-th roots.  By Proposition~\ref{P:finite-gen-pth-roots}, we may assume that $L/K$ is non-fiercely ramified as the base changes do not change the conductor.  In this case, Proposition~\ref{P:AS-space-properties}(4') implies that replacing $K$ by $K\Big( b_j^{1/p^n}; j \in J, n \in \NN \Big)^\wedge$ does not change the conductor.  Hence, we reduce to the classical case; the statement follows from Proposition \ref{P:classical-HA-thm}.

Now we prove (2), following the idea of \cite[Theorem~3.5.13]{KSK-Swan1}.  Let $L$ be a finite Galois extension of $K$ with Galois group $G_{L/K}$; then we obtain an induced filtration on $G_{L/K}$.  It suffices to check that $\Fil^a G_{L/K} / \Fil^{a+}G_{L/K}$ is abelian and killed by $p$; moreover, we may quotient further
to reduce to the case where $\Fil^{a+}G_{L/K}$ is the trivial group but $\Fil^a G_{L/K}$ is not.  As above, we may reduce to the classical case because the ramification break of any intermediate extension between $L$ and $K$ is also preserved under the operations above.  The statement follows from Proposition~\ref{P:classical-HA-thm}.
\end{proof}

\subsection{Base change for generic $p$-th roots}
\label{S:base-change}

In this subsection, we prove the key technical Theorem~\ref{T:base-change}.  We retain Hypotheses~\ref{H:J-finite-set} and \ref{H:beta_K>1}.  When proving the main theorem, we will assume a technical Hypothesis~\ref{H:gen-rotation}, which is satisfied by any recursive thickening space coming from a thickening space by Example~\ref{Ex:D=>recursive-D}.

\begin{notation}  \label{N:base-change}
For this subsection, fix $j_0 \in J$ and $n \in \NN$ coprime to $p$.  As in Definition~\ref{D:generic-pth-root}, let $K(x)^\wedge$ be the completion of $K(x)$ with respect to the $1$-Gauss norm and let $K'$ be the completion of the maximal unramified extension of $K(x)^\wedge$.  Let $\widetilde K = K'((b_{j_0} + x \pi_K^n)^{1/p})$ and $\widetilde L = L \widetilde K$.  Denote $\beta_{j_0} = (b_{j_0} + x \pi_K^n)^{1/p}$ for simplicity.  Denote the residue fields of $\widetilde K$ and $\widetilde L$ by $\tilde k$ and $\tilde l$, respectively.
\end{notation}

\begin{lemma}
If $\bar b_{j_0}^{1/p} \notin l$, we have the ramification break $b(\widetilde L / \widetilde K) = b(L/K)$.
\end{lemma}
\begin{proof}
Since $\tilde l = \tilde k l$, we have $\calO_{\widetilde L} = \calO_{\widetilde K} \otimes_{\calO_K} \calO_L$; the lemma follows from Proposition~\ref{P:AS-space-properties}(4').
\end{proof}

So we need to deal with the non-trivial case when $\bar b_{j_0}^{1/p} \in l$.  We record an elementary lemma first.

\begin{lemma}
\label{L:stoper}
Assume $s \in \ZZ_{\geq0}$ and $\beta_K > s/e+1$.  Let $\pi \in \calO_L$ be such that $\pi/\pi_L^s \in \calO_L^\times$.   Then, there is no $\mu \in \calO_{L'}$ and $b \in \calO_L$ such that $\mu^p - b - x\pi \in \pi_L^{s+1} \calO_{L'}$.
\end{lemma}
\begin{proof}
We use induction on $s$. When $s=0$, this statement is equivalent to $x \notin \tilde l^p + l$, which is true.  Assume that the statement is true for $s<s_0$ with $s_0\in \ZZ_{>0}$.  Suppose that, for $\pi \in \pi_L^{s_0}\calO_L^\times$, we can find $\mu\in \calO_{L'}$ and $b\in \calO_L$, such that  $\mu^p + b - x\pi \in \pi_L^{s_0+1} \calO_{L'}$.  We must have $\mu^p \equiv b \mod \pi_L$. Since $\tilde l^p \cap l = l^p$, we may write $\mu = \mu_0 + \pi_L\mu_1$ with $\mu_0 \in \calO_L$ and $\mu_1 \in \pi_L \calO_{L'}$ such that $b \equiv \mu_0^p \mod \pi_L$.  So, 
\[
\mu^p - b - x\pi \equiv \pi_L^p\mu_1^p + (\mu_0^p-b) +x\pi \mod p
\]
Since $\beta_K > s_0/e+1$ and $x$ is transcendental over $L$, we must have $\mu_0^p-b \in \pi_L^p\calO_L$ and $s_0\geq p$.  We would then have
\[
\mu_1^p + \frac{\mu_0^p-b}{\pi_L^p} + x \frac{\pi}{\pi_L^p} \in \pi_L^{s-p+1}\calO_{L'},
\]
which should not exist by inductive hypothesis.  Contradiction.
\end{proof}

\begin{notation}
From now on, we use $\psi_K$ instead of $\psi$ as we will consider the $\psi$-functions for different fields.
\end{notation}

\begin{notation}
\label{N:R-tilde-K}
Denote $\calR_{\widetilde K} = \calO_{\widetilde K} \llbracket \eta_0 /\pi_K, \eta_J, \eta_{m+1} \rrbracket$.  Applying Construction~\ref{C:psi-map} to $\widetilde K$ gives a function $\psi_{\widetilde K}: \calO_{\widetilde K} \rar \calR_{\widetilde K}$, which is an approximate homomorphism modulo the ideal $I_{\widetilde K} = p(\eta_0/\pi_K, \eta_{J\cup\{m+1\}}) \cdot \calR_K$.
\end{notation}

\begin{lemma}\label{L:base-change}
There exists a unique continuous $\calO_K$-homomorphism $f^*: \calR_K \rar \calR_{\widetilde K}$
such that
$f^*(\delta_j) = \eta_j$ for $j \in J^+ \backslash \{j_0\}$ and  $f^*(\delta_{j_0}) = (\beta_{j_0} + \eta_{j_0})^p - (x + \eta_{m+1}) (\pi_K + \eta_0)^n - b_{j_0}$.  It gives an approximately commutative diagram modulo $I_{\widetilde K}$.
\begin{equation}\label{Diag:base-change}
\xymatrix{
\calO_K \ar@{_{(}->}[d] \ar[r]^-{\psi_K} & \calO_K \llbracket \delta_0 / \pi_K, \delta_J \rrbracket = \calR_K \ar[d]^{f^*}\\
\calO_{\widetilde K} \ar[r]^-{\psi_{\widetilde K}} & \calO_{\widetilde K} \llbracket \eta_0 / \pi_K, \eta_{J \cup \{m+1\}} \rrbracket = \calR_{\widetilde K}
}
\end{equation}

For $a>1$, $f^*$ gives a morphism $f: A_{\widetilde K}^{m+2}[0, \theta^a] \rar A_K^{m+1}[0, \theta^a]$.
\end{lemma}
\begin{proof}
It follows immediately from Proposition~\ref{P:psi-almost-hom}.
\end{proof}

\begin{hypo}\label{H:gen-rotation}
For the next theorem, we assume that in Construction~\ref{C:small-D-sp}, there exists $\lambda_0 \in \Lambda$ such that the field extension $k_{\lambda_0} / k_{\lambda_0-1}$ is given by $k_{\lambda_0} = k_{\lambda_0 -1}(\bar b_{j_0}^{1/p})$ and $\bar \gothc_{\lambda_0} = \bar b_{j_0}^{1/p}$.
\end{hypo}

\begin{theorem}\label{T:base-change}
Assume Hypothesis~\ref{H:gen-rotation} and keep the notation as above.  Moreover, assume that $\beta_K \geq n+1$.  Let $a\in \QQ_{>1}$ and $\omega \geq n +1$.  Let $TS_{L/K, \gothR_{0, I}, \gothR_\Lambda}^a$ be a recursive thickening space with error gauge $\geq \omega$.  Then $TS_{L/K, \gothR_{0, I}, \gothR_\Lambda}^a \times_{A_K^{m+1}[0, \theta^a], f} A_{\widetilde K}^{m+2}[0, \theta^a]$ is a recursive thickening space for $\widetilde L / \widetilde K$ with error gauge $\geq \omega - n$.
\end{theorem}
The reader may skip this proof when reading this paper for the first time, but one may get some feeling of the proof by understanding Example~\ref{E:base-change-example}. 

\begin{example}
\label{E:base-change-example}
We continue with Example~\ref{E:AS-TS-example} and use the notation from there.   As in Notation~\ref{N:base-change}, we set $K'$ be the completion of $K(x) = \QQ_p(\zeta_p)(b, x)^\wedge$ with respect to the $1$-Gauss norm.  (It turns out that $K'$ having separably closed residue field is not important for this example, so we ignore this minor point.)  Let $\widetilde K = K'((b + x \pi_K)^{1/p})$ and $\widetilde L = L \widetilde K$.  Denote $\beta = (b+ x \pi_K)^{1/p}$ for simplicity.  Denote the residue fields of $\widetilde K$ and $\widetilde L$ by $\tilde k = \FF_p(x,b)$ and $\tilde l$, respectively.

We first try to understand the extension $\widetilde L / \widetilde K$ in terms of generators and relations.  Recall that the extension $\calO_L / \calO_K$ is generated by $c = (b+\pi_K)^{1/p}$ and $\pi_L = (b\pi_K)^{1/p}$ with relations $p_0 = u_0^p - b\pi_K$ and $p_1 = u_1^p - b- \pi_K$.  These relations do generate $\widetilde L / \widetilde K$, but they may not generate the extension on the level of rings of integers.  In particular, we need to modify $p_1$ to be 
\[
u_1^p - \beta^p + x\pi_K - \pi_K = 
(u_1 - \beta)^p + x\pi_K - \pi_K + p(u_1^{p-1}\beta- \cdots - \beta^{p-1}u_1).
\]
So, to get a proper relation, we should use the generator $\gothc = \frac{c-\beta}{\pi_L}$ with the proxy $\gothv$.  The relation then becomes
\[
\gothq = \gothv^p + \frac{x-1}{b} + \frac p{b\pi_K}\big( (\beta+ u_0\gothv)^{p-1}\beta - \cdots - (\beta+ u_0\gothv)\beta^{p-1}\big)
\]
Hence $\gothv$ generates an extension of $\widetilde K (\pi_L)$ of degree $p$ with inseparable residue field extension.  The upshot here is that \textit{the introduction of transcendental element $x$ guaranteed that we only divide the relation $p_1$ by an element of norm $|\pi_K|$ but not any further}.

Now, we try to understand the base change $TS_{L/K, \psi}^a \times_{A_K^{2}[0, \theta^a], f}A_{\widetilde K}^3[0, \theta^a]$.  Its ring of functions is just
\begin{equation}
\label{E:base-change-example-eqn}
K\langle u_0, u_1, \pi_K^{-a}\delta_0, \pi_K^{-a}\delta_1\rangle \big/ 
\big(\psi(p_0), \psi(p_1)\big) \otimes_{K\langle \pi_K^{-a}\delta_0, \pi_K^{-a}\delta_1\rangle, f^*}K\langle \pi_K^{-a}\eta_0, \pi_K^{-a}\eta_1, \pi_K^{-a}\eta_2\rangle,
\end{equation}
where $f^*(\delta_0) = \eta_0$ and $f^*(\delta_1) = (\beta+\eta_1)^p - (x+\eta_2)(\pi_K + \eta_0) - b$.

By substituting $u_1$ by $\beta+\eta_1+ u_0 \gothv$, we see that \eqref{E:base-change-example-eqn} becomes
\begin{align*}
K &\langle u_0, \beta+\eta_1 + u_0\gothv, \pi_K^{-a}\eta_0, \pi_K^{-a}\eta_1, \pi_K^{-a}\eta_2 \rangle \big/ (q_1, q_2),\textrm{ where}\\
q_1 &= u_0^p - (\pi_K+\eta_0)(\beta+ \eta_1)^p - (\pi_K + \eta_0)^2(x+\eta_2)\\
q_2 &= (\beta+\eta_1+u_0\gothv)^p - (\beta+\eta_1)^p+ (\pi_K+\eta_0)(x+\eta_2) - (\pi_K + \eta_0) .
\end{align*}
With the help of $q_1$, $q_2$ may be replaced by 
\[
q'_2 = 
\big((\beta+\eta_1)^p -(\pi_K + \eta_0)(x+\eta_2)\big)\gothv^p + p(\cdots) / (\pi_K + \eta_0) + x+\eta_2 - 1.
\]
It may not be too easy to see immediately that $K \langle u_0, \beta + u_0\gothv, \pi_K^{-a}\eta_0, \pi_K^{-a}\eta_1, \pi_K^{-a}\eta_2 \rangle \big/ (q_1, q'_2)$ gives a thickening space for $\widetilde L / \widetilde K$ of error gauge $\leq \beta_K - 1 =p-2$.  But at least $q_1$ is just $\psi_{\widetilde K}( u_0^p - \beta^p\pi_K - x\pi_K^2)$ and the major terms $\big((\beta+\eta_1)^p -(\pi_K + \eta_0)(x+\eta_2)\big)\gothv^p + x+\eta_2 - 1$ of $q'_2$ is close to $\psi(b\gothq)$.
\end{example}

\begin{proof} of Theorem~\ref{T:base-change}. 

\textbf{\underline{Step 1:}} Find the generators of $\calO_{\widetilde L} / \calO_{\widetilde K}$.

The difficulty comes from that $\pi_{L, I}, \gothc_\Lambda$ do not generate $\calO_{\widetilde L}$ over $\calO_{\widetilde K}$ (although they do generate $\widetilde L$ over $\widetilde K$).  We need to change the generator $\gothc_{\lambda_0}$ to an element which either gives

\underline{\textit{Case A:}} the inseparable extension $\tilde l$ of $l(\bar x)^\sep$ which happens when $\widetilde L / \widetilde K$ has na\"ive ramification degree $e$; or

\underline{\textit{Case B:}} a ramified extension of na\"ive ramification degree $p$ which happens when $\widetilde L / \widetilde K$ has na\"ive ramification degree $ep$, in which case, this generator is a uniformizer of $\widetilde L$.

Denote $L' = LK'$, which has residue field $l'=l(\bar x)^\sep$.  Then, we have $\calO_{L'} = \calO_{K'} \otimes_{\OK} \calO_L$.  Hence, $ \calO_{\widetilde K} \otimes_{\calO_K} \calO_L\cong \calO_{\widetilde K} \otimes_{\calO_{K'}} \calO_{L'} \subseteq \calO_{\widetilde L}$.  We may extend the valuation $v_{L'}(\cdot)$ to $\widetilde L$ by allowing rational valuations in Case B.  Let $\beta_{j_0} - \mu$ for $\mu \in \calO_{L'}$ be an element achieving the maximal valuation under $v_{L'}(\cdot)$ among $\beta_{j_0} + \calO_{L'}$.

\textbf{Claim:} we have $\alpha = v_{L'}(\beta_{j_0} - \mu) \leq en/p$ and

in case A, the reduction of $\tilde \gothc_{\lambda_0} = \pi_L^{-\alpha}(\beta_{j_0} - \mu)$ in $\tilde l$ generate $\tilde l$ over $l'$ (we also set $d=1$ by convention);

in case B, $v_{\widetilde L} (\pi_L^{-[\alpha]}(\beta_{j_0} - \mu)) = d/p$ for some $d \in \{1, \dots, p-1\}$, in which case, we fix a $d$-th root $\pi_{\widetilde L, r_0+1}$ of $\pi_L^{-[\alpha]}(\beta_{j_0} - \mu)$, which generates the na\"ively ramified extension $\calO_{\widetilde L} / \calO_{L'}$.

Proof of the Claim: We have the norm $\bbN_{\widetilde L/ L'}(\mu - \beta_{j_0}) = \mu^p - (b_{j_0} + x\pi_K^n)$.  By Lemma~\ref{L:stoper}, there is no $\mu \in \calO_{L'}$ whose $p$-th power can cancel with the $x\pi_K^n$ term , $v_{L'}(\bbN_{\widetilde L/ L'}(\beta_{j_0} - \mu)) \leq en$ and the first statement of the claim follows.  When $\alpha \notin \NN$, we are forced to fall in Case B and the claim is obvious.  Assume for contradiction that $\alpha \in \NN$ and the reduction of $\tilde \gothc_{\lambda_0}$ lies in $l'$.  Then there exists $\mu' \in \calO_{L'}$ such that $\mu' / \pi_L^{\alpha} \equiv \tilde \gothc_{\lambda_0} \pmod{\gothm_{\widetilde L}}$.  But then $\beta_{j_0} - \mu - \mu'$ will have bigger valuation, which contradicts our choice of $\mu$.  This proves the claim.

\vspace{5pt}
\textbf{\underline{Step 2:}}  Find the generating relations.

By previous step, we can write
\[
\calO_{\widetilde K} \langle \tilde \gothu_{0, I}, \tilde \gothu_{\Lambda \bs \lambda_0}, \tilde \gothv \rangle \big/ \big( \tilde \gothp_{0, I}, 
\tilde \gothp_{\Lambda \bs \lambda_0}, \tilde \gothq \big) \simeq \calO_{\widetilde L} .
\]
by sending $\tilde \gothu_{0, I}$ to $\gothc_{0, I}$, $\tilde \gothu_{\Lambda \bs \lambda_0}$ to $\gothc_{\Lambda \bs \lambda_0}$, and $\tilde \gothv$ to $\tilde \gothc_{\lambda_0}$ in Case A and $\pi_{\widetilde L, r_0+1}$ in Case B, where the relations $\gothp_{0, I}, 
\tilde \gothp_{\Lambda \bs \lambda_0}, \tilde \gothq$ corresponding to $\tilde \gothu_{0, I}, \tilde \gothu_{\Lambda \bs \lambda_0}, \tilde \gothv$ can be obtained using Construction~\ref{C:small-D-sp}.  Now, we link these relations to the relations $\gothp_{0, I}, \gothp_{\Lambda}$ for $\calO_L / \calO_K$.  We first lift the isomorphism 
\[
{\bar \chi}: 
\widetilde K \langle \tilde \gothu_{0, I}, \tilde \gothu_{\Lambda \bs \lambda_0}, \tilde \gothv \rangle \big/ \big( \tilde \gothp_{0, I}, 
\tilde \gothp_{\Lambda \bs \lambda_0}, \tilde \gothq \big)
\simeq \widetilde L \cong \widetilde K \otimes_{\OK} \calO_L  \simeq
\widetilde K \langle \gothu_{0, I}, \gothu_\Lambda \rangle \big/ \big(\gothp_{0, I}, \gothp_\Lambda \big) 
\]
to a homomorphism $\chi: \calO_{\widetilde K} \langle \tilde \gothu_{0, I}, \tilde \gothu_{\Lambda \bs \lambda_0}, \tilde \gothv \rangle \rar \calO_{\widetilde K} \langle \gothu_{0, I}, \gothu_\Lambda \rangle [\frac 1{\gothu_{0, r_0}}]$ sending
$\tilde \gothu_{0, I}$ to $\gothu_{0, I}$,
$\tilde \gothu_{\Lambda \bs \lambda_0}$ to $\gothu_{\Lambda \bs \lambda_0}$,
and $\tilde \gothu_{0, r_0}^{[\alpha]} \tilde \gothv$ to the lift of $\bar \chi(\tilde \gothu_{0, r_0}^{[\alpha]} \tilde \gothv)$ using the standard basis defined in Construction~\ref{C:small-D-sp}. Then  $\gothu_{0, r_0}^{(p-1)[\alpha]}\chi(\tilde \gothp_{0, I}), \gothu_{0, r_0}^{(p-1)[\alpha]}\chi(\tilde \gothp_{\Lambda \bs \lambda_0})$ and $\gothu_{0, r_0}^{p[\alpha]} \chi(\tilde \gothq)$ are contained in the ideal $(\gothp_{0, I}, \gothp_\Lambda) \calO_K \langle \gothu_{0, I}, \gothu_\Lambda \rangle$, because the maximal powers of $\tilde \gothv$ in the equations are $p-1, p-1$ and $p$, respectively.

\vspace{5pt}
\textbf{\underline{Step 3:}}  Explain the goal.

We are going to establish an $\calR_{\widetilde K}$-isomorphism $\boldsymbol \chi: \widetilde \calA \isom \calA$, where
\begin{align}
\label{E:base-change-calA}
\calA &= \gothS_K \big/ \big( \psi_K(\gothp_{0, I}) + \gothR_{0, I}, \psi_K(\gothp_\Lambda) + \gothR_\Lambda \big) \otimes_{\calR_K, f^*} \calR_{\widetilde K} \big[\frac 1p \big], \\
\label{E:base-change-calA'}
\widetilde \calA &= \gothS_{\widetilde K} \big[\frac 1p \big] \big/ \big( \psi_{\widetilde K}(\tilde \gothp_{0, I}) + \widetilde \gothR_{0, I}, 
\psi_{\widetilde K}(\tilde \gothp_{\Lambda \bs \lambda_0}) + \widetilde \gothR_{\Lambda \bs \lambda_0}, 
\psi_{\widetilde K}(\tilde \gothq) + \widetilde \gothR_{\tilde \gothq} \big).
\end{align}
Here, $\gothS_{\widetilde K} = \calR_{\widetilde K} \langle \tilde \gothu_{0, I}, \tilde \gothu_{\Lambda \bs \lambda_0}, \tilde \gothv \rangle$ and we can define $\gothN_{\widetilde K}^a$ for $a \in \frac 1{ep} \NN$ similarly to Construction~\ref{C:small-D-sp}.
We first define a ring homomorphism $\widetilde{\boldsymbol\chi}: \gothS_{\widetilde K}[\frac 1p] \to \calA$ by $\widetilde{\boldsymbol\chi}(\tilde \gothu_{0, I}) = \gothu_{0, I}$, $\widetilde{\boldsymbol\chi}(\tilde \gothu_{\Lambda \bs \lambda_0}) = \gothu_{\Lambda \bs \lambda_0}$, and $\widetilde{\boldsymbol\chi}(\tilde \gothv) = \psi_{\widetilde K}(\chi(\tilde \gothv))$; the set $\widetilde \gothR_{0, I}, \widetilde \gothR_{\Lambda \bs \lambda}, \widetilde \gothR_{\tilde\gothq}$ will be admissible with error gauge $\geq \omega - n$ so that $\widetilde{\boldsymbol\chi}$ factors through $\widetilde \calA$.

\vspace{5pt}
\textbf{\underline{Step 4:}} Bound the error gauge.
We first determine $\widetilde \gothR_{0, I}, \widetilde \gothR_{\Lambda \bs \lambda_0}, \widetilde \gothR_{\tilde\gothq}$.  We proceed similarly to Proposition~\ref{P:recursive-ts=ts}.  To write this argument uniformly, we first divide into the following four cases.

Case (a): Denote $\tilde \gothp = \gothu_{0, r_0}^{(p-1)[\alpha]}\tilde \gothp_{0, i_0}$ for some $i_0 \in I$ and $\widetilde \gothR=\gothu_{0, r_0}^{(p-1)[\alpha]} \widetilde \gothR_{0, I}$;

Case (b): Denote $\tilde \gothp = \gothu_{0, r_0}^{(p-1)[\alpha]}\tilde \gothp_{\lambda}$ for $\lambda \in \Lambda \bs \{\lambda_0\}$ and $\widetilde \gothR = \gothu_{0, r_0}^{(p-1)[\alpha]}\widetilde \gothR_\lambda$; 

Case (c): Denote $\tilde \gothp = \gothu_{0, r_0}^{p[\alpha]}\tilde \gothq$ and $\widetilde \gothR = \gothu_{0, r_0}^{p[\alpha]}\widetilde \gothR_{\tilde \gothq}$, assuming we are in Case A;

Case (d): Denote $\tilde \gothp = \gothu_{0, r_0}^{p[\alpha]}\tilde \gothq$ and $\widetilde \gothR = \gothu_{0, r_0}^{p[\alpha]}\widetilde \gothR_{\tilde \gothq}$, assuming we are in Case B;

By Step 2,
$$
\bar \chi(\tilde \gothp) = \sum_{i \in I}\gothh_{0, i}\gothp_{0, i} + \sum_{\lambda \in \Lambda} \gothh_\lambda\gothp_\lambda,
$$
for some $\gothh_{0, i}, \gothh_\lambda \in \calO_{\widetilde K} \langle \gothu_{0, I}, \gothu_\Lambda \rangle$ for $i \in I, \lambda \in \Lambda$.  Moreover, in Case (a) for some $i_0 \in I$, we can require $\gothh_{0, i} \in \gothN_K ^{\max\{(e_{i_0-1}-e_{i-1})/e, 0\}}\cdot \calO_{\widetilde K} \langle \gothu_{0, I}, \gothu_\Lambda \rangle,$ and $\gothh_\lambda \in \gothN_K^{e_{i_0-1}/e}\cdot \calO_{\widetilde K} \langle \gothu_{0, I}, \gothu_\Lambda \rangle$ for $i \in I, \lambda \in \Lambda$; in Case (d), we can require $\gothh_\lambda \in \gothN_K^{1/e}\cdot \calO_{\widetilde K} \langle \gothu_{0, I}, \gothu_\Lambda \rangle$ for $\lambda \in \Lambda$.  Thus, we want to define $\widetilde \gothR \in \gothS_{\widetilde K}$ so that $-\widetilde{\boldsymbol\chi}(\widetilde \gothR)$ equals to
\begin{align*}
\widetilde{\boldsymbol\chi} \big(\psi_{\widetilde K}(\tilde \gothp) \big)
&= \sum_{i \in I} \psi_{\widetilde K}(\gothh_{0, i}) \psi_{\widetilde K}(\gothp_{0, i}) + \sum_{\lambda \in \Lambda} \psi_{\widetilde K}(\gothh_\lambda) \psi_{\widetilde K}(\gothp_\lambda) + \gothE\\
&= \sum_{i \in I} \psi_{\widetilde K}(\gothh_{0, i}) (-\gothR_{0, i}) + \sum_{\lambda \in \Lambda} \psi_{\widetilde K}(\gothh_\lambda) (-\gothR_\lambda) + \gothE\\
&\in \left\{ \begin{array}{ll}
(\gothN^{\omega -1 + e_{i_0-1}/e}\eta_0, \gothN^{\omega +e_{i_0-1}/e }\eta_{J\cup\{m+1\}}) \cdot \gothS_K \otimes_{\calR_K} \calR_{\widetilde K} & \textrm{ case (a),}\\
(\gothN^{\omega -1}\eta_0, \gothN^{\omega}\eta_{J\cup\{m+1\}}) \cdot \gothS_K \otimes_{\calR_K} \calR_{\widetilde K}
 & \textrm{ case (b) or (c),}\\
(\gothN^{\omega -1+1/e}\eta_0, \gothN^{\omega +1/e}\eta_{J\cup\{m+1\}}) \cdot \gothS_K \otimes_{\calR_K} \calR_{\widetilde K} & \textrm{ case (d).} \\
\end{array}\right.
\end{align*}
where the error term $\gothE$ coming from $\psi$ failing to be a homomorphism (See Proposition~\ref{P:psi-almost-hom}) can be bounded as 
\[
\gothE \in \left\{
\begin{array}{ll}
(\gothN^{\beta_K}\eta_0, \gothN^{\beta_K+1}\eta_{J
\cup \{m+1\}}) \cdot \gothS_K \otimes_{\calR_K} \calR_{\widetilde K}  & \textrm{ case (a),}\\
(\gothN^{\beta_K-1}\delta_0, \gothN^{\beta_K}\delta_J) \cdot \gothS_K \otimes_{\calR_K} \calR_{\widetilde K}  & \textrm{ case (b) or (c),}\\
(\gothN^{\beta_K}\eta_0, \gothN^{\beta_K+1}\eta_{J
\cup \{m+1\}}) \cdot \gothS_K \otimes_{\calR_K} \calR_{\widetilde K} & \textrm{ case (d).}
\end{array}\right.
\]

Thus, we can find polynomials $\seriezero{\tilde \gothr_}{m+1} \in \calO_{\widetilde K} [\tilde \gothu_{0, I}, \tilde \gothu_{\Lambda \bs \lambda_0}, \tilde \gothu_{0, r_0}^{[\alpha]} \tilde \gothv] \surj \calO_{\widetilde K} \otimes_{\calO_K} \calO_L$ such that
\begin{align*}
 \tilde \gothr_0 &\in \left\{ \begin{array}{ll}
\tilde \gothu_{0, r_0}^{\omega e-e+e_{i_0-1}} \cdot \calO_{\widetilde K} [\tilde \gothu_{0, I}, \tilde \gothu_{\Lambda \bs \lambda_0}, \tilde \gothu_{0, r_0}^{[\alpha]} \tilde \gothv] & \textrm{ case (a),} \\
\tilde \gothu_{0, r_0}^{\omega e -e} \cdot \calO_{\widetilde K} [\tilde \gothu_{0, I}, \tilde \gothu_{\Lambda \bs \lambda_0}, \tilde \gothu_{0, r_0}^{[\alpha]} \tilde \gothv]
& \textrm{ case (b) or (c),}\\
\tilde \gothu_{0, r_0}^{\omega e -e + 1} \cdot \calO_{\widetilde K} [\tilde \gothu_{0, I}, \tilde \gothu_{\Lambda \bs \lambda_0}, \tilde \gothu_{0, r_0}^{[\alpha]} \tilde \gothv] & \textrm{ case (d);}
\end{array}\right. \\
\tilde \gothr_1, \dots, \tilde \gothr_{m+1}  &\in \left\{ \begin{array}{ll}
\tilde \gothu_{0, r_0}^{\omega e+ e_{i_0-1}} \cdot \calO_{\widetilde K} [\tilde \gothu_{0, I}, \tilde \gothu_{\Lambda \bs \lambda_0}, \tilde \gothu_{0, r_0}^{[\alpha]} \tilde \gothv] & \textrm{ case (a),}\\
\tilde \gothu_{0, r_0}^{\omega e} \cdot \calO_{\widetilde K} [\tilde \gothu_{0, I}, \tilde \gothu_{\Lambda \bs \lambda_0}, \tilde \gothu_{0, r_0}^{[\alpha]} \tilde \gothv]
& \textrm{ case (b) or (c),} \\
\tilde \gothu_{0, r_0}^{\omega e+ 1} \cdot \calO_{\widetilde K} [\tilde \gothu_{0, I}, \tilde \gothu_{\Lambda \bs \lambda_0}, \tilde \gothu_{0, r_0}^{[\alpha]} \tilde \gothv] & \textrm{ case (d);}
\end{array}\right.
\end{align*}
\begin{align*}
&-\widetilde{\boldsymbol\chi}
\big(\psi_{\widetilde K}(\tilde \gothp)\big)-
\widetilde{\boldsymbol\chi}
( \tilde \gothr_0\eta_0 + \cdots + \tilde \gothr_{m+1}\eta_{m+1}) \\
\in & \left\{ \begin{array}{ll}
(\eta_0/\pi_K, \eta_{J\cup\{m+1\}}) (\gothN^{\omega -1 + e_{i_0-1}/e}\eta_0, \gothN^{\omega + e_{i_0-1}/e}\eta_{J\cup \{m+1\}}) \cdot \big(\gothS_K \otimes_{\calR_K} \calR_{\widetilde K} \big) & \textrm{ case (a),}\\
(\eta_0/\pi_K, \eta_{J\cup\{m+1\}}) (\gothN^{\omega -1}\eta_0, \gothN^{\omega }\eta_{J
\cup\{m+1\}}) \cdot \big(\gothS_K \otimes_{\calR_K} \calR_{\widetilde K} \big) & \textrm{ case (b) or (c),} \\
(\eta_0/\pi_K, \eta_{J\cup\{m+1\}}) (\gothN^{\omega -1+1/e}\eta_0, \gothN^{\omega +1/e}\eta_{J
\cup\{m+1\}}) \cdot \big(\gothS_K \otimes_{\calR_K} \calR_{\widetilde K} \big) & \textrm{ case (d).}
\end{array}\right.
\end{align*}

Further, we can similarly approximate the coefficients of $\eta_j \eta_{j'}$ for $j, j' \in J^+ \cup \{m+1\}$.  Repeating this approximation gives the expression of $\widetilde \gothR \in \gothS_{\widetilde K}$.  From this and $\alpha \leq en/p$, we can obtain $\widetilde \gothR_{0, I}, \widetilde \gothR_{\Lambda \bs \lambda_0}, \widetilde \gothR_{\tilde\gothq} \in (\eta_{J^+ \cup \{m+1\}}) \cdot \gothS_{\widetilde K}$ such that
\begin{align*}
\widetilde \gothR_{0, i_0} &\in (\tilde \gothu_{0, r_0}^{\omega e - e + e_{i_0-1} - en}\eta_0, \tilde \gothu_{0, r_0}^{\omega e + e_{i_0-1} - en}\eta_{J \cup \{m+1\}}) \cdot \gothS_{\widetilde K}, \quad i_0 \in I,\\
\widetilde \gothR_{\lambda} &\in (\tilde \gothu_{0, r_0}^{\omega e - e - en}\eta_0, \tilde \gothu_{0, r_0}^{\omega e - en}\eta_{J \cup \{m+1\}}) \cdot \gothS_{\widetilde K}, \quad \lambda \in \Lambda \bs \lambda_0 \\
\widetilde \gothR_{\tilde \gothq} &\in \left \{
\begin{array}{ll}
(\tilde \gothu_{0, r_0}^{\omega e - e - en}\eta_0, \tilde \gothu_{0, r_0}^{\omega e - en}\eta_{J \cup \{m+1\}}) \cdot \gothS_{\widetilde K} & \textrm{in Case A}\\
(\tilde \gothu_{0, r_0}^{\omega e - e - en +1}\eta_0, \tilde \gothu_{0, r_0}^{\omega e - en + 1}\eta_{J \cup \{m+1\}}) \cdot \gothS_{\widetilde K} & \textrm{in Case B};
\end{array} \right.
\end{align*}
they have error gauge $\geq \omega -n$.  Moreover, $\widetilde{\boldsymbol\chi}$ induces a continuous homomorphism $\boldsymbol \chi: \widetilde\calA \to  \calA$.

\textbf{\underline{Step 5:}} Prove that $\boldsymbol \chi$ is an isomorphism.

To prove that $\boldsymbol \chi$ is an isomorphism, it suffices to show the surjectivity, as both $\widetilde \calA$ and $\calA$ are finite free modules over $\calR_{\widetilde K} [\frac 1p]$ of the same rank.  Since \eqref{E:basis-recursive} forms a basis of $\calA$ over $\calR_{\widetilde K} [\frac 1p]$, we need only to show that $\gothu_{0, I}$ and $\gothu_{\Lambda}$ are in the image of $\boldsymbol \chi$.  This is obvious for $\gothu_{0, I}$ and $\gothu_{\Lambda \bs \lambda_0}$.  For $\gothu_{\lambda_0}$, we first find an element in $\calO_{\widetilde K} [\tilde \gothu_{0, I}, \tilde \gothu_{\Lambda \bs \lambda_0}, \tilde \gothu_{0, r_0}^{[\alpha]} \tilde \gothv] \surj \calO_{\widetilde K} \otimes_{\OK} \calO_L$ whose image under $\bar \chi$ is $\gothu_{\lambda_0}$.  Then we use the similar approximation in Step 4 to find an element in $\widetilde \calA$ whose image under $\boldsymbol \chi$ is exactly $\gothu_{\lambda_0}$.  This finishes the proof.
\end{proof}

\begin{remark}\label{R:gen-p-infty-not-work}
We expect that when $\omega$ and hence $\beta_K$ is ``large" compared to $[L:K]$, Theorem~\ref{T:base-change} is also valid if we add a generic $p^\infty$-th root (defined in \cite[Definition~5.2.2]{Me-condI}); this amounts to control the discrepancy between $\calO_{\widetilde L}$ and $\calO_{\widetilde K} \otimes_{\calO_K} \calO_L$.  Hence, in this case, one can obtain a comparison theorem between the arithmetic Artin conductor and Borger's Artin conductor \cite{Borger-conductor} as in \cite[Subsection~5.4]{Me-condI}.
\end{remark}

\subsection{Non-logarithmic Hasse-Arf theorem}
\label{S:non-log-HA-thm}
In this subsection, we apply Theorem~\ref{T:base-change} to obtain the Hasse-Arf Theorem~\ref{T:main-thm-nonlog} for non-logarithmic ramification filtrations.

We assume Hypotheses~\ref{H:J-finite-set} until stating the last theorem.  As a reminder, Hypothesis~\ref{H:gen-rotation} is no longer assumed till the end of the paper.

\begin{notation}
Keep the notation as in Construction~\ref{C:generators-of-calI}.
Fix $j_0 \in J$ and $n \in \NN$.  Let $\widetilde K = K' ((b_{j_0} + x \pi_K^n)^{1/p})$ as in Notation~\ref{N:base-change}.  Denote $\beta_{j_0} = (b_{j_0} + x \pi_K^n)^{1/p}$ for simplicity.
\end{notation}

\begin{lemma}\label{L:base-change-is-isometric}
Assume $p \nmid n$ and $\beta_K \geq n$.  Let $a_{J^+} \subset \RR_{>0}$ and $a_0 = a_{j_0} = a_{m+1} > \max\{\frac{n-1}{p-1}, 1\}$.  Define $a'_j = a_j$ for $j \in J^+ \bs \{j_0\}$ and $a'_{j_0} = a_{j_0} +n-1$.  The morphism $f^*$ defined in Lemma~\ref{L:base-change} restricts to a morphism $$
f: A_{\widetilde K}^1[\theta^{a_0}, \theta^{a_0}] \times \cdots \times A_{\widetilde K}^1[\theta^{a_{m+1}}, \theta^{a_{m+1}}] \rar A_K^1[\theta^{a'_0}, \theta^{a'_0}] \times \cdots \times A_K^1[\theta^{a'_m}, \theta^{a'_m}].
$$
In other words, we change the $j_0$-th radius from $a_{j_0}$ to $a_{j_0}+n-1$.
\end{lemma}
\begin{proof}
It suffices to verify that if $|\eta_0| = |\eta_{j_0}| = |\eta_{m+1}| = \theta^{a_0}$, then $|\delta_j| = \theta^{a_0+n-1}$; indeed
$$
\delta_{j_0} = \big( (\beta_{j_0} + \eta_{j_0})^p - \beta_{j_0}^p \big) - x\big( (\pi_K + \eta_0)^n - \pi_K^n \big) + \eta_{m+1} (\pi_K + \eta_0)^n,
$$
which has norm $\theta^{a_0+n-1}$ because the second term does and other terms have bigger norms.
\end{proof}

\begin{lemma}\label{L:IR-under-gen-rot}
Keep the notation and assumption as in the previous lemma.  Let $\calE$ be a differential module over $A_K^1[0, \theta^{a'_0}] \times \cdots \times A_K^1[0, \theta^{a'_m}]$, then $IR(f^*\calE; a_{J^+}) = IR(\calE; a'_{J^+ \cup \{m+1\}})$.
\end{lemma}
\begin{proof}
The morphism $f^*$ induces the homomorphism on the differentials: $d\delta_j \mapsto d\eta_j$ for $j \in J^+ \backslash \{j_0\}$ and $d\delta_{j_0} \mapsto p(\beta_{j_0} + \eta_{j_0})^{p-1}d\eta_{j_0} + (\pi_K + \eta_0)^n d\eta_{m+1} + n(x + \eta_{m+1})(\pi_K + \eta_0)^{n-1} d\eta_0$.  Thus,
\begin{eqnarray*}
\partial'_j|_{f^*\calE} & = & \partial_j|_\calE, \quad j \in J\backslash \{j_0\}, \\
\partial'_{j_0}|_{f^*\calE} &=& p(\beta_{j_0} + \eta_{j_0})^{p-1} \partial_{j_0}|_{\calE}, \\
\partial'_{m+1} |_{f^*\calE} &=& (\pi_K + \eta_0)^n \cdot \partial_{j_0}|_\calE, \\
\partial'_0 |_{f^*\calE} & = & \partial_0 |_\calE + n(x+\eta_{m+1})(\pi_K + \eta_0)^{n-1} \cdot \partial_{j_0}|_\calE,
\end{eqnarray*}
where $\partial'_j = \partial / \partial \eta_j$ for $j = 0, \dots, m+1$.  Thus,
\begin{eqnarray*}
IR_j(f^*\calE; a_{J^+ \cup \{m+1\}}) & = & IR_j(\calE; a'_{J^+}) \quad \forall j \in J\backslash \{j_0\}, \\
IR_{j_0}(f^*\calE; a_{J^+ \cup \{m+1\}}) &\leq& IR_{j_0}(\calE; a'_{J^+}), \\
IR_{m+1} (f^*\calE; a_{J^+ \cup \{m+1\}}) &=& \theta^n \cdot IR_{j_0} (\calE; a'_{J^+}), \\
IR_0 (f^*\calE; a_{J^+ \cup \{m+1\}}) & = & \min \big\{ IR_0 (\calE, a'_{J^+}), IR_{j_0} (\calE; a'_{J^+}) \big\},
\end{eqnarray*}
where the second inequality follows from Proposition~\ref{P:sp-norm-under-Frob} and the last equality holds by Proposition~\ref{P:off-center-tame} because $x$ is transcendental over $K$.  It follows that $IR(\calE; a'_{J^+}) = IR(f^*\calE; a_{J^+ \cup \{m+1\}})$.
\end{proof}

\begin{theorem}\label{T:AS-inv-gen-rot}
Let $L/K$ be a finite Galois extension satisfying Hypotheses~\ref{H:J-finite-set} and \ref{H:beta_K>1}.  The highest non-logarithmic ramification break of $L/K$ is invariant under the operation of adding a generic $p$-th root.
\end{theorem}
\begin{proof}
Adding a generic $p$-th root corresponds to setting $n=1$ in the notation in this subsection.
Fix a choice of $\psi_K$ in Construction~\ref{C:psi-map}.  Let $TS_{L/K, \psi_K}^a$ be the standard thickening space for $L/K$.  By Example~\ref{Ex:D=>recursive-D}, we can turn this standard thickening space into a recursive thickening space (with error gauge $\geq \beta_K$).  By Theorem~\ref{T:base-change}, $TS_{L/K, \psi_K}^a \times_{A_K^{m+1}[0, \theta^a], f} A_{\widetilde K}^{m+2}[0, \theta^a]$ is a recursive thickening space for $\widetilde L/\widetilde K$ with error gauge $\geq \beta_K - 1$, which is isomorphic to some thickening space for $\widetilde L / \widetilde K$ by Proposition~\ref{P:recursive-ts=ts}.

Let $\calE$ be the differential module over $A_K^{m+1}[0, \theta^a]$ coming from $TS_{L/K, \psi_K}^a$.  Then the differential module $f^* \calE$ is associated to $\widetilde L / \widetilde K$.  Applying Lemma~\ref{L:IR-under-gen-rot} (to the case $n=1$) gives $IR(f^*\calE; \underline s) = IR(\calE; \underline s)$ for $s \geq b(L/K) - \epsilon$ with $\epsilon>0$ as in Theorem~\ref{T:etaleness-nonlog}.  The theorem follows from  Proposition~\ref{P:AS-break=spec-norms}.
\end{proof}

Combining Theorem~\ref{T:AS-inv-gen-rot} and Proposition~\ref{P:gen-rot=>HA-thm}, we have the following.

\begin{theorem}\label{T:main-thm-nonlog}
Let $K$ be a complete discrete valuation field of mixed characteristic $(0, p)$ which is not absolutely unramified.  Let $\rho: G_K \rar GL(V_\rho)$ be a representation with finite monodromy.  Then,

(1) $\Art(\rho)$ is a non-negative integer;

(2) the subquotients $\Fil^a G_K / \Fil^{a+}G_K$ are trivial if $a \notin \QQ$ and are abelian groups killed by $p$ if $a \in \QQ_{>1}$.
\end{theorem}

\subsection{Application to finite flat group schemes}
\label{S:applications}

This subsection is an analogue of \cite[Section~4.1]{Me-condI} in the mixed characteristic case.

We first recall the definition \cite{AM-sous-groupes} of Abbes-Saito ramification filtration on finite flat group schemes.

\begin{convention}
All finite flat group schemes are commutative.
\end{convention}

\begin{definition}
Let $A$ be a finite flat $\calO_K$-algebra.  Write $A = \calO_K[\serie{x_}n] / \calI$ with $\calI$ an ideal generated by $\serie{f_}r$.  For $a \in \QQ_{\geq0}$, define the rigid space
\[
X^a = \big\{ (\serie{x_}n) \in A_K^n[0,1]  \big| 
|f_i(\serie{x_}n)| \leq \theta^a,\ i = \serie{}r \big\}.
\]
The \emph{highest break} $b(A / \calO_K)$ of $A$ is the smallest number such that for all $a > b(A / \calO_K)$, $\# \pi_0^\geom(X^a) = \rank_{\calO_K}A$.  This is the same as Definition~\ref{D:AS-space} if $A = \calO_L$; but in notation, we use the ring of integers instead of the fields themselves.
\end{definition}

\begin{definition}
Now we specialize to the case when $G = \Spec A$ is a finite flat group scheme.  We have a natural map of points $G(K^\alg) \inj X^a(K^\alg)$.  Further composing with the map for geometric connected components, we obtain
$$
\sigma^a: G(K^\alg) \inj X^a(K^\alg) \rar \pi_0^\geom(X^a).
$$
By functoriality of $\sigma^a$, one see that $\pi_0^\geom(X^a)$ has a natural group structure and $\sigma^a$ is a homomorphism (\cite[2.3]{AM-sous-groupes}). Define $G^a$ to be the Zariski closure of $\ker \sigma^a$.  Also, put $G^{a+} = \varinjlim_{b>a}G^b$.
\end{definition}

\begin{lemma}\label{L:AS-base-change}
\emph{\cite[Lemme~2.1.5]{AM-sous-groupes}} Let $K' / K$ be a (not necessarily finite) extension of complete discrete valuation fields of na\"ive ramification index $e$.  Let $A$ be a finite flat $\calO_K$-algebra which is a complete intersection relative to $\calO_K$.  Put $A' = A \otimes_{\calO_K} \calO_{K'}$; then $b(A'/ \calO_{K'}) = e\cdot b(A / \calO_K)$.
\end{lemma}

\begin{definition}
We say the finite flat group scheme $G$ is \emph{generically trivial} if $G \times_{\calO_k} K$ is disjoint union of copies of $\Spec K$, with some abelian group structure.
\end{definition}

\begin{theorem}
Let $G = \Spec A$ be a generically trivial finite flat group scheme over $\calO_K$.  Then $b(A / \OK)$ is a non-negative integer.
\end{theorem}
\begin{proof}
Let $\gcd(n_1, n_2) = 1$ and let $K_{n_1}$ and $K_{n_2}$ be two tamely ramified extensions of $K$ with ramification degree $n_1$ and $n_2$, respectively.  By Lemma~\ref{L:AS-base-change}, it suffices to prove the theorem for $G \times_{\calO_K} \calO_{K_{n_1}} / \calO_{K_{n_1}}$ and $G \times_{\calO_K} \calO_{K_{n_2}} / \calO_{K_{n_2}}$, respectively.  Thus, we may assume that $\beta_K \geq 2$.  The theorem follows from Theorem~\ref{T:main-thm-nonlog} and the same argument as in \cite[Proposition~5.1.7]{Me-condI}.
\end{proof}

\subsection{Integrality for Swan conductors}
\label{S:tame}

In this subsection, we will deduce the integrality of Swan conductors from that of Artin conductors (Theorem~\ref{T:main-thm-nonlog}).  We will use the fact that the logarithmic ramification breaks behave well under tame base changes.

We will keep Hypotheses~\ref{H:J-finite-set} and \ref{H:beta_K>1} until we state Theorem~\ref{T:log-HA-thm}.

\begin{notation}\label{N:tame-extension}
Let $n \in \NN$ such that $n \equiv 1 (\mod ep)$.  Define $K_n = K(\pi_K^{1/n})$ and $L_n = LK_n$.  Since $K_n$ and $L$ are linearly independent over $K$, $\Gal(L_n/K_n) = \Gal(L/K)$.  We take the uniformizer of $K_n$ and $L_n$ to be $\pi_{K_n} = \pi_K^{1/n}$ and $\pi_{L_n} = \pi_L / \pi_{K_n}^{(n-1)/e}$, respectively.
\end{notation}

\begin{notation}
Denote $\calR_{K_n} = \calO_{K_n} \llbracket \eta_0 / \pi_{K_n}, \eta_J \rrbracket$.  Applying Construction~\ref{C:psi-map} to $K_n$ gives an approximate homomorphism $\psi_{K_n}: \calO_{K_n} \rar \calO_{K_n} \llbracket \eta_0 / \pi_{K_n}, \eta_J \rrbracket$.
\end{notation}

\begin{lemma}\label{L:tame-diag}
There exists a unique continuous $\calO_K$-homomorphism $f_n^*: \calR_K \rar \calR_{K_n}$ sending $\delta_0$ to $(\pi_{K_n} + \eta_0)^n - \pi_K$ and $\delta_j$ to $\eta_j$ for $j \in J$.  This gives an approximately commutative diagram modulo $I_{K_n} = p(\eta_0/\pi_{K_n}, \eta_J) \cdot \calR_{K_n}$:
$$
\xymatrix{
\calO_K \ar@{_{(}->}[d] \ar[r]^-{\psi_K} & \calO_K \llbracket \delta_0 / \pi_K, \delta_J \rrbracket \ar[d]^{f_n^*}\\
\calO_{K_n} \ar[r]^-{\psi_{K_n}} & \calO_{K_n} \llbracket \eta_0 / \pi_{K_n}, \eta_J \rrbracket
}
$$
\end{lemma}
\begin{proof}
Follows from Proposition~\ref{P:psi-almost-hom}.  In fact, one can carefully choose $\psi_K$ and $\psi_{K_n}$ so that the above diagram \emph{commutes}.  But we do not need this here.
\end{proof}

\begin{proposition}\label{P:tame-BC-for-TS}
Fix $a\in \QQ_{>0}$.  Let $TS_{L/K, \log, \psi_K}^a$ be the standard logarithmic thickening space.  Then the space
$$
X = TS_{L/K, \log, \psi_K}^a \times_{(A_K^1[0, \theta^{a+1}] \times A_K^m[0, \theta^a]), f_n} \big( A_{K_n}^1[0, \theta^{a+1/n}] \times A_{K_n}^m[0, \theta^a] \big)
$$
is a logarithmic thickening space for $L_n/K_n$ with error gauge $\geq n\beta_K - (n-1)$; in particular, it is admissible.
\end{proposition}
\begin{proof}
First, we have
$$
\calS_K \otimes_{\calO_K} K_n \cong \calO_{K_n} \llbracket \eta_0/\pi_{K_n}, \eta_J \rrbracket [\frac 1p] \langle u_{J^+}\rangle \big/ \big( f_n^*(\psi_K(p_{J^+})) \big).
$$

Now we consider a construction of the logarithmic thickening space of $L_n / K_n$, using the same $c_J$ as the ones for $L/K$ and $\pi_{L_n}$ in Notation~\ref{N:tame-extension}.  Therefore, the ideal $\calI_{L_n / K_n}$ is generated by $p'_{J^+}$ and $p'_0 / \pi_{K_n}^{n-1}$, where the prime means to substitute $u_0$ with $\pi_{K_n}^{(n-1)/e} u'_0$.

Lemma~\ref{L:tame-diag} implies that
\begin{equation}\label{E:tame-base-change-3}
\psi_{K_n}(p'_0 / \pi_{K_n}^{n-1}) -f_n^*(\psi_K(p'_0)) / (\pi_{K_n}+ u'_0)^{n-1} \in \pi_{K_n}^{-n+1}(\pi_{K_n}^{n\beta_K-1}\eta_0, p\eta_J) \cdot\calS_{K_n},
\end{equation}
where $\calS_{K_n} = \calO_{K_n} \llbracket \eta_0 / \pi_{K_n}, \eta_J \rrbracket \langle u'_0, u_J\rangle$.  Hence,
\begin{align*}
\calS_K \otimes_{\calO_K} K_n & \cong \calO_{K_n} \llbracket \eta_0/\pi_{K_n}, \eta_J \rrbracket [\frac 1p] \langle u'_0, u_J\rangle \big/ \big( f_n^*(\psi_K(p'_0)), f_n^*(\psi_K(p'_J)) \big) \\
& = \calS_{K_n} [\frac 1p] \big/ \big( f_n^*(\psi_K(p'_0)) / (\pi_{K_n} + \eta_0)^{n-1}, f_n^*(\psi_K(p'_J)) \big)
\end{align*}
gives rise to logarithmic thickening spaces for $L_n / K_n$ with error gauge $\geq n\beta_K - (n-1)$; note that $K_n/K$ being tamely ramified of ramification degree $n$ gives a different normalization on error gauge.
\end{proof}

\begin{proposition}\label{P:tame-base-change-1}
There exists $N \in \NN$ and $\alpha_{L/K} \in [0, 1]$ such that, for all integers $n > N$ congruent to $1$ modulo $ep$, we have
$$
n \cdot b_\log(L/K) = b(L_n/K_n) - \alpha_{L/K}.
$$
\end{proposition}
\begin{proof}
By Construction~\ref{C:off-center-tame}, $f_n^*$ gives a finite \'etale morphism $f_n: A_{K_n}^1[0, \theta^{1/n}) \times A_{K_n}^m[0, 1) \rar A_K^1[0, \theta) \times A_K^m[0, 1)$ for $a>0$.  Let $\calE$ denote the differential module associated to $L/K$ coming from a standard logarithmic thickening space.  By Proposition~\ref{P:tame-BC-for-TS}, $f_n^*\calE$ is a differential module associated to $L_n/K_n$ given by the thickening space $X$ therein (for some admissible subset of error gauge $\leq \beta_K n - (n-1)$).  In particular,
\[
ET_{L_n/K_n} \supseteq ET_{L/K} \times_{A_K^1[0, \theta) \times A_K^m[0, 1), f_n} A_{K_n}^1[0, \theta^{1/n}) \times A_{K_n}^m[0, 1) =: f_n^*(ET_{L/K}),
\]
where $ET_{L_n/K_n}$ is the \'etale locus with respect to this chosen admissible subset.

The morphism $f_n$ is an off-centered tame base change, as discussed in Subsection~\ref{S:pDE}.  By Proposition~\ref{P:off-center-tame}, for $s_{J^+} \subset \RR$ such that $A_K^1[0, \theta^{s_0}] \times \cdots \times A_K^1[0, \theta^{s_m}] \subset ET_{L/K}$, we have $IR(f_n^*\calE; s_{J^+}) = IR(\calE; s_0 + \frac{n-1}n, s_J)$.  Thus, by Corollary~\ref{C:AS-break=spec-norms},
\begin{align}
\nonumber b(L_n/K_n) & = n \cdot \min \big\{ s \, \big|\; A_{K_n}^{m+1}[0, \theta^s] \subseteq ET_{L_n / K_n} \textrm{ and } IR(f_n^*\calE; \underline s) = 1 \big\} \\
\label{E:tame-base-change}
&= n \cdot \min \big\{ s \, \big|\; A_{K_n}^{m+1}[0, \theta^s] \subseteq f_n^*(ET_{L/K}) \textrm{ and } IR(f_n^*\calE; \underline s) = 1 \big\} \\
\nonumber
&= n\cdot \min \big\{ s\, \big|\; A_K^1[0, \theta^{s+(n-1)/n}] \times A_K^m[0, \theta^s] \subseteq ET_{L/K} \textrm{ and } IR(\calE; s + (n-1)/n, \underline s) = 1 \big\},
\end{align}
where the second equality holds because we will see in a moment that the minimal of $s$ can be achieved inside $ET_{L/K}$.  (Here, we have an extra $n$ in the equation because we are supposed to use $|\pi_{K_n}| = \theta^{1/n}$ as the ``base scale" in Corollary~\ref{C:AS-break=spec-norms}.)

Applying Proposition~\ref{P:prop-diff-eqns}(c) to  $\calE$, we know the locus $Z(\calE) =\{(s_{J^+})|IR(\calE; s_{J^+}) = 1\}$ is transrational polyhedral in a neighborhood of $[b_\log(L/K), +\infty)^{m+1}$, namely, where $\calE$ is defined.  Hence, in a neighborhood of $s_1 = b_\log(L/K)$, the intersection of the boundary of $Z$ with the surface defined by $s_1 = \cdots = s_m$ is of the form
$$
s_0 - \alpha' s_1 = b_\log(L/K) +1 - \alpha' b_\log(L/K),
$$
where $\alpha'$ is the slope;
$\alpha' \in [-\infty, 0]$ by the monotonicity Proposition~\ref{P:prop-diff-eqns}(c).  When $n \gg 0$, it is clear that the line $s \mapsto (s+ \frac{n-1}n, s, \dots, s)$ hits the boundary of $Z$ at $s = b_\log(L/K) + 1/(n(1-\alpha'))$.  This justifies the (second (typesetting?)) equality in \eqref{E:tame-base-change}.  It follows that
$$
b(L_n/K_n) = n \cdot b_\log(L/K) + 1/(1-\alpha');
$$
the different normalizations for ramification filtrations on $G_K$ and $G_{K_n}$ give the extra factor $n$.
\end{proof}

\begin{remark}
With more careful calculation, one may prove the above proposition and Proposition~\ref{P:tame-base-change-2} below for any $n$ sufficiently large and coprime to $p$.
\end{remark}

\begin{notation}\label{N:tilde-K_n}
Assume $p>2$.  Let $(b_J)$ be a $p$-basis of $K$; it naturally gives a $p$-basis of $K_n$.  Let $K_n(x_J)^\wedge$ denote the completion of $K_n(x_J)$ with respect to the $(1, \dots, 1)$-Gauss norm, and let $K_n'$ denote the completion of the maximal unramified extension of $K_n(x_J)^\wedge$.  Set
$$
\widetilde K_n = K'_n\big((b_J + x_J \pi_{K_n}^2)^{1/p} \big), \quad \widetilde L_n = \widetilde K_nL.
$$
Denote $\beta_j = (b_j + x_j \pi_{K_n}^2)^{1/p}$ for $j \in J$.  By Lemma~\ref{L:base-change}, we have a continuous $\calO_{K_n}$-homomorphism $\tilde f: \calO_{K_n} \llbracket \eta_0/\pi_{K_n}, \eta_J \rrbracket \rar \calO_{\widetilde K_n} \llbracket \xi_0/\pi_{K_n}, \xi_J, \xi'_J \rrbracket$ such that $\tilde f^*(\eta_0) = \xi_0$ and $\tilde f^*(\eta_j) = (\beta_j + \xi_j)^p - (x_j + \xi'_j)(\pi_{K_n} + \xi_0)^2 - b_j$ for $j \in J$.  For $a>1$, it gives rise to $\tilde f: A_{\widetilde K_n}^{2m+1}[0, \theta^a] \rar A_{K_n}^{m+1}[0, \theta^a] \inj A_{K_n}^1[0, \theta^a] \times A_{K_n}^m[0, \theta^{a-1/n}]$, where the last morphism is the natural inclusion of affinoid subdomain.
\end{notation}

\begin{proposition}\label{P:tame+generic-rotation}
Assume $p>2$, $\beta_K \geq \frac{2m+n}n$, and $a\in \QQ_{>1}$.  Let $X$ be as in Proposition~\ref{P:tame-BC-for-TS}.  Then the space
$$
X \times_{(A_{K_n}^1[0, \theta^{a+1/n}] \times A_{K_n}^m[0, \theta^a]), \tilde f} A_{\widetilde K_n}^{2m+1}[0, \theta^{a+1/n}]
$$
is a thickening space for $\widetilde L_n/ \widetilde K_n$ with error gauge $\geq n\beta_K - 2m - n + 1$; in particular, it is admissible.
\end{proposition}
\begin{proof}
It immediately follows from Proposition~\ref{P:tame-base-change-1} and applying Theorem~\ref{T:base-change} $m$ times.
\end{proof}

\begin{proposition}\label{P:tame-base-change-2}
Assume $p>2$.  There exists $N \in \NN$ such that, for all integers $n > N$ congruent to $1$ modulo $ep$, we have
\begin{equation}\label{E:tame-BC}
n \cdot b_\log(L/K) - 1 = b(\widetilde L_n/ \widetilde K_n) - 2\alpha_{L/K},
\end{equation}
where $\alpha_{L/K}$ is the \emph{same} as in Proposition~\ref{P:tame-base-change-1}.
\end{proposition}
\begin{proof}
We continue with the notation from Proposition~\ref{P:tame-base-change-1}.  Previous proposition implies that $\tilde f^* f_n^* \calE$ is a differential module associated to $\widetilde L_n / \widetilde K_n$ when $n>m$.  By applying Lemma~\ref{L:IR-under-gen-rot} $m$ times, we have $IR(\tilde f^* f_n^*\calE; \underline s) = IR(f_n^*\calE; s, \underline {s+\frac 1n})$.  By Proposition~\ref{P:off-center-tame}, it further equals $IR(\calE; s+\frac{n-1}n, \underline {s+\frac 1n})$.  By the same argument as in Theorem~\ref{P:tame-base-change-1}, we deduce our result with the same $\alpha_{L/K}$.
\end{proof}

\begin{remark}\label{R:p=2}
When $p=2$, we study $\widetilde K_n = K'_n\big((b_J + x_J \pi_{K_n}^3)^{1/p} \big)$ instead; the same argument above proves the proposition with \eqref{E:tame-BC} replaced by
$$
n \cdot b_\log(L/K) - 2 = b(L_n/K_n) - 3\alpha_{L/K}.
$$
\end{remark}

For the following theorem, we do not impose any supplementary hypothesis on $K$.

\begin{theorem}\label{T:log-HA-thm}
Let $K$ be a complete discrete valuation field of mixed characteristic $(0, p)$ and let $\rho: G_K \rar GL(V_\rho)$ be a representation with finite monodromy.  Then $\Swan(\rho)$ is a non-negative integer if $p \neq 2$ and is in $\frac 12 \ZZ$ if $p=2$.
\end{theorem}
\begin{proof}
First, as in the proof of Proposition~\ref{P:gen-rot=>HA-thm}, we may reduce to the case when $\rho$ is irreducible and factors through a finite Galois extension $L/K$, for which Hypothesis~\ref{H:J-finite-set} hold.  In this case, $\Swan(\rho) = b_\log(L/K) \cdot \dim \rho$.

By Proposition~\ref{P:AS-space-properties}(4), we have $\Swan(\rho|_{K_n}) = n\cdot \Swan(\rho)$ for any $K_n = K(\pi_K^{1/n})$ with $\gcd (n, ep) = 1$.  We need only to prove $\Swan(\rho|_{K_n}) \in \ZZ$ for two coprime $n$'s satisfying $\gcd (n, ep) = 1$, and the statement for $\Swan(\rho)$ will follow immediately.  In particular, we may assume that $\beta_K \geq 2$.

When $p>2$, we repeat the same argument again.  There exist $n_1, n_2$ satisfying the condition of Propositions~\ref{P:tame-base-change-1} and \ref{P:tame-base-change-2} and $\gcd(n_1, n_2) = 1$.  Thus, by the non-logarithmic Hasse-Arf Theorem~\ref{T:main-thm-nonlog}, 
\begin{eqnarray*}
n_1 \Swan(\rho) + \alpha_{L/K} \dim \rho \in \ZZ, & & n_1 \Swan(\rho) + 2\alpha_{L/K} \dim \rho \in \ZZ; \\
n_2 \Swan(\rho) + \alpha_{L/K} \dim \rho \in \ZZ, & & n_2 \Swan(\rho) + 2\alpha_{L/K} \dim \rho \in \ZZ.
\end{eqnarray*}
This implies immediately that $\alpha_{L/K} \dim \rho \in \ZZ$; hence, $\Swan(\rho) \in \ZZ$.

When $p=2$, a similar argument using Remark~\ref{R:p=2} gives $\Swan(\rho) \in \frac 12 \ZZ$.
\end{proof}

\begin{remark}
When $p=2$, we expect the integrality of Swan conductors in the case $K$ is the composition of a discrete completely valued field with perfect residue field and an absolutely unramified complete discrete valuation field.  In this case, we can factor $\psi_K$ as $\calO_K \rar \calO_K \llbracket \delta_0/ \pi_K \rrbracket \rar \calO_K \llbracket \delta_0 / \pi_K, \delta_J \rrbracket$ with the second map a \emph{homomorphism}.  This fact may allow us to show that $\alpha_{L/K}$ is either 0 or 1 depending on whether $\partial_0$ dominates.

We do not know if the integrality of $\Swan(\rho)$ might fail for $p=2$ in general.
\end{remark}

\subsection{An example of wildly ramified base change}
\label{S:example}

In this subsection, we explicitly calculate an example, which we will use in the next subsection.  This example was first introduced in \cite[Proposition~2.7.11]{KSK-Swan1}.  We retain Hypotheses~\ref{H:J-finite-set} and \ref{H:beta_K>1}.

\begin{lemma}\label{L:K*-over-K}
Let $K_*$ be the finite extension of $K$ generated by a root of
\begin{equation}\label{E:swan=1-ext}
T^p + \pi_K T^{p-1} = \pi_K.
\end{equation}
Then $K_*$ is Galois over $K$.  Moreover the logarithmic ramification break $b_\log(K_*/ K) = 1$.
\end{lemma}
\begin{proof}
Let $h(T) = T^p + \pi_K T^{p-1} - \pi_K$ and $\varpi$ a root of $h$.  It is clear that $\varpi$ is a uniformizer of $K_*$.
\begin{eqnarray*}
h(\varpi+T) &=& (\varpi+T)^p + \pi_K (\varpi +T)^{p-1} - \pi_K\\
& = & T^p + \sum_{i=1}^{p-1} \binom p i\varpi^i T^{p-i} + \pi_K \sum_{i=1}^{p-1}\binom{p-1}i \varpi^{p-1-i}T^i, \\
h(\varpi + \varpi^2 T) &=& \varpi^{2p} T^p + \pi_K \sum_{i=1}^{p-1}\binom{p-1}i \varpi^{p-1+i}T^i + \sum_{i=1}^{p-1} \binom p i\varpi^{2p-i} T^{p-i}\\
&=& \pi_K^2(1-\varpi^{p-1})^2 T^p + \pi_K^2 (1-\varpi^{p-1}) (p-1)T \\
&&+ \pi_K^2(1-\varpi^{p-1}) \sum_{i=2}^{p-1}\binom{p-1}i \varpi^{i-1}T^i + \sum_{i=1}^{p-1} \binom p i\varpi^{2p-i} T^{p-i}.
\end{eqnarray*}
Here, we organized so that the terms written in the summation are small terms.
We see that $h(\varpi + \varpi^2 T) / \pi_K^2$ is congruent to $T^p - T$ modulo $\varpi$.  By Hensel's lemma, it splits completely in $K_*$.  Hence, $K_* / K$ is Galois.  Moreover, the valuation of the difference between two distinct roots is 2.  This implies that $b_\log(K_* / K) = 1$.
\end{proof}

\begin{notation}
Denote the roots of $h(T) = T^p + \pi_K T^{p-1} - \pi_K$ by $\varpi = \serie{\varpi_}p$.

For $a>0$, the standard logarithmic thickening space $TS_{K_* / K, \log, \psi_K}^a$ for $K_* / K$ is given by
$$
\calO_{TS, K_* / K, \log, \psi_K}^{a+1} = K \langle \pi_K^{-a-1}\delta_0, \pi_K^{-a}\delta_J, z \rangle \big/ \big(z^p + (\pi_K + \delta_0)z^{p-1} - (\pi_K + \delta_0) \big).
$$
\end{notation}

\begin{lemma}\label{L:Dwork-sp-K_*-over-K}
Assume $a\in \QQ_{>1}$.  The standard logarithmic thickening space $TS_{K_*/K, \log, \psi_K}^a \times_K K_*$ is isomorphic to the product of $A_{K_*}^m[0, \theta^a]$ with the disjoint union of $p$ discs $|z-\varpi_\gamma| \leq \theta^{a-(p-2)/p}$ for $\gamma = \serie{}p$.
\end{lemma}
\begin{proof}
We can rewrite $z^p + (\pi_K + \delta_0)z^{p-1} - (\pi_K + \delta_0)=0$ as
\begin{equation}\label{E:K_*-K}
\prod_{\gamma=1}^p \big( z-\varpi_\gamma) = \delta_0(1-z^{p-1}).
\end{equation}
Since $|z| \leq 1$, the right hand side of \eqref{E:K_*-K} has norm $\leq \theta^{a+1} < \theta^2$.  On the left hand side, for $\gamma \neq \gamma' \in \{1, \dots, p\}$, $|\varpi_\gamma - \varpi_{\gamma'}| = \theta^{2/p}$.  This forces one of $|z - \varpi_{\gamma_0}|$ to be strictly smaller than the others, for some $\gamma_0 \in \{1, \dots, p\}$.  Thus, $|z - \varpi_{\gamma_0}| = |\delta_0| / (\theta^{2/p})^{p-1} = \theta^{a-(p-2)/p}$.
\end{proof}

\begin{notation}
For $\gamma = \serie{}p$, we define the $K_*$-homomorphism $f_\gamma^*: \calO_K \llbracket \delta_0 / \pi_K \rrbracket \rar \calO_{K_*} \llbracket \eta_0 / \varpi_\gamma \rrbracket$ by sending $\delta_0$ to
\begin{equation} \label{E:KSK-rotation}
\frac {(\varpi_\gamma + \eta_0)^p} {1 - (\varpi_\gamma + \eta_0)^{p-1}} - \pi_K = \sum_{n=0}^\infty \big( (\varpi_\gamma + \eta_0)^{p+n(p-1)} - \varpi_\gamma^{p+n(p-1)} \big).
\end{equation}
\end{notation}

\begin{lemma} \label{L:KSK-rotation}
For $a >1$, $f_\gamma^*$ induces a $K$-morphism $f_\gamma: A_{K_*}^1[0, \theta^{a-(p-2)/p}] \rar A_K^1[0, \theta^{a+1}]$, which is an isomorphism when we tensor the target with $K_*$ over $K$.  Moreover, if we use $F_{a+1}$ and $F^*_{a-(p-2)/p}$ to denote the completion of $K(\delta_0)$ and $K_*(\eta_0)$ with respect to the $\theta^{a+1}$-Gauss norm and $\theta^{a+(p-2)/p}$-Gauss norm, respectively, then $f_\gamma^*$ extends to a homomorphism $F_{a+1} \rar F^*_{a-(p-2)/p}$.
\end{lemma}
\begin{proof}
The statement follows from the fact that the leading term in \eqref{E:KSK-rotation} is $(2p-1)\varpi_\gamma^{2p-2} \eta_0$.
\end{proof}

\begin{proposition}\label{P:IR-KSK-trick}
Assume $a>1$.  Let $\calE$ be a differential module over $A_K^1[0, \theta^{a+1}]$.  For each $\gamma \in \{ \serie{}p\}$, this gives a differential module $f_\gamma^*\calE$ over $A_{K_*}^1[0, \theta^{a-(p-2)/p}]$.  Then we have
$$
IR_0(f_\gamma^*\calE; a-(p-2)/p) = IR_0(\calE; a+1).
$$
\end{proposition}
\begin{proof}
The proof is similar to Proposition~\ref{P:off-center-tame}.  By Lemma~\ref{L:KSK-rotation}, we have the following commutative diagram
$$
\xymatrix{
F_{a+1} \ar[d]^{f_\gamma^*} \ar[r]^-{f_\gen^*} & F_{a+1} \llbracket \pi_K^{-a-1} T_0 \rrbracket_0 \ar[d]^{f_\gamma^*} \\
F^*_{a-(p-2)/p} \ar[r]^-{f_\gen^*} & F^*_{a-(p-2)/p} \llbracket \varpi_\gamma^{-pa +p-2} T'_0 \rrbracket_0 
}
$$
where we extend $f_\gamma^*$ by $f_\gamma^* (T_0) = \dfrac {(\varpi_\gamma + \eta_0 + T'_0)^p} {1 - (\varpi_\gamma + \eta_0 + T'_0)^{p-1}} - \dfrac {(\varpi_\gamma + \eta_0)^p} {1 - (\varpi_\gamma + \eta_0)^{p-1}}$.

We claim that for $r \in [0, 1)$, $f_\gamma^*$ induces an isomorphism between
$$
F^*_{a-(p-2)/p} \times_{f_\gamma^*, F_{a+1}} \big( A_{F_{a+1}}^1[0,  r \theta^{a+1}) \big) \simeq A_{F^*_{a-(p-2)/p}}^1[0, r \theta^{a-(p-2)/p}).
$$
Indeed, if $|T'_0| < r \theta^{a-(p-2)/p}$, then
\begin{eqnarray*}
T_0 &=& \frac {(\varpi_\gamma + \eta_0 + T'_0)^p} {1 - (\varpi_\gamma + \eta_0 + T'_0)^{p-1}} - \frac {(\varpi_\gamma + \eta_0)^p} {1 - (\varpi_\gamma + \eta_0)^{p-1}} \\
&=& \big( (\varpi_\gamma + \eta_0 + T'_0)^p - (\varpi_\gamma + \eta_0)^p\big) + \big( (\varpi_i + \eta_0 + T'_0)^{2p-1} - (\varpi_\gamma + \eta_0)^{2p-1}\big) + \cdots \\
& \in & (2p-1) (\varpi_\gamma + \eta_0)^{2p-2}T'_0 + \big( (\varpi_\gamma + \eta_0)^{2p-1}T'_0, T'^p_0 \big) \cdot \calO_{K_*} \langle \varpi_\gamma^{-pa +p-2} \eta_0 \rangle \llbracket \varpi_\gamma^{-pa +p-2} T'_0 \rrbracket
\end{eqnarray*}
Hence, $|T_0| = \theta^{(2p-2)/p} \cdot |T'_0| < r \theta^a$.

Conversely, if $|T_0| < r\theta^a$, we rewrite the above equation as
\begin{equation}\label{E:sp-norm-KSK-trick-1}
T'_0 \in \frac 1 {(2p-1)(\varpi_\gamma + \eta_0)^{2p-2}} T_0 + (\varpi_\gamma T'_0) \cdot \calO_{K_*} \langle \varpi_\gamma^{-pa +p-2} \eta_0 \rangle \llbracket \varpi_\gamma^{-pa +p-2} T'_0 \rrbracket.
\end{equation}
We substitute \eqref{E:sp-norm-KSK-trick-1} back into itself recursively.  The equation converges to a $T'_0$, which is an inverse.

Therefore, Lemma~\ref{L:sp-norms-vs-gen-rad} implies that for $r \in [0, 1)$,
\begin{eqnarray*}
& & IR_0(\calE; a + 1) \leq r \\
&\LRar& f_\gen^*(\calE \otimes F_{a+1}) \textrm{ is trivial on }A_{F_{a+1}}^1[0, r\theta^{a+1}) \\
&\LRar& \tilde f_\gamma^* f_\gen^*(\calE \otimes F_{a+1}) = f_\gen^* \big(f_\gamma^*\calE \otimes F^*_{a-(p-2)/p} \big) \textrm{ is trivial on }A_{F^*_{a-(p-2)/p}}^1[0, r\theta^{a-(p-2)/p})\\
&\LRar& IR_0(f_\gamma^*\calE; a-(p-2)/p) \leq r.
\end{eqnarray*}
The proposition follows.
\end{proof}

\begin{construction}\label{C:psi-K*}
Fix a $p$-basis $(b_J)$ of $K$; it naturally gives a $p$-basis of $K_*$.  Fix a choice of $\psi_K: \calO_K \rar \calO_K \llbracket \delta_0 / \pi_K, \delta_J \rrbracket$ as in Construction~\ref{C:psi-map}.  We will use the method in Construction~\ref{C:psi-map} to define $\psi_{K_*, \gamma}$ for $\gamma = \serie{}p$ such that the following diagram \emph{commutes}.
\begin{equation} \label{E:KK_*-base-change}
\xymatrix{
\calO_K \ar@{_{(}->}[d] \ar[r]^-{\psi_K} & \calO_K \llbracket \delta_0 / \pi_K, \delta_J \rrbracket \ar[d]^{f_\gamma^*}\\
\calO_{K_*} \ar[r]^-{\psi_{K_*}} & \calO_{K_*} \llbracket \eta_0 / \varpi_\gamma, \delta_J \rrbracket
}
\end{equation}

For any element $h \in \calO_{K_*}$, first write $h = \sum_{i = 0}^{p-1} h_i \varpi_\gamma^i$ where $h_i \in \calO_K$.  As in Construction~\ref{C:psi-map}, write each of $h_i$ as $h_i^\circ \pi_K^{e_i}$ for $e_i = v_K(h_i)$ and $h_i^\circ \in \OK$; chose a compatible system of $r$-th $p$-basis decomposition of $h_i^\circ$ as
$$
h_i^\circ = \sum_{e_J = 0}^{p^r-1} b_J^{e_J} \Big( \sum_{n=0}^\infty \big( \sum_{n'=0}^{\lambda_{i, (r), e_J, n}} \alpha_{i, (r), e_J, n, n'}^{p^r}\big) \pi_K^n \Big)
$$
for some $\alpha_{i, (r), e_J, n, n'} \in \calO_K^\times \cup \{0\}$ and some $\lambda_{i, (r), e_J, n} \in \ZZ_{\geq 0}$.  We choose the system of $r$-th $p$-basis decomposition of $h / \varpi_\gamma^{v_{K_*}(h)}$ to be
$$
\frac h{\varpi_\gamma^{v_{K_*}(h)}} = \frac 1{\varpi_\gamma^{v_{K_*}(h)}} \sum_{i = 0}^{p-1} \varpi_\gamma^i
\sum_{e_J = 0}^{p^r-1} b_J^{e_J} \Big( \sum_{n=0}^\infty \big( \sum_{n'=0}^{\lambda_{i, (r), e_J, n}} \alpha_{i, (r), e_J, n, n'}^{p^r}\big) (\varpi_\gamma^{p-1} + \varpi_\gamma^{2p-1} + \cdots)^{n+e_i} \Big)
$$
and define $\psi_{K_*, \gamma}(h)$ to be the limit
$$
\lim_{r \rar +\infty} 
\sum_{i = 0}^{p-1} (\varpi_\gamma + \eta_0)^i
\sum_{e_J = 0}^{p^r-1} (b_J + \delta_J)^{e_J} \Big( \sum_{n=0}^\infty \big( \sum_{n'=0}^{\lambda_{i, (r), e_J,n}} \alpha_{i, (r), e_J, n, n'}^{p^r}\big) \big( (\varpi_\gamma + \eta_0)^{p-1} + (\varpi_\gamma+ \eta_0)^{2p-1} + \cdots \big)^{n+e_i} \Big).
$$

This gives a $\psi_{K_*, \gamma}$ defined in the way of Construction~\ref{C:psi-map}; the diagram \eqref{E:KK_*-base-change} is commutative.
\end{construction}

\begin{hypo} \label{H:L-Galois}
For the rest of this subsection, let $L/K_*$ be a finite Galois extension satisfying Hypotheses~\ref{H:J-finite-set} and \ref{H:beta_K>1} and such that $L/K$ is Galois.
\end{hypo}

\begin{proposition}\label{P:BC-K*-over-K}
Let $a\in \QQ_{>1}$. Then there exists admissible $(R_{J^+}) \subset (\delta_{J^+}) \cdot \calS_K$ such that the logarithmic thickening space for $L/K$, after extension of scalars from $K$ to $K_*$, is isomorphic to a disjoint union of $p$ (different) logarithmic thickening spaces for $L / K_*$:
$$
TS_{L/K, \log, R_{J^+}}^a \times_K K_* \isom \coprod_{\gamma = 1}^p TS_{L/K_*, \log,  \psi_{K_*, \gamma}}^{pa -p+1}.
$$
\end{proposition}
\begin{proof}
Write $\calO_{K_*} \langle u_{J^+} \rangle / (p_{J^+}) = \calO_L$ using Construction~\ref{C:generators-of-calI}.  Since $\calO_K \langle z \rangle / (z^p + \pi_K z^{p-1} - \pi_K) = \calO_{K_*}$, we may replace the coefficients in $p_{J^+}$ by elements in $\calO_K \langle z \rangle$ with degree $\leq p-1$ in $z$, denoting the result polynomials by $p'_{J^+}$.  Thus by Lemma~\ref{L:Dwork-sp-K_*-over-K} and the commutativity of \eqref{E:KK_*-base-change},
\begin{align*}
& \prod_{\gamma = 1}^p K_* \langle \varpi_\gamma^{-pa+p-2} \eta_0, \varpi_\gamma^{-pa+p-1} \eta_J \rangle \langle u_{J^+} \rangle \big/ (\psi_{K_*, \gamma}(p_{J^+})) \\
\cong \ & K_* \langle \pi_K^{-a-1} \delta_0, \pi_K^{-a} \delta_J \rangle \langle u_{J^+}, z \rangle \big/ \big( \psi_K(p'_{J^+}), z^p + (\pi_K + \delta_0) z^{p-1} - (\pi_K + \delta_0) \big),
\end{align*}
where the latter one is
a recursive logarithmic thickening space for $L/K$, base changed to $K_*$.
By Proposition~\ref{P:recursive-ts=ts}, this recursive logarithmic thickening space is isomorphic to a logarithmic thickening space $TS_{L/K, \log, R_{J^+}}^a$ for $L/K$ for some admissible subset $R_{J^+} \subset (\delta_{J^+}) \cdot \calS_K$.
\end{proof}

\begin{corollary}\label{C:ELK=+ELKgamma}
Let $\calE_{L/K}$ be the differential module over $A_K^1[0, \theta^{a+1}] \times A_K^m[0, \theta^a]$ coming from $TS_{L/K, \log, R_{J^+}}^a$.  For $\gamma \in \{\serie{}p\}$, let $\calE_{L/K_*, \gamma}$ be the differential module over $A_{K_*}^1[0, \theta^{a-(p-2)/p}] \times A_{K_*}^m[0, \theta^{a-(p-1)/p}]$ coming from $TS_{L/K_*, \log, \psi_{K_*, \gamma}}^{ap-p+1}$.  Then $\calE_{L/K} \otimes_K K_* \simeq \bigoplus_{\gamma = 1}^p f_{\gamma*}\calE_{L/K_*, \gamma}$.
\end{corollary}
\begin{proof}
It follows from Lemma~\ref{L:Dwork-sp-K_*-over-K} and Proposition~\ref{P:BC-K*-over-K}.
\end{proof}

\subsection{Subquotients of logarithmic ramification filtration}
\label{S:log-HA-thm}

In this subsection, we prove Theorem~\ref{T:log-subquot-p-ele} that the subquotients $\Fil^a_\log G_K / \Fil^{a+}_\log G_K$ of logarithmic ramification filtration are abelian groups killed by $p$ if $a \in \QQ_{>0}$ and are trivial if $a \notin \QQ$.  This uses the tricky base change discussed in previous subsection.

We assume Hypothesis~\ref{H:L-Galois} until we state the main Theorem~\ref{T:log-subquot-p-ele}.

\begin{notation}
\label{N:tilde-K-gamma}
Fix $\gamma \in \{\serie{}p\}$.  Let $(b_J)$ be a finite $p$-basis of $K$.  It naturally gives a $p$-basis of $K_*$.  Denote by $K(x_J)^\wedge$ the completion of $K(x_J)$ with respect to the $(1, \dots, 1)$-Gauss norm and by $K'$ the completion of the maximal unramified extension of $K(x_J)^\wedge$.  Write $K'_* = K_*K'$ and $L' = K'_*L$.  Set
$$
\widetilde K_\gamma = K'_*((b_J + x_J \varpi_\gamma^{p-1})^{1/p}).
$$
Denote $\beta_J = (b_J + x_J \varpi_\gamma^{p-1})^{1/p}$ for simplicity.  Denote the residue fields of $\widetilde K_\gamma$ and $\widetilde L_\gamma = L \widetilde K_\gamma$ by $\tilde k$ and $\tilde l$, respectively.  Take the uniformizer and $p$-basis of $\widetilde K_\gamma$ to be $\varpi_\gamma$ and $\{\beta_J, x_J\}$, respectively.
\end{notation}

\begin{situation}\label{Sit:field-extension}
We have the following diagram of field extensions:
$$
\xymatrix{
L \ar@{-}[d] \ar@{-}[r] & L' \ar@{-}[r] \ar@{-}[d] & \widetilde L_\gamma \ar@{-}[d]\\
K_* \ar@{-}[r] \ar@{-}[d] & \ar@{-}[r] K'_* \ar@{-}[d] & \widetilde K_\gamma \\
K \ar@{-}[r] & K'
}
$$
Note that $(\widetilde K_\gamma)_{\gamma = \serie{}p}$ are extensions of $K'_*$ conjugate over $K'$.  The ramification filtrations on $G_{\widetilde K_\gamma}$ are stable under the conjugate action of $\Gal(K'_* / K')$.  Precisely, for any $b \geq 0$ and $g \in \Gal(K'_* / K')$, $g \Fil^b_\log G_{\widetilde K_\gamma} g^{-1} = \Fil^b_\log G_{g(\widetilde K_\gamma)}$ and $g \Fil^b G_{\widetilde K_\gamma} g^{-1} = \Fil^b G_{g(\widetilde K_\gamma)}$ inside $G_{K'}$.  In particular, since $L' / K'$ and hence 
$\widetilde L_\gamma / \widetilde K_\gamma$ is Galois, $b(\widetilde L_\gamma / \widetilde K_\gamma)$ and $b_\log(\widetilde L_\gamma / \widetilde K_\gamma)$ do not depend on $\gamma = \serie{}p$.
\end{situation}

For the following theorem, we do not impose any supplementary hypothesis on the field $K$.

\begin{theorem}\label{T:log-subquot-p-ele}
Let $K$ be a complete discrete valuation field of mixed characteristic $(0, p)$.  Let $G_K$ be its Galois group.  Then the subquotients $\Fil^a_\log G_K / \Fil^{a+}_\log G_K$ of the logarithmic ramification filtration are trivial if $a \notin \QQ$ and are abelian groups killed by $p$ if $a \in \QQ_{>0}$.
\end{theorem}
\begin{proof}
We will proceed as in the proof of Theorem~\ref{T:main-thm-nonlog}.  Fix $a>0$.  Let $L$ be a finite Galois extension of $K$ with Galois group $G_{L/K}$ with an induced ramification filtration.  We may assume that $\Fil_\log^{a+}G_{L/K}$ is the trivial group but $\Fil_\log^a G_{L/K}$ is not.  We may also assume Hypothesis~\ref{H:J-finite-set}.  Furthermore, by Proposition~\ref{P:AS-space-properties}(4), we are free to make a tame base change and assume that $a = b_\log(L/K) > 1$ and $p\beta_K \geq m(p-1) + 1$.  Finally, we may replace $L$ by $LK_*$ since $b_\log(K_* / K) = 1$ by Lemma~\ref{L:K*-over-K}.  We need to show that $\Fil_\log^a G_{L/K}$ is an abelian group killed by $p$ if $a \in \QQ_{>1}$ and is trivial if $a \notin \QQ$.

We claim that each of the logarithmic ramification breaks $b>1$ of $L/K$ will become a non-log ramification break $bp - p + 2$ on $\widetilde L_1 / \widetilde K_1$.  In other words, $\Fil_\log^b G_{L/K} \subseteq \Fil^{pb - p+2} G_{\widetilde L_\gamma / \widetilde K_\gamma}$ for any $\gamma \in \{1, \dots, p\}$ and $b >1$.  (It does not matter which $\gamma$ we choose as they give the same answer by Situation~\ref{Sit:field-extension}.)  Then the theorem is a direct consequence of the non-logarithmic Hasse-Arf theorem~\ref{T:main-thm-nonlog}(2).

To prove the claim, it suffices to prove the highest ramification breaks as the others will follow from the calculation of the other $L$'s.

For each $\gamma \in \{\serie{}p\}$, there exists a unique continuous $\calO_{K_*} \llbracket \eta_0/\varpi_\gamma \rrbracket$-homomorphism $\tilde f_\gamma^*: \calO_{K_*}\llbracket \eta_0 / \varpi_\gamma, \delta_J \rrbracket \rar \calO_{\widetilde K_\gamma} \llbracket \eta_0 /\varpi_\gamma, \eta_J, \eta'_J \rrbracket$ such that $\tilde f_\gamma^*{\delta_j} = (\beta_j + \eta_j)^p - (x_j + \eta'_j)(\varpi_\gamma + \eta_0)^{p-1} - b_j$ for $j \in J$.  For $a>1$, $\tilde f_\gamma^*$ gives a morphism $\tilde f_\gamma: A_{\widetilde K_\gamma}^{2m+1}[0, \theta^a] \rar A_{K_*}^{m+1}[0, \theta^a]$.

Let $TS_{L/K_*, \psi_{K_*, \gamma}}^a$ be the standard thickening space for $L/K_*$ and $\psi_{K_*, \gamma}$.  We have a Cartesian diagram
$$
\xymatrix{
& \ar[ld] TS_{L/K_*, \psi_{K_*, \gamma}}^a \ar[d]^\Pi & \ar[l]_-{\tilde f_\gamma} TS_{L/K_*, \psi_{K_*, \gamma}}^a \times_{A_{K_*}^{m+1}[0, \theta^a], \tilde f_\gamma} A_{\widetilde K_\gamma}^{2m+1}[0, \theta^a] \ar[d]^\Pi \\
A_{K_*}^1[0, \theta^{a+\frac{2p-2}p}] \times A_{K_*}^m[0, \theta^a] & \ar[l]_-{f_\gamma} A_{K_*}^{m+1}[0, \theta^a] & \ar[l]_{\tilde f_\gamma} A_{\widetilde K_\gamma}^{2m+1}[0, \theta^a]
}
$$
By applying Theorem~\ref{T:base-change} $m$ times, $TS_{L/K_*, \psi_{K_*, \gamma}}^a \times_{A_{K_*}^{m+1}[0, \theta^a], \tilde f_\gamma} A_{\widetilde K_\gamma}^{2m+1}[0, \theta^a]$ is an admissible recursive non-logarithmic thickening space (of error gauge $\geq p\beta_K - m(p-1) \geq 1$), which is isomorphic to an admissible non-logarithmic thickening space for $\widetilde L_\gamma / \widetilde K_\gamma$ by Proposition~\ref{P:recursive-ts=ts}.  Thus $\tilde f_\gamma^* \calE_{L/K_*, \gamma}$ is a differential module associated to $\widetilde L_\gamma / \widetilde K_\gamma$.

By Proposition~\ref{P:IR-KSK-trick} and Lemma~\ref{L:IR-under-gen-rot}, we have
$$
IR(\tilde f_\gamma^* \calE_{L/K_*, \gamma}; \underline s) = IR \Big(\calE_{L/K_*, \gamma}; s, \underline {s+\frac{p-2}p} \Big) = IR \Big( (f_\gamma)_*\calE_{L/K_*, \gamma}; s+\frac{2p-2}p, \underline {s+\frac{p-2}p} \Big).
$$
The claim follows by Corollaries~\ref{C:ELK=+ELKgamma} and \ref{C:AS-break=spec-norms}.
\end{proof}

\bibliographystyle{plain}

\end{document}